\documentclass[pdflatex,sn-mathphys-num]{sn-jnl}

\usepackage{graphicx}%
\usepackage{multirow}%
\usepackage{amsmath,amssymb,amsfonts}%
\usepackage{mathtools}%
\usepackage{amsthm}%
\usepackage{mathrsfs}%
\usepackage[title]{appendix}%
\usepackage{xcolor}%
\usepackage{textcomp}%
\usepackage{manyfoot}%
\usepackage{booktabs}%
\usepackage{algorithm}%
\usepackage{algorithmicx}%
\usepackage{subcaption}
\usepackage{booktabs}
\usepackage{pdflscape}
\usepackage{makecell}
\usepackage{rotating}
\usepackage[numbers]{natbib} 
\usepackage{algpseudocode}%
\usepackage{listings}%
\usepackage{titlesec}
\usepackage{enumerate}
\usepackage{hyperref} 
\usepackage{cleveref} 

\theoremstyle{thmstyleone}
\newtheorem{theorem}{Theorem}[section]
\newtheorem{proposition}{Proposition}[section]%

\newtheorem{assumption}{Assumption}[section]
\DeclareMathOperator*{\argmin}{arg\,min}
\theoremstyle{thmstyletwo}%
\newtheorem{remark}{Remark}[section]%

\theoremstyle{thmstyleone}
\newtheorem{lemma}{Lemma}[section] 

\begin{document}

\title[Article Title]{The Golden Ratio Primal-Dual Algorithm with Two New Stepsize Rules for Convex-Concave Saddle Point Problems}


\author*[1,2]{\fnm{Santanu} \sur{Soe}}\email{ma22d002@smail.iitm.ac.in, santanu.soe@student.unimelb.edu.au}

\author[2]{\fnm{Matthew K. }\sur{Tam}}\email{matthew.tam@unimelb.edu.au}
\equalcont{These authors contributed equally to this work.}

\author[1]{\fnm{V.} \sur{Vetrivel}}\email{vetri@iitm.ac.in}
\equalcont{These authors contributed equally to this work.}

\affil*[1]{\orgdiv{Department of Mathematics}, \orgname{Indian Institute of Technology Madras}, \orgaddress{\street{Sardar Patel road}, \city{Chennai}, \postcode{600036}, \state{Tamil Nadu}, \country{India}}}

\affil[2]{\orgdiv{School of Mathematics and Statistics}, \orgname{The University of Melbourne}, \orgaddress{\street{Parkville}, \city{Melbourne}, \postcode{3010}, \state{Victoria}, \country{Australia}}}

\abstract{In this paper, we present two stepsize strategies for the extended Golden Ratio primal-dual algorithm (E-GRPDA) designed to address structured convex optimization problems in finite-dimensional real Hilbert spaces. The first rule features a non-increasing primal stepsize that remains bounded below by a positive constant and is updated adaptively at each iteration, eliminating the need to compute the Lipschitz constant of the gradient of the function and the norm of the operator without involving backtracking. The second stepsize rule is adaptive, adjusting based on the local smoothness of the smooth component function and the norm of the operator involved. In other words, we present an adaptive version of the E-GRPDA algorithm. We prove that E-GRPDA achieves an ergodic sublinear convergence rate result with both stepsize rules, based on the function-value residual and constraint violation rather than on the so-called primal–dual gap function. Additionally, we establish an R-linear convergence rate for E-GRPDA with the first stepsize rule, under some standard assumptions and with appropriately chosen parameters. Through numerical experiments on various convex optimization problems, we demonstrate the effectiveness of our approaches and compare their performance to the existing ones.}





\maketitle

\section{Introduction}\label{sec1}

Let $\mathbb{X}$ and $\mathbb{Y}$ be finite-dimensional real Hilbert spaces. In this paper, we consider the following structured optimization problem
\begin{equation}\label{1.1} \min_{x \in \mathbb{X}} \hspace{.1cm} f(x) + g(Kx) + h(x), 
\end{equation}
where $f: \mathbb{X} \to (-\infty, \infty]$ and $g: \mathbb{Y} \to (-\infty, \infty]$ are convex, proper, and lower semicontinuous (lsc) functions, $K: \mathbb{X} \to \mathbb{Y}$ is a linear operator, and $h: \mathbb{X} \to \mathbb{R}$ is a convex, differentiable function with $\bar{L}$-Lipschitz continuous gradient, i.e., 
\begin{equation*} 
\lVert \nabla h(x) - \nabla h(y) \rVert \leq \bar{L} \lVert x - y \rVert, \quad \forall x, y \in \mathbb{X}. 
\end{equation*} 
The \emph{Fenchel dual problem} of \eqref{1.1} is given by \begin{equation}\label{1.3} 
\max_{u\in\mathbb{Y}} -(f+h)^*(-K^*u)-g^*(u),
\end{equation}
where $g^*(y)=\sup_{z\in \mathbb{Y}}\{\langle y,z\rangle -g(z)\}$ denotes the \emph{Legendre--Fenchel conjugate} of $g$ and $K^*$ is the \emph{adjoint} of the linear operator $K$.

A notable example of a problem in the form of \eqref{1.1} is the \emph{Fused LASSO problem} \cite{tibshirani2005sparsity, yuan2006model}, which is commonly used in areas such as genomics, signal processing, and image segmentation. 
This problem takes the form 
\begin{equation}\label{fussed_LASSO} 
\min_{x\in\mathbb{R}^n}F(x) := \lambda_1 \lVert x \rVert_1 + \lambda_2\|Dx\|_1 + \frac{1}{2} \lVert Ax - b \rVert^2,
\end{equation} 
which corresponds to \eqref{1.1} with $f(x) = \lambda_1\|x\|_1$, $g(x) = \lambda_2\|x\|_1$, $K=D$, and $h(x) = \frac{1}{2}\|Ax - b\|^2$, where $A \in \mathbb{R}^{m \times n}$, $b \in \mathbb{R}^{m}$, and $D \in \mathbb{R}^{(n-1) \times n}$ is the \emph{difference operator} \cite{yan2018new}. Here, $\lambda_1,\lambda_2>0$ are regularization parameters. Furthermore, when $\lambda_2 = 0$, (\ref{fussed_LASSO}) reduces to the well-known \emph{LASSO problem} \cite{tibshirani1996regression}, which is a particular case of (\ref{1.1}) when $h = 0$. Another relevant example of \eqref{1.1} is the \emph{Elastic Net regularization} \cite{zou2005regularization}, which combines both $\ell_1$ and $\ell_2$ penalties. 
Given $f(x) = \lambda_1\|x\|_1$, $g(\cdot) = \frac{1}{2}\|\cdot- b\|^2$, and $h(x) = \lambda_2\|x\|^2$, where $K \in \mathbb{R}^{m \times n}$, $b \in \mathbb{R}^m$, and $\lambda_1, \lambda_2 > 0$ are regularization parameters, the Elastic Net problem can be written as
\begin{equation*} 
\min_{x\in\mathbb{R}^n} F(x) := \lambda_1 \lVert x \rVert_1 + \frac{1}{2}\|Kx-b\|^2 + \lambda_2\|x\|^2.
\end{equation*} 
In addition to these, optimization problems of the form (\ref{1.1}) are prevalent in various other fields, such as image denoising, machine learning, game theory, operations research, and many more, see \cite{goldstein2015adaptive, r6, zhu2008efficient, nedic2009subgradient, bertsekas2009projection, yang2011alternating,chambolle2011first} and references therein.

The convex minimization problem \eqref{1.1} can be solved using the Forward-Backward type splitting methods \cite{lions1979splitting, abbas2015dynamical, nesterov2013gradient}, provided that the \emph{proximal operator} of $f + g \circ K $ can be computed efficiently. However, in practice, this can be challenging even for a simple nonsmooth function $ g $. For instance, consider the case where $ g(x) = \|x\|_1$, and $ f = 0$. It is well known that the proximal map of $ g $ is given by the \emph{soft thresholding operator}, while there is no closed-form solution for the proximal operator of $g\circ K$ when $ K $ is not an orthogonal or diagonal matrix, see \cite[Section 2.2]{10.1561/2400000003} for more discussions. To overcome this difficulty, we can decouple $g$ and $K$ by introducing $g^*$ into \eqref{1.1}. This gives the \emph{saddle point} formulation of the problem as
\begin{equation}\label{1.2}
    \min_{x\in \mathbb{X}}\max_{y\in\mathbb{Y}}~\mathbb{L}(x,y)=f(x)+ h(x)+ {\langle Kx, \,y\rangle}- g^*(y).
\end{equation}
Under strong duality, the saddle point formulation \eqref{1.2} represents a framework for simultaneously solving both the primal problem \eqref{1.1} and the dual problem \eqref{1.3}.
We begin by reviewing the literature when the smooth term $h=0$ and then proceed to discuss the scenario when $h$ is nonzero.\\
\\
\textbf{Case 1 ($h=0$):}
 To solve \eqref{1.1} and \eqref{1.3}, a classical primal-dual approach is the \emph{alternating direction methods of multipliers} (ADMM); see, for example, \cite{gabay1976dual,glowinski1975approximation,lions1979splitting}. Although ADMM can be applied directly to solve $\min_{x} f(x) + (g\circ K)(x)$, it requires knowledge of the proximal operator of $ g\circ K$, which can be challenging, as highlighted in the example mentioned earlier. To utilize the proximal operators of $f$ and $g$ separately, we can introduce a new variable $z$ and reformulate $\min_{x} f(x) + g(Kx)$ as
 \begin{equation}\label{ADMM}
     \min_{x,z} f(x) + g(z) ~\text{subject to} ~ Kx-z=0.
 \end{equation}
Now, applying ADMM to (\ref{ADMM}) results in the following iterative updates
 \begin{equation}\label{ADMM_1}
     \begin{cases}
         x_{n+1} = \argmin_{x}f(x) + \frac{\delta'}{2}\|Kx-z_n+u_n\|^2,\\
         z_{n+1} = \argmin_{z} g(z) + \frac{\delta'}{2}\|Kx_{n+1}-z+u_n\|^2,\\
         u_{n+1} = u_n + Kx_{n+1} - z_{n+1},
     \end{cases}
 \end{equation}
where $\delta'>0$. Although the $z$-update of~\eqref{ADMM_1} can be computed easily if the proximal operator of $g$ is known, solving the $x$-update requires an iterative approach and the assumption that $K$ has a full column rank. Under the assumption $\operatorname{rank}K=\operatorname{dim}\mathbb{X}$, the $x$-update can be written as $x_{n+1} = \operatorname{prox}_{\frac{1}{\delta'}f}^{K^*K}((K^*K)^{-1}K^*(z_n - u_n))$, where $ \operatorname{prox}_{\frac{1}{\delta'}f}^{K^*K} $ is the generalized proximal operator \cite{zhang2011unified,pock2011diagonal}, which restricts the possible choices for $K$. Moreover, when $ f(x) = \frac{1}{2} \|x- b\|^2$ for some $b\in\mathbb{X}$ and $g = 0$, solving the $x$-update in~\eqref{ADMM_1} is equivalent to solving the linear system 
\begin{equation*}
    (I +\delta'K^* K) x_{n+1} = b + \delta' K^*(z_n-u_n).
\end{equation*}
For large, dense matrices, efficiently solving such systems becomes challenging due to the substantial memory required to store the matrix during computation. These limitations can be circumvented by employing a simple primal-dual approach known as the \emph{Arrow–Hurwicz method} \cite{uzawa1958iterative}. This method leverages the proximal operators of $f$ and $g^*$ individually, and removes the need to solve complicated subproblems or systems involving the operator $K$. Given $(x_0,y_0)\in \mathbb{X}\times \mathbb{Y}$, the iteration scheme for $n\geq 1$ is
\begin{equation}\label{1.2.1}
  \begin{cases}
     x_n= \operatorname{prox}_{\tau f}(x_{n-1}- \tau K^*y_{n-1}),\\
     y_n= \operatorname{prox}_{\sigma g^*}(y_{n-1}+\sigma Kx_{n}),
  \end{cases}
\end{equation}
where $\tau, \sigma > 0$ are the primal and dual stepsize parameters, respectively. The convergence of the Arrow–Hurwicz method \cite{uzawa1958iterative} heavily depends on the choice of objective functions, their domains, and the appropriate selection of stepsizes $\tau$ and $\sigma$. For example, when the domain of $g^*$ is bounded, the convergence of the method has been established with a convergence rate of $\mathcal{O}(1/\sqrt{N})$, where $N$ is the number of iterations, measured by the primal-dual gap function; see \cite{chambolle2011first,nedic2009subgradient}. However, the sequence generated by this method does not converge in general; see \cite{he2012convergence, he2022convergence} for a series of counterexamples in both one and higher dimensions. To achieve convergence of (\ref{1.2.1}) in more general cases, Chambolle and Pock \cite{chambolle2011first} introduced an extrapolation step in (\ref{1.2.1}) following the computation of the primal variable $x_n$. For $(x_0,y_0)\in \mathbb{X}\times \mathbb{Y}$, $\delta\in[0,1]$, and $n\geq1$, the \emph{Chambolle--Pock primal-dual algorithm} iterates as
\begin{equation}\label{1.2.2}
    \begin{cases}
         x_n= \operatorname{prox}_{\tau f}(x_{n-1}- \tau K^*y_{n-1}),\\
         \Tilde{x}_n = x_n +\delta(x_n-x_{n-1}),\\
      y_n= \operatorname{prox}_{\sigma g^*}(y_{n-1}+ \sigma K\Tilde{x}_{n}).
    \end{cases}
\end{equation}
For $\delta = 0$, the primal-dual algorithm \eqref{1.2.2} reduces to the Arrow-Hurwicz algorithm \eqref{1.2.1}. Chambolle and Pock proved the convergence of \eqref{1.2.2} for $\delta = 1$, commonly known as the \emph{Primal-Dual Hybrid Gradient (PDHG)} method. They demonstrated that this method achieves a convergence rate of $\mathcal{O}(1/N)$, measured by the primal-dual gap function, provided that the stepsize condition $ \tau \sigma \lVert K \rVert^2 < 1 $ is satisfied, see \cite[Theorem 1]{chambolle2011first}. Additionally, convergence is also guaranteed when $\tau\sigma\|K\|^2=1$, see \cite[Theorem 3.3]{condat2013primal}, which simplifies stepsize selection by reducing it to the tuning of a single parameter. More recently, Chang and Yang introduced a modified version of the Arrow-Hurwicz algorithm called the \emph{Golden Ratio primal-dual algorithm} (GRPDA) \cite[Algorithm 3.1]{chang2021goldengrpda}. This algorithm allows for larger primal-dual stepsizes compared to PDHG, with its benefits demonstrated through experimental results \cite{chang2021goldengrpda}. For $(x_0, y_0) \in \mathbb{X} \times \mathbb{Y}$, with $z_0 = x_0$ and \( n \geq 1 \), the GRPDA iteration scheme is given by
\begin{equation}\label{GRPDA}
    \begin{cases}
        z_n = \dfrac{\psi-1}{\psi}x_{n-1} + \dfrac{1}{\psi}z_{n-1},\\
         x_n= \operatorname{prox}_{\tau f}(z_{n}- \tau K^*y_{n-1}),\\
      y_n= \operatorname{prox}_{\sigma g^*}(y_{n-1}+ \sigma Kx_{n}).
    \end{cases}
\end{equation} 
The main difference between the Chambolle-Pock primal-dual algorithm \eqref{1.2.2} and GRPDA \eqref{GRPDA} is that, instead of the extrapolation term, a convex combination based on the Golden Ratio is added in the \(x\)-subproblem of \eqref{1.2.1}. This convex combination allows GRPDA to converge when the stepsize condition $\tau \sigma \lVert K \rVert^2 < \phi$ is satisfied, where $\phi$ is the Golden Ratio. The authors also demonstrated that GRPDA achieves a convergence rate of $\mathcal{O}(1/N)$, measured by the primal-dual gap function, see \cite[Theorem 3.2]{chang2021goldengrpda}. Furthermore, it was shown that a form of \eqref{1.2.2} with \(\delta=1\) and \eqref{GRPDA} is equivalent when GRPDA is interpreted as a preconditioned ADMM, see \cite[Section 3]{chang2021goldengrpda} for detailed calculations. When the primal-dual algorithm \eqref{1.2.2} is interpreted as preconditioned ADMM, the preconditioned matrix must necessarily be positive definite to prove global convergence, which is guaranteed by the stepsize condition $\tau\sigma\|K\|^2<1$. However, in the case of GRPDA, the preconditioned matrix is allowed to be indefinite, offering a broader choice of parameters for practical applications. Nevertheless, all the above-mentioned algorithms require the knowledge of the norm of the operator $K$ during the application, which is a disadvantage when the norm is unknown.\\
\\
\textbf{Case 2 ($h\neq 0$):} When the smooth term is nonzero, i.e., $ h \neq 0$, all the aforementioned methods can still be applied by replacing $f$ with $f+h$. However, in this case, computing the proximal operator of $f + h$ could be more challenging than that of 
$f$ alone. To address this shortcoming, the authors \cite{condat2013primal, vu2013splitting} extended the PDHG algorithm by incorporating the first-order approximation of the gradient of the differentiable function $h$ into the primal variable updates during the iterations of \eqref{1.2.2}, resulting in the \emph{\text{Condat--V\~u algorithm}}. Given $(x_0, y_0) \in \mathbb{X} \times \mathbb{Y}$ and $n\geq1$, the Condat--V\~u algorithm iterates as
\begin{equation}\label{condat_vu}
    \begin{cases}
         x_n= \operatorname{prox}_{\tau f}(x_{n-1}- \tau K^*y_{n-1}- \tau\nabla h(x_{n-1})),\\
         \Tilde{x}_n = 2x_n-x_{n-1}\\
      y_n= \operatorname{prox}_{\sigma g^*}(y_{n-1}+ \sigma K\Tilde{x}_{n})
    \end{cases}
\end{equation}
In this case, the sequences generated by \eqref{condat_vu} converge under the stepsize condition $\tau\sigma\lVert K\rVert^2+\tau({\bar L}/{2})\leq 1$. In recent years, various three-operator primal-dual algorithms, akin to the Condat–-V\~u algorithm, have been developed to address the problem \eqref{1.1}. These algorithms ensure global convergence by requiring a stepsize inequality that depends on both the $\|K\|$ and the Lipschitz constant $\bar{L}$. Notable examples include the \emph{primal-dual fixed point algorithm} (PDFP) \cite{chen2016primal}, the \emph{primal-dual three-operator splitting} (PD3O) \cite{yan2018new}, and the \emph{primal-dual Davis-Yin splitting} (PDDY) \cite{salim2022dualize}, all of which share the stepsize condition $\tau\sigma\lVert K \rVert^2 < 1 $ and $ \tau \bar{L} < 2 $, ensuring convergence to a saddle point under the assumptions of convexity and global Lipschitz continuity of the smooth component function. More recently, Malitsky and Tam \cite{malitsky2023first&tam} developed a three-operator primal-dual algorithm, also known as the \emph{primal-dual twice reflected} (PDTR) method, which converges globally under the stepsize condition $\tau\sigma\lVert K\rVert^2 + 2\tau\bar{L} < 1$. Since GRPDA \eqref{GRPDA} can be seen as an extension of the PDHG algorithm for $h=0$, Zhou et al. \cite{zhou2022new} similarly extended GRPDA \eqref{GRPDA} to the three-operator splitting (i.e., $h\neq0$) settings, referred to as the \emph{extended Golden Ratio primal-dual algorithm} (E-GRPDA). Given $(x_0, y_0) \in \mathbb{X} \times \mathbb{Y}$ with $z_0=x_0$, E-GRPDA iterates as
\begin{equation}\label{E-GRPDA}
    \begin{cases}
          z_n = \dfrac{\psi-1}{\psi}x_{n-1} + \dfrac{1}{\psi}z_{n-1},\\
         x_n= \operatorname{prox}_{\tau f}(z_{n}- \tau K^*y_{n-1}- \tau\nabla h(x_{n-1})),\\
           y_n  = \operatorname{prox}_{\sigma g^*}\left(y_{n-1} + \sigma Kx_n\right).\\
    \end{cases}
\end{equation}
The sequence generated by \eqref{E-GRPDA} converges, as measured by the function value residual and constraint violation, provided the stepsize condition $\tau\sigma\lVert K\rVert^2+2\tau\bar L(1-\mu)\leq \psi(1-\mu)$ is satisfied with $\mu\in(0,1)$ and $\psi\in(1,\phi]$.

Nonetheless, in all the aforementioned primal-dual algorithms and their extensions, prior knowledge of the operator norm of $K$ and the Lipschitz constants $\bar{L}$ is required to select appropriate values for $\tau$ and $\sigma$ to achieve the desired numerical performance within a given number of iterations. These requirements are often impractical or expensive to compute for many optimization problems. For example, in CT image reconstruction \cite{sidky2012convex}, the exact spectral norm of $K$ is typically large, dense, and expensive to compute. Therefore, for scenarios like this, we need stepsizes that do not depend on the parameters $\bar L$ and $\|K\|$. To avoid this setback, Malitsky and Pock \cite{malitsky2018first} incorporate a backtracking linesearch strategy into a version of the Condat--V\~u algorithm~\eqref{condat_vu}. However, it still requires global Lipschitzness of $\nabla h$. In a similar manner, Chang et al. \cite{chang2022goldenlinesearch} also introduced the linesearch technique into \eqref{GRPDA} to remove the dependency on $\|K\|$. While linesearch enables the algorithm to adaptively determine appropriate stepsizes during execution, it requires an inner loop within each outer iteration to run until a stopping criterion is met, which can be time-consuming. 

Adaptive methods offer effective alternatives to linesearch by refining poor initial stepsizes and eliminating dependence on global Lipschitz continuity by exploiting the approximation of the inverse local Lipschitz constant. In light of this idea, several authors have developed adaptive methods for different cases of \eqref{1.1}. In 2019, Malitsky and Michenko \cite{malitsky2020adaptive} proposed an adaptive stepsize rule for \eqref{1.1} when $f=g=0$, using the local information of the gradient of $h$. More recently, Vladarean et al. \cite{vladarean2021first} introduced an \emph{adaptive primal-dual algorithm} (APDA) for a special case of (\ref{1.1}) when $f=0$, which also uses the local information of $\nabla h$ around the current iterates.
Building on similar concepts, Latafat et al. \cite{latafat2023adaptive} introduced a \emph{fully adaptive three-operator splitting method} (adaPDM+) to solve (\ref{1.1}) by leveraging the local Lipschitz continuity of $\nabla h$ and incorporating a linesearch technique to remove the reliance on $\|K\|$ as a parameter. Independently, Malitsky and Mishchenko \cite{Yuramalitsky} extended their works from \cite{malitsky2020adaptive} and proposed the \emph{adaptive proximal gradient method} (AdaProxGD) for the special case of~(\ref{1.1}) when $g=0$. Given the limitations of linesearch, we aim to develop algorithms for~(\ref{1.1}) that utilize both the local estimates of the Lipschitz constant $\bar L$ and the operator norm of $K$.

In 2020, Malitsky \cite[Algorithm 1]{malitsky2020golden} proposed an \emph{adaptive Golden Ratio algorithm} (aGRAAL) for solving mixed variational inequality problems (\textbf{MVIP}). This method addresses the following \textbf{MVIP} for a given monotone, locally Lipschitz continuous function $F$ on a finite-dimensional Hilbert space $\mathcal{V}$, and a proper, convex, and lsc function $g:\mathcal{V}\to(-\infty, \infty]$.
\begin{equation}\label{malitsky_goolden}
   \text{Find}~ u^* \in \mathcal{V} ~\text{such that}~ \langle F(u^*), u - u^* \rangle + \theta(u) - \theta(u^*) \geq 0,~ \forall u \in \mathcal{V}.
\end{equation}
The framework of (\ref{malitsky_goolden}) can be recovered by writing the optimality conditions of (\ref{1.1}) and taking $u = (x, y)$ with $ F(u) = (\nabla h(x) + K^* y, -Kx) $ and $ \theta(u) = f(x) + g^*(y) $. However, as pointed out in \cite{chang2021goldengrpda,vladarean2021first}, if aGRAAL is applied to this derived variational inequality framework of (\ref{1.1}) with $ u = (x,y), v = (z,w) \in \mathcal{V} = \mathbb{X} \times \mathbb{Y} $, and the inner product $ \langle (x,y), (z,w) \rangle = \langle x, z \rangle + \langle y, w \rangle $, the following primal-dual algorithm is obtained
\begin{algorithm}[H]
\caption{aGRAAL}\label{aGRAAL}
\begin{algorithmic}[1]
    \State \textbf{Initialization:} Let $x_0,x_1\in\mathbb{X}$, $y_0,y_1\in\mathbb{Y}$, and set $(\bar x_0,\bar y_0)=(x_0,y_0)$. Suppose $\lambda_0>0$, $\psi\in(1,\phi]$ (where $\phi=\tfrac{1+\sqrt5}{2}$ is the golden ratio), $\rho =\psi^{-1}+\psi^{-2}$, $\theta_0>0$, and $\lambda_{\max}\gg0$.
    \State \textbf{Step 1:} \emph{Update}
    \begin{equation*}
        \lambda_n = \min\left\{ \rho \lambda_{n-1}, ~\frac{\psi \theta_{n-1}}{4 \lambda_{n-1}} \frac{\|u_n - u_{n-1}\|^2}{\|F(u_n) - F(u_{n-1})\|^2}, ~\lambda_{\max} \right\}.
    \end{equation*}
    
    \State \textbf{Step 2:} \emph{Compute}
    \begin{equation}\label{aGRAAL_primal_dual}
        \begin{aligned}
            \bar{x}_n &= \dfrac{(\psi - 1) x_n + \bar x_{n-1}}{\psi}, \quad
            x_{n+1} = \operatorname{prox}_{\lambda_n f}\left(\bar{x}_n - \lambda_n K^* y_n - \lambda_n \nabla h(x_n)\right),\\
            \bar{y}_n &= \dfrac{(\psi - 1) y_n + \bar y_{n-1}}{\psi}, \quad
            y_{n+1} = \operatorname{prox}_{\lambda_n g^*}\left(\bar{y}_n + \lambda_n K x_{n}\right).
        \end{aligned}
    \end{equation}
    
    \State \textbf{Step 3:} \emph{Update}
    \begin{equation*}
        \theta_n = \frac{\psi\lambda_n}{\lambda_{n-1}}.
    \end{equation*}
    
    \State \textbf{Step 4:} Let $n \leftarrow n+1$ and return to \textbf{Step 1}.
\end{algorithmic}
\end{algorithm}

The Algorithm \ref{aGRAAL} presented above is a Jacobian-type algorithm that does not leverage information from the current iterate. More precisely, the immediate primal update $x_{n+1}$ is not used in calculating the subsequent dual update $y_{n+1}$. Furthermore, in both the $x$-and $ y$-subproblems of \eqref{aGRAAL_primal_dual}, the primal and dual stepsizes are equal, which can be limiting as it may result in poorer convergence bounds. Therefore, an adaptive method that performs updates in a Gauss-Seidel manner at each iteration would be beneficial in this case, allowing for efficient resolution of (\ref{E-GRPDA}) while leveraging the problem's inherent structure.
\subsection{Motivations and Contributions}

Our main motivation is to propose new stepsize strategies for the extended Golden Ratio primal-dual algorithm (\ref{E-GRPDA}), where the primal-dual stepsizes are independent of the parameters $\|K\|$ and $\bar L$. Broadly speaking, our motivations are twofold. Firstly, we propose a stepsize rule for E-GRPDA (\ref{E-GRPDA}) in which the primal stepsize \(\tau_n\) is non-increasing and separated from zero. The existence of the global Lipschitz constant $\bar L$ suffices to prove the global convergence of our proposed Algorithm \ref{algorithm 1}. Nevertheless, verifying such global smoothness can be rather challenging, particularly when the function $h$ exhibits a complex nonsmooth structure. Moreover, non-increasing stepsizes may result in slower convergence and might require more iterations to reach a specified error bound.

Secondly, to address these issues, we propose an adaptive method (Algorithm \ref{algorithm 2}) for (\ref{E-GRPDA}), where the primal stepsize $\tau_n$ adapts to the local geometry of the gradient of $h$, thereby eliminating the dependency on global smoothness of $h$. However, it is worth noting that, although the primal stepsize $\tau_n$ does not adapt to $\|K\|$, this is not a substantial limitation because, in many cases, the Lipschitz constant $\bar L$ often exceeds $\|K\|$ and accurately estimating $\bar L$ itself can be quite subtle or even infeasible, see \cite[Remark 4.1]{vladarean2021first} for a similar comment. 

The rest of this paper outlines the following contributions:
\begin{itemize}
    \item In Section \ref{sec2}, we discuss some useful lemmas and blanket assumptions that apply to the proposed algorithms throughout the paper.
    \vspace{.1cm}
    \item In Section \ref{sec3}, we analyze the extended GRPDA (\ref{E-GRPDA}) with a new stepsize rule proposed in Algorithm \ref{algorithm 1}, which, unlike traditional approaches, does not require knowledge of $\bar L$ and eliminates the need for backtracking by using the estimates of $\|K\|$ around the current iterates. We establish both global iterate convergence and an ergodic $\mathcal{O}(1/N)$ rate of convergence, where $N$ is the iteration count, measured by function value residual and constraint violation. By carrying out refined analysis, we extend the region of convergence of Algorithm \ref{algorithm 1} from $(1,\phi]$ to $(1,1+\sqrt{3})$ as it was done for GRPDA \eqref{GRPDA} in \cite{chang2022grpdarevisited}. Furthermore, we show that when the initial primal stepsize $\tau_0$ is less than or equal to the positive constant $\min\{\frac{\mu}{\sqrt{\beta}\|K\|},\frac{\mu'}{\bar L}\}$, the stepsize condition for (\ref{E-GRPDA}) can still be recovered, which is an advancement for the E-GRPDA algorithm.
    \vspace{.1cm}
    \item In Section \ref{sec4}, we demonstrate nonergodic R-linear convergence rate of Algorithm~\ref{algorithm 1} when both $g^*$ and $h$ are strongly convex functions under some appropriate parameter choices.
    \vspace{.1cm}
    \item In Section \ref{sec5}, we introduce an adaptive stepsize rule to solve (\ref{E-GRPDA}) as formulated in Algorithm~\ref{algorithm 2}, where the stepsize $\tau_n$ can increase in flatter regions, unlike in Algorithm~\ref{algorithm 1}. This approach also does not require backtracking and further improves upon Algorithm~\ref{algorithm 1} by using the local Lipschitz information of $\nabla h$ instead of relying on global Lipschitz continuity. Moreover, it can readily adjust to poor initial stepsizes. We establish the global convergence of the method and demonstrate an $\mathcal{O}(1/N)$ sublinear rate result, measured by functional value residual and constraint violation.

    \item Section \ref{section_6} presents numerical experiments on various optimization problems, including lasso, logistic regression, graph net, fused lasso, and image inpainting. We demonstrate the applicability and performance of the proposed strategies, showing that they outperform existing methods.
\end{itemize}

\section{Preliminaries and Assumptions}\label{sec2}
Throughout this paper, we denote the Golden Ratio by $\phi = \frac{1 + \sqrt{5}}{2}$. The norm of a linear operator $K:\mathbb{X}\to\mathbb{Y}$ is defined by $\lVert K \rVert = \operatorname{sup}\{ \lVert Kx\rVert \colon \lVert x \rVert = 1,~ x \in \mathbb{X}\}$.
\vspace{5pt}

Given a proper, convex and lsc function $f: \mathbb{X}\to (-\infty,\infty]$, the \emph{effective domain} of $f$ is given by $\operatorname{dom}(f) = \{x\in\mathbb{X}\colon f(x)<+\infty\}$. The \emph{subdifferential} of $f$ at $x^*\in\operatorname{dom}(f)$ is denoted by $\partial f(x^*)$, and $\partial f: \mathbb{X}\to 2^{\mathbb{X}}$ is defined as
\begin{equation*}
    \partial f(x^*)=\left\{v\in\mathbb{X}: f(x)\geq f(x^*)+\langle v,\, x-x^*\rangle ~\forall x\in\mathbb{X}\right\}.
\end{equation*}
For $\beta>0$, the \emph{proximal operator} \cite[Example 2.5.1]{wotaoyu2022large} of $f$ with the parameter $\beta$ is given by
\begin{equation*}
    \operatorname{prox}_{\beta f}(x)= \argmin_{y\in\mathbb{X}}\big \{f(y)+\dfrac{1}{2\beta}{\lVert y-x\rVert}^2\big\}.
\end{equation*}
For any $\beta>0$, the proximal operator $\operatorname{prox}_{\beta h}$ is well defined; see \cite[Section 1.3.11]{wotaoyu2022large}. Now, we are ready to state the following lemma.
\vspace{5pt}
 \begin{lemma}\label{usefullemma3}
    \cite[Proposition 12.26]{bauschke2017correction} Suppose that $f: \mathbb{X}\to (-\infty,\infty]$ is a proper, convex and lsc function. Then, for any $x\in\mathbb{X}$ and $\lambda>0$, we have $\Tilde q= \operatorname{prox}_{\lambda f}(x)$ if and only if  
\begin{equation*}
    \lambda(f(\Tilde q)-f(y))\leq \langle \Tilde q-x,\, y-\Tilde q \rangle, ~\forall y\in\mathbb{X}.
\end{equation*}
\end{lemma}

Given a non-empty set $D$, the \emph{indicator function} $i_{D}$ is defined as $i_D(x)=0$ if $x\in D$, and $i_D(x)=\infty$ otherwise. Observe that problem \eqref{1.1} can be equivalently reformulated as a constrained optimization problem by introducing an auxiliary variable $w \in \mathbb{Y}$. Specifically, we consider
\begin{equation}\label{eq:lagrangian}
    \min_{x,\,w} \; f(x) + g(w) + h(x) \quad \text{subject to} \quad Kx = w.
\end{equation}
Let $y \in \mathbb{Y}$ denote the Lagrange multiplier for the constraint $Kx=w$. Then, we define the objective function as
\begin{equation*}\label{eq:lagrangian objective}
    \Phi(x, w) := f(x) + h(x) + g(w),
\end{equation*}
and the corresponding Lagrangian as
\begin{equation*}\label{eq:lagrangian function}
    \mathscr{L}(x, w, y) := \Phi(x, w) + \langle y,\, Kx - w \rangle.
\end{equation*}

Next, we make the following blanket assumptions.
\vspace{5pt}
\begin{assumption}\label{assumption1}
    We assume that the set of solutions of (\ref{1.2}) is non-empty. Moreover, it holds that $0\in\operatorname{ri}\big(K(\operatorname{dom}(f))-\operatorname{dom}(g)\big)$.
\end{assumption}
\vspace{5pt}
\indent
Here, $``\operatorname{ri}"$ denotes the \emph{relative interior} \cite[Definition 6.9]{bauschke2017correction} of a convex set. Under Assumption \ref{assumption1}, it follows from \cite[Corollaries 28.2.2 and 28.3.1]{Rockafellar+1970} that $(\bar x, \bar w, \bar y)$ is a saddle point of $\mathscr{L}(\cdot)$ if and only if $(\bar x, \bar w)$ is a solution of \eqref{eq:lagrangian} and $\bar y$ is solution to the dual problem (\ref{1.3}), respectively, which is equivalently characterized by the following inequality
\begin{equation*}
\mathscr{L}(\bar{x}, \bar w,y) \leq \mathscr{L}(\bar{x}, \bar w,\bar{y}) \leq \mathscr{L}(x, w,\bar{y}), ~ \forall (x,w, y) \in \mathbb{X} \times \mathbb{Y}\times\mathbb{Y}. 
\end{equation*}
In this context, the pair $(\bar{x}, \bar{y})$ solves the saddle-point (min-max) problem \eqref{1.2}. Throughout the paper, we denote by $\mathbf{\Omega}$ the set of all saddle points of $\mathscr{L}$, which is nonempty under Assumption \ref{assumption1} and is explicitly written as
\begin{equation*}
\mathbf{\Omega} = \left\{ (\bar{x}, \bar{w}, \bar{y}) \in \mathbb{X} \times \mathbb{Y} \times \mathbb{Y} :
-K^*\bar{y} \in \partial f(\bar{x}) + \nabla h(\bar{x}),~\bar{y} \in \partial g(\bar{w}),~ K\bar{x} = \bar{w} \right\}.
\end{equation*}
Let $(\bar x, \bar w,\bar y)\in\mathbf{\Omega}$ be any saddle point. Then, for given $(x,w,y)\in\mathbb{X}\times\mathbb{Y}\times\mathbb{Y}$, we define
\begin{equation}\label{primal_dual_gap}
\mathbb{J}(x,w,y)=\mathscr{L}(x,w,y)-\mathscr{L}(\bar x, \bar w, y) = \Phi(x,w)+\langle y,\, Kx-w \rangle -\Phi(\bar x,\, \bar w).
\end{equation}
We use this conventional measure for our convergence analysis because the so-called primal-dual gap function might vanish at non-stationary points, and convergence results based on this can be vacuous, see \cite{chang2022grpdarevisited, zhou2022new, chambolle2011first} for a similar discussion.
\vspace{8pt}
\begin{assumption}\label{assumption_2} 
    The proximal operators of $f$ and $g$ are ``proximal friendly"; that is, the proximal operators of $f$ and $ g$ can be evaluated efficiently.
\end{assumption}
\vspace{8pt}
Both Assumption \ref{assumption1} and Assumption \ref{assumption_2} are common in literature, see \cite{chang2021goldengrpda, chang2022goldenlinesearch,vladarean2021first}. Note that, under Assumption \ref{assumption_2}, the proximal operator of the conjugate $g^*$ can be computed using the \emph{Moreau decomposition} \cite[Theorem 6.45]{beck2017first}. We end this section with the following useful lemmas.
\vspace{8pt}
\begin{lemma}\label{UsefulLemma1} 
    \cite{bauschke2017correction} For any $a,b,c\in \mathbb{X}$ and $\alpha\in\mathbb{R}$, we have
    \begin{subequations}\label{14}
\begin{align}
    \langle a-b,\,a-c\rangle & = \dfrac{1}{2}\lVert a-b \rVert^2 + \dfrac{1}{2}\lVert a-c \rVert^2 - \dfrac{1}{2}\lVert b-c \rVert^2,\label{14a} \\
     \lVert \alpha a + (1-\alpha) b \rVert^2 & = \alpha \lVert a \rVert^2 + (1-\alpha)\lVert b \rVert^2 - \alpha(1-\alpha)\lVert a-b \rVert^2.\label{14b}
\end{align}
\end{subequations}
\end{lemma}
\vspace{5pt}
\begin{lemma}\label{usefullemma2}
    \cite{chang2022grpdarevisited} Suppose that $(p_n)$ and $(q_n)$ are two nonnegative real sequences and that there exists a natural number $n_0$ such that $p_{n+1}\leq p_n - q_n$, $\forall n\geq n_0$. Then, $\sum_{n=1}^{\infty}q_n<\infty$  and $\displaystyle \lim_{n\to\infty} p_n$ exists. 
\end{lemma}
\vspace{8pt}
\begin{lemma}\label{usefullemma_3}
     For $m_1,m_2\in\mathbb{R}$, $p\geq 0$, and $q>0$, we have $\frac{pq}{p+q}(m_1+m_2)^2 \leq pm_1^2 + qm_2^2$.
\end{lemma}
\section{GRPDA with partially adaptive stepsize}\label{sec3}
In this section, we introduce a partially adaptive stepsize rule for E-GRPDA \eqref{E-GRPDA} that neither requires prior knowledge of the norm of the operator $K$ nor the Lipschitz constant $\bar L$ as inputs of the algorithm. We refer to this algorithm, which is presented in Algorithm \ref{algorithm 1}, as the Partially adaptive Golden Ratio Primal-Dual Algorithm (P-GRPDA). By ``partially adaptive,'' we indicate that the stepsize $(\tau_n)$ is constrained from increasing at any iteration.
\vspace{8pt}
\begin{assumption}\label{assumption4}
    Suppose that $f$ and $g$ are proper, convex, and lsc, and $h$ is convex and $\bar L$-smooth.
\end{assumption} 
\begin{algorithm}[H]
\caption{Partially Adaptive GRPDA (P-GRPDA) for (\ref{E-GRPDA})}
\label{algorithm 1}
\begin{algorithmic}[1]
\State \textbf{Initialization:} 
Let $x_0 \in \mathbb{X}$, $y_0 \in \mathbb{Y}$, and set $z_0 = x_0$. Choose $\beta > 0$, $0 < 2\mu' < \mu < \psi/2$, where $\psi \in (1, \phi]$. Suppose $\tau_0 > 0$ and $n = 1$.
    \State \textbf{Step 1:} \emph{Compute} 
        \begin{align}
            z_n &= \dfrac{\psi-1}{\psi}x_{n-1} + \dfrac{1}{\psi}z_{n-1} \label{eqn:39a},\\[1pt]
            x_{n} &= \operatorname{prox}_{\tau_{n-1} f}(z_n-\tau_{n-1} K^*y_{n-1}-\tau_{n-1}\nabla h(x_{n-1})). \label{eqn:39b}
            \end{align}
    \State \textbf{Step 2:} \emph{Update}
    \begin{equation}
         \begin{aligned}
            \tau_n &= \min\left\{\tau_{n-1}, ~\dfrac{\mu\|x_n-x_{n-1}\|}{\sqrt{\beta}\|Kx_n-Kx_{n-1}\|},~ \dfrac{\mu'\|x_n-x_{n-1}\|}{\|\nabla h(x_n)-\nabla h(x_{n-1})\|}\right\} ,\label{eq:39c}\\
           \sigma_n &=\beta\tau_n.
        \end{aligned}
    \end{equation}
    \State \textbf{Step 3:} \emph{Compute}    
        \begin{align}
          w_n &= \operatorname{prox}_{\frac{1}{\sigma_n}g}\Bigl(\frac{y_{n-1}}{\sigma_n} + Kx_n\Bigr)\label{eq:w_n update},\\
          y_n &= y_{n-1} + \sigma_n\bigl(Kx_n - w_n\bigr)\label{eq:y_n update}.
        \end{align}

    \State \textbf{Step 4:} Let $n \leftarrow n+1$ and return to \textbf{Step 1}.
\end{algorithmic}
\end{algorithm}
Before turning our attention to the convergence analysis, some comments regarding Algorithm \ref{algorithm 1} are in order.
\vspace{8pt}
 \begin{remark}\label{Remark_tau_n}
        Notice that using Moreau decomposition both \eqref{eq:w_n update} and \eqref{eq:y_n update} can be compactly written as $y_n=\operatorname{prox}_{\sigma_n g^*}\bigl(y_{n-1} + \sigma_n Kx_n\bigr)$. Here, $\beta>0$ is an algorithmic parameter to maintain the balance in convergence by scaling the dual steps. In (\ref{eq:39c}), when $Kx_n = Kx_{n-1}$ or $\nabla h(x_n) = \nabla h(x_{n-1})$ holds for some $n$, we adopt the convention that $1/0 = \infty$. Under this convention, the stepsize $\tau_n$ is chosen as the minimum of $\tau_{n-1}$ and the value of the last two terms in (\ref{eq:39c}) whose denominator is nonzero. Additionally, if $x_n = x_{n-1}$, we adopt the convention $\frac{0}{0}=\infty$, and set $\tau_n = \tau_{n-1}$. Moreover, when $Kx_n = Kx_{n-1}$ and $\nabla h(x_n) = \nabla h(x_{n-1})$ simultaneously hold for some $n$, the convention $\frac{1}{0} = \infty$ ensures that $\tau_n$ defaults to the previous primal stepsize $\tau_{n-1}$.
 \end{remark}
\vspace{8pt}
\begin{remark}
 From \eqref{eq:39c}, it follows that $(\tau_n)$ is non-increasing, which may slow down Algorithm \ref{algorithm 1} for a given convergence bound. However, larger $\mu$ and $\mu'$ for a given $\psi$ play a crucial role in estimating $\|K\|$ and $\bar{L}$. Therefore, in practice, the largest pair $(\mu, \mu')$ is selected for a given $\psi$. In Theorem \ref{larger_mu-and_mu'}, we extend the allowable values of $\mu$ and $\mu'$ for $1 < \psi < 1 + \sqrt{3}$.
 \end{remark}
Next, we present the following proposition.
\vspace{8pt}
\begin{proposition}\label{tau_n_bdd_below}
    Let $(\tau_n)$ be the stepsize sequence generated by Algorithm \ref{algorithm 1}. Then, $(\tau_n)$ is bounded below by $\eta \coloneqq \min\left\{ \tau_0, \frac{\mu}{\sqrt{\beta}\|K\|}, \frac{\mu'}{\bar{L}} \right\}$ and $\lim_{n\to\infty}\tau_n\geq\eta>0$.
\end{proposition}

\begin{proof}
We use induction to prove our first claim. First note that, from the definition of $\tau_n$ in \eqref{eq:39c}, we have $\tau_n\leq\tau_{n-1},~\forall n\geq1$. Noting the conventions described in Remark~\ref{Remark_tau_n}, the definition of $\|K\|$ together with Lipschitz continuity of $\nabla h$ yields the following inequality for all $n\geq 1$

\begin{equation}\label{inequlaity_1}
    \begin{aligned}
    \tau_{n} &= \min\left\{\tau_{n-1}, \dfrac{\mu\|x_n-x_{n-1}\|}{\sqrt{\beta}\|Kx_n-Kx_{n-1}\|}, \dfrac{\mu'\|x_n-x_{n-1}\|}{\|\nabla h(x_n)-\nabla h(x_{n-1})\|}\right\}\\
    &\geq\min\left\{ \tau_{n-1}, \frac{\mu}{\sqrt{\beta} \|K\|}, \frac{\mu'}{\bar{L}} \right\}.
\end{aligned}
\end{equation}
Now, for $n=1$, from \eqref{inequlaity_1}, we have $\tau_1 \geq\eta$, where $\eta = \min\left\{ \tau_0, \frac{\mu}{\sqrt{\beta} \|K\|}, \frac{\mu'}{\bar{L}} \right\}$. Suppose $\tau_n\geq\eta$ holds for some $n$. Then, by induction and \eqref{inequlaity_1}, we have 
\begin{equation*}
    \begin{aligned}
    \tau_{n+1}&\geq\min\left\{ \tau_n, \frac{\mu}{\sqrt{\beta} \|K\|}, \frac{\mu'}{\bar{L}} \right\} \geq \eta.
\end{aligned}
\end{equation*}

Therefore, we conclude the first claim of our proof.
For the second claim, note that the sequence $(\tau_n)$ is non-increasing and bounded below by $\eta>0$. Consequently, this implies that $(\tau_n)$ is convergent and $\lim_{n \to \infty} \tau_n = \tau \geq \eta > 0$ for some $\tau>0$.
\end{proof}

\subsection{Convergence analysis} 
In this section, we establish the convergence of Algorithm \ref{algorithm 1}.
\vspace{.1cm}
\begin{lemma}\label{Lemma4.1.1}
    Let $(\bar x , \bar w, \bar y)\in\mathbf{\Omega}$. Under Assumptions \ref{assumption1}, \ref{assumption_2} and \ref{assumption4}, let $\{(z_n,x_n,w_n,y_n,\tau_n)\}$ be the sequence generated by Algorithm \ref{algorithm 1}. Then, for any $y\in\mathbb{Y}$, there exists a natural number $n_2$ such that, $\forall n\geq n_2$, the following holds
   \begin{multline}\label{lemma4.1.1.1}
        2\tau_n\mathbb{J}(x_n,w_n, y) + \frac{\psi}{\psi-1}\lVert \bar x-z_{n+2} \rVert^2 
        + \frac{1}{\beta}\lVert y-y_{n} \rVert^2 
        + \mu'\lVert x_n-x_{n+1} \rVert^2 \\
        \leq \frac{\psi}{\psi-1}\lVert \bar x-z_{n+1} \rVert^2 
        + \frac{1}{\beta}\lVert \bar y-y_{n-1} \rVert^2
        + \mu'\lVert x_n-x_{n-1} \rVert^2 - \psi\bar \theta_n\lVert x_n-z_{n+1} \rVert^2,
    \end{multline}
where $\bar\theta_n = \dfrac{\tau_n}{\tau_{n-1}}$.
\end{lemma}
\begin{proof}
Using (\ref{eqn:39b}), (\ref{eq:w_n update}), \eqref{eq:y_n update}, and Lemma \ref{usefullemma3}, we obtain
\begin{equation}\label{3.2.1}
\begin{aligned}
\tau_{n}\big(f(x_{n+1})-f(\bar x)\big)
&\le \left\langle x_{n+1}-z_{n+1}+\tau_n K^*y_n+\tau_n\nabla h(x_n),\,\bar x-x_{n+1}\right\rangle,\\
g(w_n)-g(\bar w) &\le \langle y_n, \,w_n-\,\bar w\rangle.
\end{aligned}
\end{equation}
Similarly, since $x_n-z_n=\psi(x_n-z_{n+1})$,
\begin{equation}\label{3.2.3}
\begin{aligned}
\tau_{n-1}\big(f(x_n)-f(x_{n+1})\big)
&\le \left\langle \psi(x_n-z_{n+1})+\tau_{n-1}K^*y_{n-1}+\tau_{n-1}\nabla h(x_{n-1}),\,x_{n+1}-x_n\right\rangle .
\end{aligned}
\end{equation}
Multiplying the second inequality in \eqref{3.2.1} by $\tau_n$ and \eqref{3.2.3} by
$\bar\theta_n=\dfrac{\tau_n}{\tau_{n-1}}$, and adding them to the first inequality in \eqref{3.2.1}, we get
\begin{multline}\label{3.2.5}
\tau_n\left(f(x_n)-f(\bar x)+g(w_n)-g(\bar w)\right)
\le \langle x_{n+1}-z_{n+1}, \,\bar x-x_{n+1}\rangle\\
+ \psi\bar\theta_n\langle x_n-z_{n+1},\,x_{n+1}-x_n\rangle
+ \tau_n\langle K^*y_{n-1}, \,x_{n+1}-x_n\rangle
+ \tau_n\langle K^*y_n,\,\bar x-x_{n+1}\rangle\\
+ \tau_n\langle y_n,\,w_n-\bar w\rangle
+ \tau_n\langle \nabla h(x_{n-1}),\,x_{n+1}-x_n\rangle
+ \tau_n\langle \nabla h(x_n),\,\bar x-x_{n+1}\rangle.
\end{multline}
Note that following $K\bar x = \bar w$, we have
\begin{align}\label{eq:Ky-collect}
&\langle K^*y_{n-1},\,x_{n+1}-x_n\rangle
  + \langle K^*y_n,\,\bar x-x_{n+1}\rangle
  + \langle y_n,\,w_n-\bar w\rangle \nonumber\\
&= \langle y_{n-1},\,K(x_{n+1}-x_n)\rangle
   + \langle y_n,\,-Kx_{n+1}+w_n\rangle\nonumber\\
&= \langle y_{n-1},\,K(x_{n+1}-x_n)\rangle
   - \langle y_n,\,K(x_{n+1}-x_n)\rangle
   + \langle y_n,\,-Kx_n+w_n\rangle \nonumber\\
&= \langle y_{n-1}-y_n,\,K(x_{n+1}-x_n)\rangle
   + \langle y_n,\,-Kx_n+w_n\rangle .
\end{align}
Furthermore, the last two terms of \eqref{3.2.5} can be expressed as
\begin{align}\label{eq:grad-collect}
&\;\langle \nabla h(x_{n-1}),\,x_{n+1}-x_n\rangle
  + \langle \nabla h(x_n),\,\bar x-x_{n+1}\rangle \nonumber\\
&= \langle \nabla h(x_n)-\nabla h(x_{n-1}),\,x_n-x_{n+1}\rangle
   + \langle \nabla h(x_n),\,\bar x-x_n\rangle .
\end{align}
Substituting \eqref{eq:Ky-collect}–\eqref{eq:grad-collect} into \eqref{3.2.5} yields
\begin{multline}\label{3.2.5_new}
\tau_n\!\left(f(x_n)-f(\bar x)+g(w_n)-g(\bar w)\right)
\le \langle x_{n+1}-z_{n+1},\,\bar x-x_{n+1}\rangle\\
+ \psi\bar\theta_n\langle x_n-z_{n+1},\,x_{n+1}-x_n\rangle
+ \tau_n\langle y_n,\,-Kx_n+w_n\rangle
+ \tau_n\langle y_{n-1}-y_n,\,K(x_{n+1}-x_n)\rangle\\
+ \tau_n\langle \nabla h(x_n)-\nabla h(x_{n-1}),\,x_n-x_{n+1}\rangle
+ \tau_n\langle \nabla h(x_n),\,\bar x-x_n\rangle.
\end{multline}
Moreover, from the $y$–update \eqref{eq:y_n update} with $\sigma_n=\beta\tau_n$, we have
\begin{equation}\label{eq:y-id}
\tau_n\langle y,\,Kx_n-w_n\rangle=\frac{1}{\beta}\langle y-y_n,\,y_n-y_{n-1}\rangle +\tau_n\langle y_n ,\, Kx_n-w_n\rangle.
\end{equation}
By adding \eqref{eq:y-id} to \eqref{3.2.5_new}, we obtain
\begin{multline}\label{lemma4.1.0}
\tau_n\!\left(f(x_n)-f(\bar x)+\langle y, Kx_n-w_n\rangle+g(w_n)-g(\bar w)\right)
\le \langle x_{n+1}-z_{n+1},\,\bar x-x_{n+1}\rangle\\
+ \psi\bar\theta_n\langle x_n-z_{n+1},\,x_{n+1}-x_n\rangle
+ \frac{1}{\beta}\langle y-y_n,\,y_n-y_{n-1}\rangle
+ \tau_n\langle y_n-y_{n-1},\, K(x_{n+1}-x_n)\rangle\\
+ \tau_n\langle \nabla h(x_n)-\nabla h(x_{n-1}),\,x_n-x_{n+1}\rangle
+ \tau_n\langle \nabla h(x_n),\,\bar x-x_n\rangle.
\end{multline}
By convexity of $h$, we have
$$
\tau_n\langle \nabla h(x_n),\,\bar x-x_n\rangle\leq \tau_n\left(h(\bar x)-h(x_n)\right).
$$
By combining the previous inequality with \eqref{lemma4.1.0} and using the definition of $\mathbb{J}$ in~\eqref{primal_dual_gap}, we have
\begin{multline}\label{lemma4.1_.1}
\tau_n\,\mathbb{J}(x_n,w_n,y)
\le \langle x_{n+1}-z_{n+1},\,\bar x-x_{n+1}\rangle
+ \psi\bar\theta_n\langle x_n-z_{n+1},\,x_{n+1}-x_n\rangle\\
+ \frac{1}{\beta}\langle y-y_n, \, y_n-y_{n-1}\rangle
+ \tau_n\langle y_n-y_{n-1},\, K(x_{n+1}-x_n)\rangle\\
+ \tau_n\langle \nabla h(x_n)-\nabla h(x_{n-1}),\,x_n-x_{n+1}\rangle .
\end{multline}
Furthermore, by applying~\eqref{14a} to the first three terms on the RHS of~\eqref{lemma4.1_.1}, we obtain
 \begin{multline}\label{lemma4.1_2}
     \tau_n\mathbb{J}(x_n,w_n, y)+  \dfrac{1}{2}\lVert \bar x-x_{n+1} \rVert^2 + \dfrac{1}{2\beta}\lVert y-y_{n} \rVert^2\leq \dfrac{1}{2}\lVert \bar x-z_{n+1} \rVert^2+ \dfrac{1}{2\beta}\lVert y-y_{n-1} \rVert^2 \\ 
     -\dfrac{1}{2}\lVert x_{n+1}-z_{n+1} \rVert^2
    - \dfrac{\psi\bar\theta_n}{2}\lVert x_n-z_{n+1} \rVert^2-\dfrac{\psi\bar\theta_n}{2}\lVert x_n-x_{n+1} \rVert^2\\
    + \dfrac{\psi\bar\theta_n}{2}\lVert x_{n+1}-z_{n+1} \rVert^2 -\dfrac{1}{2\beta}\lVert y_n-y_{n-1} \rVert^2  + \tau_n\langle y_n-y_{n-1},\, K(x_{n+1}-x_n)\rangle\\ +
        \tau_n\langle \nabla h(x_n)-\nabla h(x_{n-1}),\, x_n-x_{n+1}\rangle.
 \end{multline}
Now, using the update in~\eqref{eq:39c} along with the Cauchy--Schwarz inequality, we have
 \begin{equation}\label{lemma4.1_3}
     \begin{aligned}
        \tau_n\langle \nabla h(x_n)-\nabla h(x_{n-1}), x_n-x_{n+1}\rangle&\leq
         \tau_n\lVert \nabla h(x_n)-\nabla h(x_{n-1}) \rVert \lVert x_n-x_{n+1} \rVert \\ & 
         \leq{\mu'}\lVert  x_n - x_{n-1} \rVert \lVert x_n-x_{n+1} \rVert\\ &
          \leq\dfrac{\mu'}{2}\lVert  x_n - x_{n-1} \rVert^2 + \dfrac{\mu'}{2}\lVert x_n-x_{n+1} \rVert^2.
     \end{aligned}
 \end{equation}
Similarly, from (\ref{eq:39c}), we obtain 

\begin{equation}\label{lemma4.1_4}
    \begin{aligned}
       \tau_n\langle y_n-y_{n-1},\, K(x_{n+1}-x_n)\rangle&\leq  \tau_n \lVert Kx_n-Kx_{n+1} \rVert\lVert y_n-y_{n-1} \rVert \\ & 
        \leq \frac{\tau_n}{\tau_{n+1}}\tau_{n+1}\lVert Kx_n-Kx_{n+1} \rVert\lVert y_n-y_{n-1} \rVert\\&
        \leq\frac{\tau_n}{\tau_{n+1}}\frac{\mu}{\sqrt{\beta}}\lVert x_n-x_{n+1} \rVert\lVert y_n-y_{n-1} \rVert\\&
        \leq\dfrac{\mu\tau_n^2}{2\tau_{n+1}^2}\lVert x_n-x_{n+1} \rVert^2 +\dfrac{\mu}{2\beta}\lVert y_n-y_{n-1} \rVert^2.
    \end{aligned}
\end{equation}
Also, from equation (\ref{eqn:39a}), see that
\begin{equation}\label{derived_equality2}
    \lVert \bar x-x_{n+1} \rVert^2= \dfrac{\psi}{\psi-1}\lVert \bar x-z_{n+2} \rVert^2 - \dfrac{1}{\psi-1}\lVert \bar x-z_{n+1} \rVert^2+ \dfrac{1}{\psi}\lVert x_{n+1}-z_{n+1} \rVert^2.
    \end{equation}
Combining (\ref{lemma4.1_2}), (\ref{lemma4.1_3}), (\ref{lemma4.1_4}) and (\ref{derived_equality2}), we get
     \begin{multline}\label{lemma4.1_5}
        2\tau_n\mathbb{J}(x_n,w_n, y) + \frac{\psi}{\psi-1}\lVert \bar x-z_{n+2} \rVert^2 + \frac{1}{\beta}\lVert y-y_n \rVert^2 
        \leq \frac{\psi}{\psi-1}\lVert \bar x-z_{n+1} \rVert^2 + \frac{1}{\beta}\lVert y-y_{n-1} \rVert^2 \\ + \left(\psi\bar\theta_n - 1 - \frac{1}{\psi}\right)\lVert x_{n+1}-z_{n+1} \rVert^2 
        - \psi\bar\theta_n\lVert x_n-z_{n+1} \rVert^2 + \mu'\lVert x_n-x_{n-1} \rVert^2 \\ - \left(\psi\bar\theta_n - \mu\frac{\tau_n^2}{\tau_{n+1}^2} - \mu'\right)\lVert x_n-x_{n+1} \rVert^2 - \left(\frac{1}{\beta}- \frac{\mu}{\beta}\right)\lVert y_n-y_{n-1} \rVert^2.
    \end{multline}
Noting the fact that $(\tau_n)$ is non-increasing, we have 
\begin{equation}\label{psi_ineq}
    \psi\bar\theta_n - 1 - \frac{1}{\psi} \leq \psi - 1 - \frac{1}{\psi}\leq 0, ~~\forall \psi\in(1,\phi].
\end{equation}
From Proposition \ref{tau_n_bdd_below}, observe that $\lim_{n\to\infty}\left(\psi\bar\theta_n-\mu\frac{\tau_n^2}{\tau_{n+1}^2}- \mu'\right)=\psi-\mu-\mu'> 2\mu-\mu-\mu'=\mu-\mu'$. So, there exists a natural number $n_2$ such that
\begin{equation}\label{lemma4.1_6}
    \psi\bar\theta_n-\mu\frac{\tau_n^2}{\tau_{n+1}^2}- \mu'> \mu-\mu',~\forall n\geq n_2.
\end{equation}
Furthermore, observe that $\mu<\dfrac{\psi}{2}<1\implies\left(\dfrac{1}{\beta}-\dfrac{\mu}{\beta}\right)>0$. Thus, by combining (\ref{lemma4.1_5}), (\ref{psi_ineq}) and (\ref{lemma4.1_6}), $\forall n \geq n_2$, we obtain
\begin{multline}\label{lemma4.1_7}
     2\tau_n\mathbb{J}(x_n,w_n,y)+ \dfrac{\psi}{\psi-1}\lVert \bar x-z_{n+2} \rVert^2 + \dfrac{1}{\beta}\lVert y-y_{n} \rVert^2 + (\mu-\mu')\lVert x_n-x_{n+1} \rVert^2\\ \leq \dfrac{\psi}{\psi-1}\lVert \bar x-z_{n+1} \rVert^2 + \dfrac{1}{\beta}\lVert y-y_{n-1} \rVert^2 + {\mu'}\lVert x_n-x_{n-1} \rVert^2 - {\psi\bar\theta_n}\lVert x_n-z_{n+1} \rVert^2.
\end{multline}
By using the fact that $\mu-\mu' > \mu'$ in~\eqref{lemma4.1_7}, we conclude the proof of Lemma~\ref{Lemma4.1.1}.

\end{proof}
\begin{theorem}\label{thm_3.1.1}
   Let $(\bar x, \bar w, \bar y)\in\Omega$. Under Assumptions \ref{assumption1}, \ref{assumption_2} and \ref{assumption4}, let $\{(z_n,x_n,w_n,y_n,\tau_n)\}$ be the sequence generated by Algorithm \ref{algorithm 1}. Then $\{(x_n,y_n)\}$ converges to a solution of (\ref{1.2}).
\end{theorem}
\begin{proof}
   Since $(\bar x, \bar w, \bar y)$ is a saddle point, it follows from \eqref{primal_dual_gap} that $2\tau_n\mathbb{J}(x_n,w_n,\bar y)\geq0$ $\forall n$. Now, by Lemma \ref{Lemma4.1.1}, there exists a natural number $n_2$ such that, $\forall n\geq n_2$, we have 
   $$p_{n+1}(y)\leq p_n(y)- q_n,$$
   where 
\begin{equation}\label{Theorem4.1_1}
    \begin{aligned}
     p_{n}(y) & \coloneqq  \dfrac{\psi}{\psi-1}\lVert \bar x-z_{n+1} \rVert^2+\dfrac{1}{\beta}\lVert y-y_{n-1} \rVert^2+{\mu'}\lVert x_n-x_{n-1} \rVert^2 ,\\ 
    q_n & \coloneqq {\psi\bar\theta_n}\lVert x_n-z_{n+1} \rVert^2.
    \end{aligned}
\end{equation}
From Lemma \ref{usefullemma2}, we have $\lim_{n\to\infty} p_n(y)\in \mathbb{R}$ and $\lim_{n\to\infty} q_n =0$. Consequently $\lim_{n\to\infty}{\psi\bar\theta_n}\lVert x_n-z_{n+1} \rVert^2=0$. Furthermore, in Proposition \ref{tau_n_bdd_below}, following the fact that $\lim_{n\to\infty}\tau_n = \tau\geq\eta$, we obtain
\begin{equation}\label{Theorem4.1_2}
   \lim_{n\to\infty}\lVert x_n-z_{n+1} \rVert^2=\lim_{n\to\infty}\dfrac{1}{\psi^2}\lVert x_n-z_{n} \rVert^2=0.
\end{equation}
Moreover, using \eqref{Theorem4.1_2} and  the triangle inequality, we get
\begin{equation*}
    \lim_{n\to\infty}\lVert x_n-x_{n-1} \rVert^2 =0. 
\end{equation*}
Combining the fact that $\lim_{n\to\infty}p_n(y)$ is finite with (\ref{Theorem4.1_2}), we obtain that all the sequences $\{x_n\}$, $\{y_n\}$ and $\{z_n\}$ are bounded. Now, suppose that $(\Tilde x, \Tilde y)$ is a cluster point of $\{(x_n, y_n)\}$, and $\{(x_{n_{k}}, y_{n_{k}})\}$ is a subsequence of $\{(x_n, y_n)\}$ that converges to $(\Tilde x, \Tilde y)$, i.e., $\lim_{n\to\infty}x_{n_{k+1}}=\Tilde x$ and $\lim_{n\to\infty}y_{n_{k}}=\Tilde y$. Then, from equation (\ref{Theorem4.1_2}), we have $\lim_{n\to\infty}z_{n_{k+1}}=\Tilde x$. Furthermore, using~\eqref{eqn:39b}, \eqref{eq:w_n update}, \eqref{eq:y_n update}, \eqref{3.2.1}, \eqref{3.2.3}, Lemma~\ref{usefullemma3} and Remark~\ref{Remark_tau_n}, it follows that for all $(x,y)\in\mathbb{X}\times\mathbb{Y}$
\begin{equation*}
    \begin{aligned}
         \langle x_{n_{k}} - z_{n_{k}} + \tau_{n_{k}-1} K^* y_{n_{k}-1} + \tau_{n_{k}-1} \nabla h(x_{n_{k}-1}), \,x - x_{n_{k}} \rangle 
        &\geq\tau_{n_{k}-1} (f(x_{n_{k}}) - f(x)),\\
        \left\langle \frac{1}{\beta} (y_{n_{k}} - y_{n_{k}-1}) + \tau_{n_{k}} K x_{n_{k}}, \,y - y_{n_{k}-1}\right\rangle
        &\geq\tau_{n_{k}} (g^*(y_{n_{k}}) - g^*(y)).
    \end{aligned}
\end{equation*}
Recalling that both $f,g^*$ are lower-semi continuous and letting $k\to\infty$, we derive
 \begin{equation}\label{theorem4.1_5}
     \langle K^*\Tilde y+\nabla h(\Tilde x),\, x-\Tilde x \rangle\geq f(\Tilde x)-f(x)~~ \text{and}~~\langle K\Tilde x, \,y-\tilde y \rangle\geq g^*(\Tilde y)-g^*(y).
 \end{equation}
Both the inequalities in (\ref{theorem4.1_5}) imply that $(\Tilde{x},\Tilde{y})$ is a solution of (\ref{1.2}). Notice that Lemma~\ref{Lemma4.1.1} holds for any saddle point $(\bar x, \bar w, \bar y)\in\Omega$. Thus putting $\bar x= \Tilde{x}$ and $\bar y=\Tilde{y}$ in (\ref{Theorem4.1_1}, we obtain $\lim_{n\to\infty}p_{n_{k}}(\Tilde y)=0$. Furthermore, $\lim_{n\to\infty}p_{n}(\Tilde y)$ exists, and hence $\lim_{n\to\infty}p_{n}(\Tilde y)=0$. So, we have $\lim_{n\to\infty}z_n=\Tilde{x}$ and $\lim_{n\to\infty}y_n=\Tilde{y}$. Again, by observing (\ref{Theorem4.1_2}), we have $\lim_{n\to\infty}x_n=\Tilde{x}$. This completes the proof.
\end{proof}
\begin{remark}
    Since $(\tau_n)$ is a non-increasing sequence, if $ \tau_{0}\leq\min\left\{\dfrac{\mu}{\sqrt{\beta}\|K\|},\dfrac{\mu'}{\bar L}\right\}$, then $\tau_n=\tau_0$, $\forall n$, i.e., $(\tau_n)$ is constant. Let $\tau_n=\tau_0=\tau$. Then for $ \psi\in(1,\phi]$, we have $\tau\sigma\lVert K \rVert^2+2\tau \bar L=\beta\tau^2\lVert K \rVert^2+2\tau\bar L\leq\mu^2+2\mu'<\frac{\psi^2}{4}+\frac{\psi}{2}<\psi.$
    Thus, we get the stepsize condition required for the global convergence of extended GRPDA \cite{zhou2022new} with fixed stepsize.
\end{remark}
\subsection{Extended convergence region of P-GRPDA}
In \cite[Theorem 2.1]{chang2022grpdarevisited}, the authors showed that the upper bound of the region of convergence for GRPDA \eqref{GRPDA} can be extended from the Golden Ratio $\phi$ to $1+\sqrt{3}$. Furthermore, they established the convergence of GRPDA under the following stepsize condition: $\tau\sigma\|K\|^2 < \frac{\psi(2+2\psi-\psi^2)}{\psi+1}$, where $\psi \in (1,1+\sqrt{3})$. This stepsize is more relaxed compared to those required for both GRPDA \eqref{GRPDA} and the Chambolle–Pock primal-dual algorithm \eqref{1.2.2}. In the following theorem, we also extend the convergence region of Algorithm \ref{algorithm 1} from $(1, \phi]$ to $(1, 1+\sqrt{3})$.
\vspace{8pt}
\begin{theorem}\label{larger_mu-and_mu'}
     Under Assumptions \ref{assumption1}, \ref{assumption_2}, and \ref{assumption4}, let the sequence $\{(z_n,x_n,w_n,y_n,\tau_n)\}$ be generated by Algorithm \ref{algorithm 1}, where the parameters $\mu,\mu'$ satisfy 
     \begin{equation}\label{extension_on_mu}
         0<3\mu'<\mu<\frac{\psi}{2}+\frac{\psi(1+\psi-\psi^2)}{2(\psi+1)},~~\psi\in(1,1+\sqrt{3}).
     \end{equation}
Then, $\{(x_n,y_n)\}$ converges to a solution of (\ref{1.2}).
\end{theorem}
\begin{proof}
  To prove this theorem, we proceed with the calculations as outlined in Lemma~\ref{Lemma4.1.1} up to the inequality \eqref{lemma4.1_5}. From this point, we claim the following inequality holds
      \begin{equation}\label{psi_in_(1,phi]}
        \frac{\psi\bar\theta_n(1+\psi-\psi^2\bar\theta_n)}{\psi+1}\|x_{n+1}-x_n\|^2 \leq \left(1 +\frac{1}{\psi}-\psi\bar\theta_n\right)\|x_{n+1}-z_{n+1}\|^2 + \psi\bar\theta_n \|x_n-z_{n+1}\|^2.
    \end{equation}
To show this, we distinguish two cases based on the value of $\psi$: (i) $\psi \in (1, \phi]$ and (ii) $\psi \in (\phi, 1 + \sqrt{3})$. In the first case, using the fact that $(\tau_n)$ is non-increasing, we have $1 +\frac{1}{\psi}-\psi\bar\theta_n\geq 1 +\frac{1}{\psi}-\psi\geq0, ~\forall \psi\in(1,\phi]$. Now, by taking $ p = 1 + \frac{1}{\psi} - \psi \bar{\theta}_n$, $q = \psi \bar{\theta}_n $, $m_1 = \|x_{n+1} - z_{n+1}\|$ and $ m_2 = \|x_n - z_{n+1}\|$ in Lemma \ref{usefullemma_3}, and by noting that $\|x_{n+1} - x_n\|^2 \leq (m_1 + m_2)^2$, we obtain \eqref{psi_in_(1,phi]}. Before proving the second case, notice that the extended bounds on $\mu,\mu'$ and $\psi$ in \eqref{extension_on_mu} do not change the result in Proposition~\ref{tau_n_bdd_below}, which indicates that $(\tau_n)$ is bounded below by $\eta$ and that $\lim_{n\to\infty}\tau_n\geq \eta>0$, $\forall \psi\in(1,1+\sqrt{3})$. Thus, for the second case $\psi\in(\phi, 1+\sqrt{3})$, notice that
\begin{equation*}
\lim_{n\to\infty}\frac{\psi\bar\theta_n(\psi^2\bar\theta_n-\psi-1)}{\psi+1} = \frac{\psi(\psi^2-\psi-1)}{\psi+1}>0.
\end{equation*}
Hence, there exists a natural number $k_3$ such that $\frac{\psi\bar\theta_n(\psi^2\bar\theta_n-\psi-1)}{\psi+1}>0, \forall n\geq k_3$. Again, by setting $p= \frac{\psi\bar\theta_n(\psi^2\bar\theta_n-\psi-1)}{\psi+1}, q = \psi\bar\theta_n$ with $m_1 = \|x_{n+1}-x_n\|$ and $m_2 = \|x_n - z_{n+1}\|$ in Lemma \ref{usefullemma_3}, and by noting the facts that $\|x_{n+1}-z_{n+1}\|^2\leq (m_1+m_2)^2$, and $\frac{pq}{p+q} = \frac{\psi^2\bar\theta_n-\psi-1}{\psi}$, we obtain 
\begin{align}\label{case_2}
    \frac{\psi^2\bar\theta_n-\psi-1}{\psi}\|x_{n+1} - z_{n+1}\|^2 \leq \frac{\psi\bar\theta_n(\psi^2\bar\theta_n-\psi-1)}{\psi+1}\|x_{n+1}-x_n\|^2 + \psi\bar\theta_n\|x_n-z_{n+1}\|^2.
\end{align}
Rearranging \eqref{case_2} gives us \eqref{psi_in_(1,phi]}. Therefore, $\forall n\geq k_3$, substituting inequality~\eqref{psi_in_(1,phi]} into~\eqref{lemma4.1_5}, and from \eqref{extension_on_mu}, observing that $0<\mu<1~~\forall\psi\in(1,1+\sqrt{3})$, we obtain
     \begin{multline}\label{combine_region_1}
        2\tau_n\mathbb{J}(x_n,w_n, y) + \frac{\psi}{\psi-1}\lVert \bar x-z_{n+2} \rVert^2 + \frac{1}{\beta}\lVert y-y_n \rVert^2 
        \leq \frac{\psi}{\psi-1}\lVert \bar x-z_{n+1} \rVert^2 + \frac{1}{\beta}\lVert y-y_{n-1} \rVert^2 \\ -\left(\frac{\psi\bar\theta_n(1+\psi-\psi^2\bar\theta_n)}{\psi+1} + \psi\bar\theta_n - \mu\frac{\tau_n^2}{\tau_{n+1}^2} - \mu'\right)\|x_{n+1}-x_n\|^2 + \mu'\lVert x_n-x_{n-1} \rVert^2.
    \end{multline}
Recalling that Proposition \ref{tau_n_bdd_below} holds for $\psi, \mu, \mu'>0$, and hence, we deduce $\lim_{n\to\infty}\bar\theta_n =1$. Thus, from \eqref{extension_on_mu} and Proposition \ref{tau_n_bdd_below}, we have

\begin{equation*}
\begin{aligned}
\lim_{n\to\infty}\left(\frac{\psi\bar\theta_n(1+\psi-\psi^2\bar\theta_n)}{\psi+1} + \psi\bar\theta_n - \mu\frac{\tau_n^2}{\tau_{n+1}^2} - \mu'\right)& = \frac{\psi(1+\psi-\psi^2)}{\psi+1} + \psi- \mu - \mu'\\ &
    >2\mu-\mu-\mu'\\&
    = \mu-\mu'>2\mu'~~\text{as $3\mu'<\mu$}.
\end{aligned}
\end{equation*}
Therefore, there exists a natural number $k_4$ such that,  $\forall n\geq k_4$
\begin{equation}\label{combine_region_2}
    \frac{\psi\bar\theta_n(1+\psi-\psi^2\bar\theta_n)}{\psi+1} + \psi\bar\theta_n - \mu\frac{\tau_n^2}{\tau_{n+1}^2} - \mu'>2\mu'.
\end{equation}
Let $k_5 = \max\{k_3,k_4\}$. Then, $\forall n\geq k_5$, combining \eqref{combine_region_1} and \eqref{combine_region_2}, we get
\begin{multline}\label{combine_region_3}
        2\tau_n\mathbb{J}(x_n,w_n, y) + \frac{\psi}{\psi-1}\lVert \bar x-z_{n+2} \rVert^2 + \frac{1}{\beta}\lVert y-y_n \rVert^2 + \mu'\|x_{n+1}-x_n\|^2\\
        \leq \frac{\psi}{\psi-1}\lVert \bar x-z_{n+1} \rVert^2 + \frac{1}{\beta}\lVert y-y_{n-1} \rVert^2  + \mu'\lVert x_n-x_{n-1} \rVert^2 - \mu'\|x_n-x_{n+1}\|^2.
    \end{multline}
Since $(\bar x, \bar w, \bar y)$ is a saddle point of $\mathscr{L}$, it follows from \eqref{primal_dual_gap} that $2\tau_n\mathbb{J}(x_n,w_n,\bar y)\geq0$ $\forall n$. Furthermore, inequality (\ref{combine_region_3}) can be written as $p_{n+1}(y)\leq p_n(y)- q_n$, $\forall n\geq k_5$, where 
\begin{equation}\label{p_n_q_n}
    \begin{aligned}
    & p_{n}(y) \coloneqq \dfrac{\psi}{\psi-1}\lVert \bar x-z_{n+1} \rVert^2+\dfrac{1}{\beta}\lVert y-y_{n-1} \rVert^2+{\mu'}\lVert x_n-x_{n-1} \rVert^2 ,\\ &
    q_n \coloneqq \mu'\lVert x_n-x_{n+1} \rVert^2.
    \end{aligned}
\end{equation}
From \eqref{combine_region_3}, \eqref{p_n_q_n}, and using Lemma \ref{usefullemma2}, we obtain $\sum_{n=1}^{\infty}\|x_n - x_{n-1}\|^2< \infty$. To complete the proof, it remains to show that $\lim_{n \to \infty} \|x_n - z_n\| = 0$. Once this is established, the rest of the proof follows from Theorem \ref{thm_3.1.1}. By using \eqref{eqn:39a}, we have
\begin{equation}\label{combine_region_4}
\begin{aligned}
     \|x_n-z_n\| & \leq \frac{\psi-1}{\psi}\|x_n-x_{n-1}\| +\frac{1}{\psi}\|x_n-z_{n-1}\| \\ &
     \leq \frac{\psi-1}{\psi}\|x_n-x_{n-1}\| +\frac{1}{\psi}\|x_n-x_{n-1}\|+\frac{1}{\psi}\|x_{n-1}-z_{n-1}\| \\ &
     = \frac{1}{\psi}\|x_{n-1}-z_{n-1}\| + \|x_n-x_{n-1}\|.
\end{aligned}
\end{equation}
Again, using the facts $\sum_{n=1}^{\infty}\|x_n - x_{n-1}\|^2 < \infty$ and $\frac{1}{\psi}<1~~\forall\psi\in(1, 1+\sqrt{3})$, it follows from \eqref{combine_region_4} that $\sum_{n=1}^{\infty}\|x_n - z_n\|^2 < \infty$. Consequently, we have $\lim_{n \to \infty} \|x_n - z_n\| = 0$. Hence, the proof is concluded.
\end{proof}
\section{Convergence rate}\label{sec4}
In this section, we establish the ergodic sublinear convergence rate results of Algorithm~\ref{algorithm 1}, measured by the function value residual and constraint violation. To facilitate this, for $N\geq 1$, we define 
\begin{equation}\label{x_ntilde and w_ntilde}
    \Tilde{x}_N\coloneqq \dfrac{1}{N}\sum_{n=1}^{N} x_n,~~~ \Tilde{w}_N\coloneqq\dfrac{1}{N}\sum_{n=1}^{N} w_n.
\end{equation}
\begin{theorem}\label{thm_4.0.1}(\textbf{Sublinear rate of convergence})
    Under Assumptions \ref{assumption1}, \ref{assumption_2} and~\ref{assumption4}, let $\{(z_n,x_n,w_n,y_n,\tau_n)\}$ be the sequence generated by Algorithm \ref{algorithm 1}, and suppose that $(\bar x, \bar w,\bar y)\in\Omega$. Then, there exists a natural number $n_2$ and a constant $P_1>0$ such that the following holds
    \begin{equation*}
       \left | \Phi(\Tilde{x}_N,\, \Tilde{w}_N)- \Phi(\bar x,\, \bar w)\right| \leq \frac{P_1}{N}\hspace{.17cm}and \hspace{.17cm}\left | K\Tilde{x}_N- \Tilde{w}_N \right|\leq \frac{2P_1}{bN},
    \end{equation*}
where $N\geq 1$ and $b$ is a positive constant satisfying $b\geq 2\|\bar y\|$.
\end{theorem}
\begin{proof}
     For any $y\in\mathbb{Y}$, It follows from \eqref{Theorem4.1_1}, and \eqref{lemma4.1.1.1}
    \begin{equation*}\label{Roc1}
        \begin{aligned}
              2\tau_n\mathbb{J}(x_n,w_n,y)\leq p_{n}(y)-p_{n+1}(y),\hspace{.17cm} \forall n\geq n_2.
        \end{aligned}
    \end{equation*}
  From Proposition \ref{tau_n_bdd_below}, we know that $\tau_n\geq\eta=\min\left\{ \tau_0, \frac{\mu}{\sqrt{\beta}\|K\|}, \frac{\mu'}{\bar{L}} \right\}~\forall n$. Hence
    \begin{equation*}
    2\eta\mathbb{J}(x_n,w_n,y)\leq p_{n}(y)-p_{n+1}(y),\hspace{.17cm} \forall n\geq n_2.
    \end{equation*}
By taking summation over $n=n_2,n_2+1,\ldots,n_2+N-1$, we obtain
\begin{equation*}
\begin{aligned}
2\eta\sum_{n=n_2}^{n_2+N-1}\mathbb{J}(x_n,w_n,y)
&\leq p_{n_2}(y)-p_{n_2+N}(y)\\
&\leq p_{n_2}(y)\\
&= \frac{\psi}{\psi-1}\,\lVert \bar x-z_{n_2+1}\rVert^2
  + \frac{1}{\beta}\,\lVert y-y_{n_2-1}\rVert^2
  + \mu'\,\lVert x_{n_2}-x_{n_2-1}\rVert^2 .
\end{aligned}
\end{equation*}
Now, using the definition of $\Tilde{x}_N,\Tilde{w}_N$ in \eqref{x_ntilde and w_ntilde} and the joint convexity of $\mathbb{J}(x_n,w_n,y)$ in $(x,w)$ for any $y$, we have
 \begin{equation}\label{Roc_3}
 \begin{aligned}
     \mathbb{J}(\Tilde{x}_N,\Tilde{w}_N,y)&=  \Phi(\Tilde{x}_N,\, \Tilde{w}_N)+ \langle y, K\Tilde{x}_N-\Tilde{w}_N \rangle - \Phi(\bar x,\, \bar w) \\ 
     &\leq \dfrac{1}{N}\sum_{n=n_2}^{n_2+N-1} \mathbb{J}(x_n,w_n,y)\\
     &\leq \dfrac{p_{n_2}(y)}{2\eta N}.
 \end{aligned}
 \end{equation}
Define $P_3\coloneqq\dfrac{\psi}{\psi-1}\lVert \bar x-z_{n_2+1} \rVert^2+\dfrac{1}{\beta} (b+\lVert y_{n_2-1} \rVert)^2+ {\mu'}\lVert x_{n_2}-x_{n_2-1}\rVert^2$. Clearly, we can observe that $P_3$ is an upper bound of $p_n(y)$. Hence, by taking maximum on both sides of~(\ref{Roc_3}) over $\lVert y
\rVert\leq b$ yields
\begin{equation}\label{Roc_4}
    \Phi(\Tilde{x}_N,\, \Tilde{w}_N)+ b\lVert K\Tilde{x}_N-\Tilde{w}_N \rVert- \Phi(\bar x,\, \bar w)\leq\dfrac{P_1}{N},
\end{equation}
where $P_1=P_3/2\eta$. Since $(\bar x,\bar w,\bar y)$ is a saddle point of $\mathscr{L}$, we have $\mathscr{L}(\bar x,\bar w,\bar y)\leq \mathscr{L}(x_N,w_N,\bar y)$, Now, using $K\bar x=\bar w$ and $\lVert \bar y\rVert\leq b/2$, we obtain
\begin{equation}\label{Roc_5}
    \Phi(\bar x,\, \bar w)-\Phi(\Tilde{x}_N,\,\Tilde{w}_N)\leq\langle \bar y,\, K\Tilde{x}_N-\Tilde{w}_N \rangle\leq \dfrac{b}{2}\lVert K\Tilde{x}_N-\Tilde{w}_N \rVert.
\end{equation}
Altogether, (\ref{Roc_4}) with (\ref{Roc_5}) implies 
 \begin{equation*}
 \begin{aligned}
     & b\lVert K\Tilde{x}_N-\Tilde{w}_N \rVert\leq \dfrac{b}{2}\lVert K\Tilde{x}_N-\Tilde{w}_N \rVert+ \dfrac{P_1}{N}.   
 \end{aligned}
 \end{equation*}
 Thus $\lVert K\Tilde{x}_N-\Tilde{w}_N \rVert\leq \dfrac{2P_1}{bN}$. Using this estimate in (\ref{Roc_4}) together with (\ref{Roc_5}), we obtain
 \begin{equation*}
     \left | \Phi(\Tilde{x}_N,\, \Tilde{w}_N)- \Phi(\bar x,\,\bar w)\right| \leq \frac{P_1}{N}.
 \end{equation*}
\end{proof}
\subsection{Nonergodic rate of convergence} In this section, we prove the R-Linear rate of convergence for Algorithm \ref{algorithm 1} using the primal-dual gap function. The primal-dual gap function has been widely used in literature; see \cite{malitsky2018first,chambolle2016ergodic,chambolle2011first,vladarean2021first,chang2021goldengrpda,chang2022goldenlinesearch}. First, we define the primal-dual gap function, and then we recall the definition of Linear convergence of a sequence.

Under Assumption \ref{assumption1} and from \cite{Rockafellar+1970}, it follows that, $\hat x\in\mathbb{X}$ is a solution of the primal problems (\ref{1.1}) if and only if there exists $\hat y\in\mathbb{Y}$ such that $(\hat x , \hat y)$ is a saddle point of (\ref{1.2}), that is
\begin{equation}\label{gap_2}
    \mathbb{L}(\hat x, y)\leq \mathbb{L}(\hat x, \hat y)\leq\mathbb{L}(x, \hat y)~~~\forall (x,y)\in\mathbb{X}\times\mathbb{Y},
\end{equation}
where $\mathbb{L}$ is defined in \eqref{1.2}. Here, such $(\hat x, \hat y)$ is the solution of the min-max or saddle point problem (\ref{1.2}), and  $\hat y$ is the solution of the dual problem (\ref{1.3}). Notice that, \eqref{gap_2} can be equivalently written as
\begin{equation*}
\begin{cases}
\displaystyle
\mathcal{P}_{(\hat x, \hat y)}(x)
= f(x)-f(\hat x)
  + \bigl\langle x-\hat x,\;K^*\hat y + \nabla h(\hat x)\bigr\rangle\geq 0
  & \forall x\in\mathbb{X},\\[1ex]
\displaystyle
\mathcal{D}_{(\hat x, \hat y)}(y)
= g^*(y)-g^*(\hat y)
  - \bigl\langle K\hat x,\;y-\hat y\bigr\rangle\geq 0
  ~~\forall y\in\mathbb{Y}.
\end{cases}
\end{equation*}
Given $(\hat x, \hat y)\in\mathbb{X}\times\mathbb{Y}$, the \emph{primal-dual gap} function is defined as
\begin{equation}\label{primaldualgapfunction}
    \mathcal{G}_{(\hat x, \hat y)}(x,y)=\mathcal{P}_{(\hat x, \hat y)}(x)+\mathcal{D}_{(\hat x, \hat y)}(y)\geq 0, ~\forall(x,y)\in\mathbb{X}\times\mathbb{Y}.
\end{equation}
Let us now recall the definition of Linear convergence of a sequence $(v_n)\in\mathbb{X}$. We say the sequence $(v_n)$ converges Q-Linearly to $v\in\mathbb{X}$ if there is $q\in(0,1)$ and a natural number $k$ such that $\|v_{n+1}-v\|\leq q \|v_n-v\|$ for all $n\geq k$. And $(v_n)$ converges R-Linearly to $v$ if there exists a sequence $\epsilon_n$ and a natural number $k_1$ and $M>0$ such that $\|v_n-v\|\leq M\epsilon_n$ for all $n\geq k_1$, and $\epsilon_n$ converges Q-Linearly to 0.

Given a differentiable function $h$, we say that $h$ is strongly convex if there exists a constant $\gamma_h>0$ such that
\begin{equation}\label{strng_cnvx_h}
    h(y)-h(z)\geq \langle \nabla h(z), y-z\rangle+ \frac{\gamma_h}{2}\lVert y-z\rVert^2~\forall y,z\in\mathbb{X}.
\end{equation}
Furthermore, for a nonsmooth function $g^*$, we call $g^*$ to be strongly convex if there exists a constant $\gamma_{g^*}>0$ such that
\begin{equation}\label{strng_cnvx_g^*}
       g^*(y)-g^*(z) \geq \langle u, y-z \rangle + \frac{\gamma_{g^*}}{2}\|y-z\|^2,~ \forall y,z\in\mathbb{Y},~ \forall u\in\partial g^*(z).
   \end{equation}
\begin{assumption}\label{assump_4.1}
   Suppose that $g^*, h$ are strongly convex functions with constants $\gamma_{g^*}, \gamma_h>0$, respectively.
\end{assumption}
\vspace{8pt}
\begin{theorem}\label{theorem5.1.1}
   Under Assumptions \ref{assumption1}, \ref{assumption_2} and \ref{assump_4.1}, let $\{(x_n,y_n)\}$ be the sequences generated by Algorithm \ref{algorithm 1}. Suppose that $\{(x_n,y_n)\}$ converges to the unique primal-dual solution $(\bar x, \bar y)$. Then, there exist constants $V_1,  V_2,  Z >0$, a scalar $\zeta\in(0,1)$, and a natural number $n_4$ such that 
 \begin{equation*}
    \begin{aligned}
        &\lVert \bar{x} - z_{n+2} \rVert^2 \leq Z\zeta^n , \quad \lVert \bar{y} - y_n \rVert^2 \leq V_1 \zeta^n, \quad \lVert x_n - x_{n+1} \rVert^2 \leq V_2 \zeta^n, \quad \forall n \geq n_4,
    \end{aligned}
\end{equation*}
and thus $\{(x_n,y_n)\}$ converges R-linearly to zero. 

\end{theorem}
\begin{proof} 
From the dual update \eqref{eq:w_n update}, \eqref{eq:y_n update} in Algorithm \ref{algorithm 1}, and Remark \ref{Remark_tau_n},  we have
\begin{equation}\label{dual_optimality_condition}
y_{n-1}-y_n+\sigma_nKx_n\in\sigma_n\partial g^*(y_n).
\end{equation}
Using \eqref{dual_optimality_condition} along with the definition of strong convexity of $g^*$ in \eqref{strng_cnvx_g^*}, and the fact $\sigma_n=\beta\tau_n$, we obtain
\begin{multline}\label{strng_g^*}
   \tau_n(g^*(y_n) - g^*(\bar y))\leq\left\langle \frac{1}{\beta}(y_n-y_{n-1})-\tau_nKx_n,\, \bar y-y_n\right\rangle - \frac{\tau_n\gamma_{g^*}}{2}\|\bar y - y_n\|^2.
\end{multline}
Again, by elementary calculations similar to those in Lemma~\ref{Lemma4.1.1}, together with~\eqref{strng_g^*} and the definition of the primal--dual gap function in~\eqref{primaldualgapfunction}, we obtain
    \begin{multline*}
        \tau_n\mathcal{G}_{(\bar x, \bar y)}(x_n,y_n)\leq 
       \langle x_{n+1}-z_{n+1}, \, \bar x-x_{n+1}\rangle 
       + \psi\bar\theta_n\langle x_n-z_{n+1},\, x_{n+1}-x_n\rangle\\  +\dfrac{1}{\beta}\langle y_n-y_{n-1}, \,\bar y-y_n\rangle + \tau_n\langle K(x_n-x_{n+1}), \,y_n-y_{n-1} \rangle
     + \tau_{n}\langle \nabla h(x_{n-1}), \,x_{n+1}-x_n \rangle\\
     + \tau_n\big(h(x_n)-h(\bar x)+ \langle\nabla h(x_n), \,\bar x-x_{n+1}\rangle\big) - \frac{\tau_n\gamma_{g^*}}{2}\|\bar y - y_n\|^2.
    \end{multline*}
Furthermore, using (\ref{strng_cnvx_h}), we have
\begin{multline*}
     \tau_n\mathcal{G}_{(\bar x, \bar y)}(x_n,y_n)\leq  \langle x_{n+1}-z_{n+1}, \, \bar x -x_{n+1}\rangle  + \psi\bar\theta_n\langle  x_n-z_{n+1},\, x_{n+1}-x_n\rangle \\ 
      +\dfrac{1}{\beta}\langle y_n-y_{n-1},\, \bar y-y_n\rangle+
        \tau_{n}\langle K(x_n-x_{n+1}),\, y_n-y_{n-1} \rangle \\
        + \tau_n\langle \nabla h(x_n)-\nabla h(x_{n-1}),\, x_n-x_{n+1}\rangle 
        -\frac{\tau_n\gamma_{g^*}}{2}\|\bar y - y_n\|^2-\frac{\gamma_h\tau_n}{2}\|\bar x - x_n\|^2.
\end{multline*}  
Doing similar calculations as in Lemma \ref{Lemma4.1.1}, and following \eqref{psi_ineq}, \eqref{lemma4.1_6}, there exists a natural number $n_2$ such that $\forall n\geq n_2$
\begin{multline}\label{Roceqn1}
     2\tau_n\mathcal{G}_{(\bar x, \bar y)}(x_n,y_n)+ \dfrac{\psi}{\psi-1}\lVert \bar x-z_{n+2} \rVert^2 + \dfrac{1}{\beta}\lVert \bar y-y_{n} \rVert^2 + (\mu-\mu')\lVert x_n-x_{n+1} \rVert^2 \\ \leq \dfrac{\psi}{\psi-1}\lVert \bar x-z_{n+1} \rVert^2 +\dfrac{1}{\beta}\lVert \bar y-y_{n-1} \rVert^2+{\mu'}\lVert x_n-x_{n-1} \rVert^2 - {\psi\bar\theta_n}\lVert x_n-z_{n+1} \rVert^2 \\ 
     -{\gamma_h\tau_n}\lVert \bar x-x_n\lVert^2-{\tau_n\gamma_{g^*}}\lVert \bar y-y_n\lVert^2.
\end{multline}
Since $(\bar x, \bar y)$ is a primal-dual solution of (\ref{1.2}), it follows from (\ref{primaldualgapfunction}) that $\tau_n\mathcal{G}_{(\bar x, \bar y)}(x_n,y_n)\geq 0$ for all $n$. Again, from the facts that $(\tau_n)$ converges to a positive constant and $4\mu'<\psi$, we have
\vspace{3pt}
\begin{equation*}
    \lim_{n\to\infty}\psi\bar\theta_n=\psi>4\mu'.
\end{equation*}
\vspace{3pt}
Thus, there exists a natural number $n_3$ such that 
\begin{equation*}
    \psi\bar\theta_n>4\mu',~\forall n\geq n_3.
\end{equation*}
Consequently, taking $n_4=\max\{n_2,n_3\}$, from (\ref{Roceqn1}), $\forall n\geq n_4$, we obtain
\begin{multline}\label{Roceqn2}
       \dfrac{\psi}{\psi-1}\lVert \bar x-z_{n+2} \rVert^2 + \dfrac{1}{\beta}\lVert \bar y-y_{n} \rVert^2 + (\mu-\mu')\lVert x_n-x_{n+1} \rVert^2 \leq \dfrac{\psi}{\psi-1}\lVert \bar x-z_{n+1} \rVert^2\\
       +\dfrac{1}{\beta}\lVert \bar y-y_{n-1} \rVert^2 +{\mu'}\lVert x_n-x_{n-1} \rVert^2
        - 4\mu'\lVert x_n-z_{n+1} \rVert^2 \\
          -{\gamma_h\tau_n}\lVert \bar x-x_n\lVert^2-{\tau_n\gamma_{g^*}}\lVert \bar y-y_n\lVert^2.
\end{multline}
The strong convexity of $h$ and Lipschitz continuity of $\nabla h$ imply that $0<\gamma_h\leq \bar L$. Consequently, $0<\eta\gamma_h\leq \bar L\dfrac{\mu'}{\bar L}=\mu'$, where $\eta = \min\left\{ \tau_0, \frac{\mu}{\sqrt{\beta}\|K\|}, \frac{\mu'}{\bar{L}} \right\}$, as defined in Proposition \ref{tau_n_bdd_below}.
By applying this fact, we obtain
\begin{equation}\label{Roceqn4}
    \begin{aligned}
        \gamma_h\tau_n\lVert \bar x-x_n\lVert^2+4\mu'\lVert x_n-z_{n+1} \rVert^2&\geq\eta\gamma_h\lVert \bar x-x_n\lVert^2 + \mu'\lVert x_n-z_{n+1} \rVert^2\\ &
        \geq \eta\gamma_h\big(\lVert \bar x-x_n\lVert^2+\lVert x_n-z_{n+1} \rVert^2\big)\\&
        \geq\eta\gamma_h\lVert \bar x-z_{n+1}\lVert^2.
    \end{aligned}
\end{equation}
Combining (\ref{Roceqn2}) and (\ref{Roceqn4}), we have
\begin{equation*}
    \begin{aligned}
         &\dfrac{\psi}{\psi-1}\lVert \bar x-z_{n+2} \rVert^2 + \left(\dfrac{1}{\beta}+ \tau_n\gamma_{g^*}\right)\lVert \bar y-y_n\rVert^2+ (\mu-\mu')\lVert x_n-x_{n+1} \rVert^2\\ &
         \leq\Big(\dfrac{\psi}{\psi-1}-\eta\gamma_h\Big)\lVert \bar x-z_{n+1} \rVert^2+\dfrac{1}{\beta}\lVert \bar y-y_{n-1} \rVert^2+ \mu'\lVert x_n-x_{n-1} \rVert^2 . 
    \end{aligned}
\end{equation*}
Rearranging all the terms and remembering $\tau_n\geq\eta~~\forall n$, we obtain
\begin{equation}\label{Roceqn6}
    \begin{aligned}
        &\lVert \bar x-z_{n+2} \rVert^2 + \dfrac{\psi-1}{\psi}\left(\dfrac{1}{\beta}+\eta\gamma_{g^*}\right)\lVert \bar y-y_{n} \rVert^2 + \dfrac{\psi-1}{\psi}(\mu-\mu')\lVert x_n-x_{n+1} \rVert^2\\ &
         \leq\Big(1-\dfrac{\psi-1}{\psi}\eta\gamma_h\Big)\lVert \bar x-z_{n+1} \rVert^2+\dfrac{1}{\beta}\dfrac{\psi-1}{\psi}\lVert \bar y-y_{n-1} \rVert^2 + \mu'\dfrac{\psi-1}{\psi}\lVert x_n-x_{n-1} \rVert^2 .
    \end{aligned}
\end{equation}
Now, take $\zeta =\max\left\{1-\dfrac{\psi-1}{\psi}\eta\gamma_h,~~\dfrac{\mu'}{\mu-\mu'}, ~~\dfrac{1}{1+\beta\eta\gamma_{g^*}}\right\}$, then $\zeta\in(0,1)$ and it follows from (\ref{Roceqn6}) that,$~\forall n\geq n_4$
\begin{multline}\label{Roceqn7}
    \lVert \bar{x}-z_{n+2} \rVert^2 + \frac{\psi-1}{\psi}\left(\frac{1}{\beta}+\eta\gamma_{g^*}\right)\lVert \bar{y}-y_{n} \rVert^2 + \dfrac{\psi-1}{\psi}(\mu-\mu')\lVert x_n-x_{n+1} \rVert^2\\ 
\leq \zeta\left(\lVert \bar{x}-z_{n+1} \rVert^2 + \frac{\psi-1}{\psi}\left(\frac{1}{\beta}+\eta\gamma_{g^*}\right)\lVert \bar{y}-y_{n-1} \rVert^2+\frac{\psi-1}{\psi}(\mu-\mu')\lVert x_n-x_{n-1} \rVert^2\right).
\end{multline}
Iterating (\ref{Roceqn7}) yields
\begin{multline*}
     \lVert \bar{x}-z_{n+2} \rVert^2 + \dfrac{\psi-1}{\psi}\left(\dfrac{1}{\beta}+\eta\gamma_{g^*}\right)\lVert \bar{y}-y_{n} \rVert^2 + \dfrac{\psi-1}{\psi}(\mu-\mu')\lVert x_n-x_{n+1} \rVert^2\\ 
        \leq \zeta^{n-n_4+1}M,~\forall n\geq n_4,
\end{multline*}
where
\begin{multline*}
            M=\lVert \bar{x}-z_{n_4+1} \rVert^2 + \dfrac{\psi-1}{\psi}\left(\dfrac{1}{\beta}+\eta\gamma_{g^*}\right)\lVert \bar{y}-y_{{n_4}-1} \rVert^2\\ + \dfrac{\psi-1}{\psi}(\mu-\mu')\lVert x_{{n_4}}-x_{n_4-1} \rVert^2.
\end{multline*}
Moreover, setting $Z=\dfrac{M}{\zeta^{n_4-1}}$, we obtain
 \begin{multline}\label{Roceqn9}
        \lVert \bar{x}-z_{n+2} \rVert^2 + \dfrac{\psi-1}{\psi}\left(\dfrac{1}{\beta}+\eta\gamma_{g^*}\right)\lVert \bar{y}-y_{n} \rVert^2 + \dfrac{\psi-1}{\psi}(\mu-\mu')\lVert x_n-x_{n+1} \rVert^2\\
        \leq\zeta^n Z ~\forall n\geq n_4.
\end{multline}
Therefore, (\ref{Roceqn9}) implies that $\lVert \bar{x}-z_{n+1} \rVert^2 + \dfrac{\psi-1}{\psi}\left(\dfrac{1}{\beta}+\eta\gamma_{g^*}\right)\lVert \bar{y}-y_{n-1} \rVert^2 + \dfrac{\psi-1}{\psi}(\mu-\mu')\lVert x_n-x_{n-1} \rVert^2$ converges R-Linearly to zero. Moreover, from (\ref{Roceqn9}), $\forall n\geq n_4$, we have  
\begin{equation*}
        \lVert \bar{x}-z_{n+2} \rVert^2\ \leq Z\zeta^n, \quad \lVert \bar{y}-y_{n} \rVert^2\leq V_1\zeta^n,~\text{where}~ V_1=\dfrac{\beta \psi Z}{(\psi-1)(1+\beta\eta\gamma_{g^*})}.
\end{equation*}
Since $\psi>1$, it follows that $V_1>0$. From (\ref{Roceqn9}), we obtain
\begin{equation*}
    \lVert x_n-x_{n+1} \rVert^2\leq V_2\zeta^n,~\text{where}~V_2=\dfrac{\psi Z}{(\psi-1)(\mu-\mu')}.
\end{equation*}
As $\mu-\mu'>\mu>0$, we have $V_2>0$ and thus $\lVert x_n-x_{n+1} \rVert\leq \sqrt{V_2}\zeta^{n/2}$. Now, for $m>n\geq n_4$, by applying the triangle inequality, we have
\begin{equation}\label{cauchy_seqn}
\begin{aligned}
     \|x_n-x_m\|& \leq \|x_n-x_{n+1}\| + \|x_{n+1}-x_{n+2}\| + \ldots + \|x_{m-1}-x_m\|\\ &
     \leq \sqrt{V_2}\left(\zeta^{n/2} + \zeta^{n+1/2} + \ldots + \zeta^{m-1/2}\right) \\ &
     \leq \sqrt{V_2}\zeta^{n/2}\left(\zeta + \zeta^{1/2} + \ldots + \zeta^{m-n-1/2}\right) \\ &
     = \sqrt{V_2}\zeta^{n/2}\frac{1-\zeta^{m-n/2}}{1-\zeta^{1/2}}.
\end{aligned}
\end{equation}
Since $\zeta\in(0,1)$, it follows from (\ref{cauchy_seqn}) that $ \|x_n-x_m\|\leq\frac{\sqrt{V_2}\zeta^{n/2}}{1-\zeta^{1/2}}$, $\forall m, n\geq n_4$. This proves that $(x_n)$ is a Cauchy sequence, and hence $(x_n)$ converges to $\bar x$. Finally, letting $m\to\infty$, we obtain $\|x_n-\bar x\|\leq\frac{\sqrt{V_2}\zeta^{n/2}}{1-\zeta^{1/2}}$ for all $n\geq n_4$. This completes the proof.
\end{proof}

\section{The Adaptive extended Golden Ratio primal-dual algorithm}\label{sec5}
In this section, we propose an adaptive version of E-GRPDA, namely \emph{aEGRPDA}, which does not rely on linesearch procedures (like Algorithm~\ref{algorithm 1}) and exploits local curvature information of $\nabla h$ to establish global convergence of the iterates. The key differences between Algorithm~\ref{algorithm 2} and Algorithm~\ref{algorithm 1} are that the stepsizes are allowed to increase from time to time and that $h$ is only required to be locally smooth. It was shown in~\cite{malitsky2020golden} that aGRAAL for monotone variational inequalities converges at an $\mathcal{O}(1/N)$ rate, and we derive a similar result in this section. Next, for $n\ge 1$, we define

\vspace{5pt}
\begin{equation*}
    \bar L_n=\frac{\lVert\nabla h(x_n)-\nabla h(x_{n-1})\rVert}{\lVert x_n-x_{n-1}\rVert}~~\text{when}~~x_n-x_{n-1}\neq 0.
\end{equation*}

\begin{algorithm}[H]
\caption{The aEGRPDA for (\ref{E-GRPDA})}
\label{algorithm 2}
\begin{algorithmic}[1]
    \State \textbf{Initialization:} Let $x_0 \in \mathbb{X}$, $y_0 \in \mathbb{Y}$, and set $z_0 = x_0$. Choose $\beta > 0$, $\psi \in (1, \phi]$, $\rho = \psi^{-1}+\psi^{-2}$, $\theta_0 > 0$ and $\tau_{\max}\gg0$. Suppose $\tau_0 > 0$ and $n = 1$.  
    \State \textbf{Step 1:} \emph{Compute}:
        \begin{align*}
            z_n &= \dfrac{\psi-1}{\psi}x_{n-1} + \dfrac{1}{\psi}z_{n-1},\\
            x_{n} &= \operatorname{prox}_{\tau_{n-1} f}(z_n-\tau_{n-1} K^*y_{n-1}-\tau_{n-1}\nabla h(x_{n-1})).
            \end{align*}
 \State \textbf{Step 2:} \emph{Update}:
    \begin{equation}
        \begin{aligned}
            \tau_n & = \min\left\{\rho\tau_{n-1},~\dfrac{\psi\theta_{n-1}}{9({\bar L_n}^2+\beta\psi\|K\|^2)}\dfrac{1}{\tau_{n-1}},~\tau_{\max}\right\}\label{eq:alg2_c},\\ 
            \sigma_n  &= \beta \tau_n.
        \end{aligned}
    \end{equation}
    \State \textbf{Step 3:} \emph{Compute}:    
        \begin{align*}
          w_n &= \operatorname{prox}_{\frac{1}{\sigma_n}g}\Bigl(\frac{y_{n-1}}{\sigma_n} + Kx_n\Bigr),\\
          y_n &= y_{n-1} + \sigma_n\bigl(Kx_n - w_n\bigr).
        \end{align*}
   \State \textbf{Step 4:} \emph{Update}:~$\theta_n=\dfrac{\psi\tau_n}{\tau_{n-1}}$.
   \State \textbf{Step 5:} Let $n \leftarrow n+1$ and return to \textbf{Step 1}.
\end{algorithmic}
\end{algorithm}

Before we delve into the convergence analysis of Algorithm~\ref{algorithm 2}, we pause to make a few observations.
\vspace{8pt}
\begin{remark}
    In Algorithm \ref{algorithm 2}, $\tau_{\max}$ is chosen to be a large value. In practice, we take $\tau_{\max} = 10^7$ and $\rho = \frac{1}{\psi} +\frac{1}{\psi^2}$ for a given $\psi$. As in Algorithm \ref{algorithm 1}, we also adopt the convention $\frac{0}{0} = \infty$. Under this convention, when $x_n= x_{n-1}$, we calculate $\tau_n$ as
    $\tau_n = \min\{\rho\tau_{n-1}, \tau_{\max}\}$.
\end{remark}
\vspace{8pt}
\begin{remark}
    The E-GRPDA algorithm \eqref{E-GRPDA} converges under the stepsize condition $\tau\sigma\lVert K\rVert^2+2\tau\bar L<\psi$ for any $\psi\in(1,\phi]$, where $\sigma=\beta\tau$. Now, substituting the value of $\sigma$ back into the stepsize inequality gives
    \begin{equation*}
        \beta\tau^2\lVert K\rVert^2+2\tau\bar L<\psi\implies\beta\lVert K\rVert^2+\dfrac{2\bar L}{\tau}-\dfrac{\psi}{\tau^2}<0.
    \end{equation*}
    Solving this, we have $\tau\in\left(0,\dfrac{\psi}{\bar L+\sqrt{\bar L^2+\psi\beta\lVert K\rVert^2}}\right)$.
\end{remark}
\noindent
 We show that Algorithm \ref{algorithm 2} converges when this estimate is satisfied locally given by $\tau_n\in\left(0,\dfrac{\psi}{3\sqrt{\bar L_n^2+\psi\beta\|K\|^2}}\right)$ for any $\psi\in(1,\phi]$. In particular, we analyze the convergence of Algorithm \ref{algorithm 2} when $\tau_n\tau_{n-2}\leq\dfrac{\psi^2}{9(\bar L_n^2+\psi\beta \|K\|^2)}$. This inequality implies that
 \begin{equation*}
    \tau_n\leq\dfrac{\psi}{9(\bar L_n^2+\psi\beta\|K\|^2)}\dfrac{\psi\tau_{n-1}}{\tau_{n-2}}\dfrac{1}{\tau_{n-1}}=\dfrac{\psi\theta_{n-1}}{9(\bar L_n^2+\psi\beta\|K\|^2)}\dfrac{1}{\tau_{n-1}}~\forall n\geq 1,
 \end{equation*}
 where $\tau_0$, $\theta_0>0$ and $\theta_{n} = \dfrac{\psi\tau_{n}}{\tau_{n-1}}$.
\vspace{8pt}
\begin{remark}\label{remark_5.2}
From equation (\ref{eq:alg2_c}), notice that
    \begin{equation*}
        \tau_n\leq\dfrac{\theta_n\theta_{n-1}}{9\tau_n}\dfrac{1}{\bar L_n^2+\beta\psi\|K\|^2}~~\text{as}~~\dfrac{\psi\theta_{n-1}}{\tau_{n-1}}=\dfrac{\theta_n\theta_{n-1}}{\tau_n}.     
    \end{equation*}
    Therefore, it follows that $\tau_n^2\bar L_n^2\leq\dfrac{\theta_n\theta_{n-1}}{9}$ or $\tau_n\bar L_n\leq \dfrac{\sqrt{\theta_n\theta_{n-1}}}{3}$. Similarly, from \eqref{eq:alg2_c}, we also derive $\tau_n\leq\dfrac{1}{3}\sqrt{\dfrac{\theta_n\theta_{n-1}}{\beta\psi\|K\|^2}}$.
\end{remark}
\vspace{8pt}
\begin{proposition}\label{prop_1}
Suppose that $h$ is locally smooth and the sequence $(x_n)$ generated by Algorithm \ref{algorithm 2} is bounded. Then, the sequences $(\tau_n)$ and $(\theta_n)$ are bounded below by positive constants.
\end{proposition}
 \begin{proof}
        Since the sequence $(x_n)$ is bounded, by definition of local smoothness of $h$, there exist a constant $\bar L>0$ such that $\lVert \nabla h(x_n)-\nabla h(x_{n-1})\rVert\leq\bar L\lVert x_n-x_{n-1}\rVert$ $\forall n$. Thereafter, doing a similar calculation as in \cite[Lemma 4.2]{tam2023bregman}, we obtain
        \begin{equation*}
            \tau_n\geq\dfrac{\psi}{9(\bar L^2+\beta\psi\|K\|^2)}\dfrac{1}{\tau_{\max}}~~\text{and}~~ \theta_n\geq\dfrac{\psi^2}{9(\bar L^2+\beta\psi \|K\|^2)}\dfrac{1}{{\tau_{\max}^2}}~~\forall n.
        \end{equation*}
    \end{proof}
\begin{assumption}\label{assumption_5.1}
    Suppose that $f$ and $g$ are proper, convex, and lower semicontinuous. The function $h:\mathbb{X}\to\mathbb{R}$ is convex and locally smooth, i.e., for all compact sets $D$, where $D\subset\mathbb{X}$, there exists $L_D>0$ such that
\begin{equation*}
    \lVert \nabla h (x)-\nabla h (y) \rVert\leq \bar L_D\lVert x-y\rVert,~\forall x,y\in D.
\end{equation*}
\end{assumption}
\vspace{8pt}
\begin{lemma}\label{lemma_5.0.1}
 Under Assumptions \ref{assumption1}, \ref{assumption_2} and \ref{assumption_5.1}, let  $\{(z_n,x_n,w_n,y_n,\tau_n)\}$ be the sequence generated by Algorithm \ref{algorithm 2}. Suppose that $(x^*,w^*,y^*)\in\mathbf{\Omega}$ is a saddle point of $\mathscr{L}$. Then, for any $y\in\mathbb{Y}$, the following holds as $n\geq 1$
    \begin{multline}\label{adap_4}
        2\tau_n\mathbb{J}(x_n,w_n,y)+ \dfrac{\psi}{\psi-1}\lVert x^*-z_{n+2} \rVert^2 + \dfrac{1}{\beta}\lVert y-y_{n}\rVert^2+\dfrac{\theta_n}{3}\lVert x_n-x_{n+1} \rVert^2 \\\leq
        \dfrac{\psi}{\psi-1}\lVert x^*-z_{n+1} \rVert^2 +\dfrac{1}{\beta}\lVert y-y_{n-1} \rVert^2 + \dfrac{\theta_{n-1}}{3}\lVert x_n-x_{n-1} \rVert^2- {\theta_n}\lVert x_n-z_{n+1} \rVert^2.
    \end{multline}
\end{lemma}
\begin{proof}
Following a similar arguments as in Lemma \ref{Lemma4.1.1}, we obtain
    \begin{multline}\label{adap_1}
        2\tau_n\mathbb{J}(x_n,w_n,y)+  \lVert x^*-x_{n+1} \rVert^2 + \dfrac{1}{\beta}\lVert y-y_{n} \rVert^2\leq \lVert x^*-z_{n+1} \rVert^2 \\ +
        \dfrac{1}{\beta}\lVert y-y_{n-1} \rVert^2 -\lVert x_{n+1}-z_{n+1} \rVert
        ^2
        - \theta_n\lVert x_n-z_{n+1} \rVert^2-\theta_n\lVert x_n-x_{n+1} \rVert^2 \\ + \theta_n\lVert x_{n+1}-z_{n+1} \rVert^2 -\dfrac{1}{\beta}\lVert y_n-y_{n-1} \rVert^2 + 2\tau_n\langle y_n-y_{n-1},\, K(x_{n+1}-x_n)\rangle  \\ +
        2\tau_n\langle \nabla h(x_n)-\nabla h(x_{n-1}),\, x_n-x_{n+1}\rangle.
    \end{multline}
 Using the Cauchy-Schwarz inequality and Remark \ref{remark_5.2}, we have
  \begin{equation}\label{adap_2}
     \begin{aligned}
        2\tau_n\langle \nabla h(x_n)-\nabla h(x_{n-1}), x_n-x_{n+1}\rangle&\leq  
         2\tau_n\lVert \nabla h(x_n)-\nabla h(x_{n-1}) \rVert \lVert x_n-x_{n+1} \rVert \\ & 
         \leq2\tau_n\bar L_n\lVert  x_n - x_{n-1} \rVert \lVert x_n-x_{n+1} \rVert\\ &
          \leq\frac{2\sqrt{\theta_n\theta_{n-1}}}{3}\lVert  x_n - x_{n-1} \rVert \lVert x_n-x_{n+1} \rVert\\&
          \leq\dfrac{\theta_n}{3}\lVert x_n-x_{n+1}\rVert^2+\dfrac{\theta_{n-1}}{3}\lVert x_n-x_{n-1}\rVert^2.
     \end{aligned}
 \end{equation}
Similarly, by Remark \ref{remark_5.2}, we obtain
\begin{equation}\label{adap_3}
    \begin{aligned}
        2\tau_n\langle y_n-y_{n-1},\, K(x_{n+1}-x_n)\rangle&\leq 2\tau_n \lVert Kx_n-Kx_{n+1} \rVert \hspace{.1cm} \lVert y_n-y_{n-1} \rVert \\ & 
        \leq2\tau_n\|K\|\lVert x_n-x_{n+1} \rVert\lVert y_n-y_{n-1} \rVert\\&
        \leq\frac{2}{3}\sqrt{\frac{\theta_n\theta_{n-1}}{\beta\psi}}\lVert x_n-x_{n+1} \rVert\lVert y_n-y_{n-1} \rVert\\&
        \leq \dfrac{\theta_n}{3}\lVert x_n-x_{n+1} \rVert^2+\dfrac{\theta_{n-1}}{3\beta\psi}\lVert y_n-y_{n-1} \rVert^2.
    \end{aligned}
\end{equation}
Now, combining (\ref{adap_1}), (\ref{adap_2}), (\ref{adap_3}), and (\ref{derived_equality2}), we derive
    \begin{multline}\label{adap_5}
        2\tau_n\mathbb{J}(x_n,w_n,y)+ \dfrac{\psi}{\psi-1}\lVert x^*-z_{n+2} \rVert^2 + \dfrac{1}{\beta}\lVert y-y_{n} \rVert^2\leq \dfrac{\psi}{\psi-1}\lVert x^*-z_{n+1} \rVert^2\\+
        \dfrac{1}{\beta}\lVert y-y_{n-1} \rVert^2 + \left(\theta_n-1-
        \dfrac{1}{\psi}\right)\lVert x_{n+1}-z_{n+1} \rVert
        ^2 - {\theta_n}\lVert x_n-z_{n+1} \rVert^2 + \dfrac{\theta_{n-1}}{3}\lVert x_n-x_{n-1} \rVert^2\\
        -\left(\theta_n-\dfrac{\theta_n}{3}-\dfrac{\theta_n}{3}\right)\lVert x_n-x_{n+1} \rVert^2 
         -\dfrac{1}{\beta}\left(1-\dfrac{\theta_{n-1}}{3\psi}\right)\lVert y_n-y_{n-1} \rVert^2.
    \end{multline}
Since $\tau_n \leq \rho \tau_{n-1}~~\forall n$, it follows that $\theta_n \leq \psi \rho \leq 1 + \dfrac{1}{\psi}~~\forall n$. Furthermore, we have 
\begin{equation*}
    1 - \dfrac{\theta_{n-1}}{3\psi} \geq 1 - \dfrac{1}{3} \left(\dfrac{1}{\psi} + \dfrac{1}{\psi^2}\right) > 0, ~~\forall\psi\in (1, \phi].
\end{equation*}
Finally, observing $\theta_n-\dfrac{\theta_n}{3}-\dfrac{\theta_n}{3}=\dfrac{\theta_n}{3}$ together with~(\ref{adap_5}), completes the proof of Lemma~\ref{lemma_5.0.1}.
\end{proof}
 \begin{theorem}
  Under Assumptions \ref{assumption1}, \ref{assumption_2} and \ref{assumption_5.1}, suppose that $\{(z_n,x_n,w_n,y_n,\tau_n)\}$ is the sequence generated by Algorithm \ref{algorithm 2}. Let $(x^*, w^*, y^*)\in\Omega.$
   \begin{enumerate}
       \item Then, the sequence $\{(x_n,y_n)\}$ converges to a solution of (\ref{1.2}).
       \item  There exists a constant $C_1>0$ such that
    \begin{equation*}
       \left | \Phi(\Tilde{x}_N,\, \Tilde{w}_N)- \Phi(x^*, w^*)\right| \leq \frac{C_1}{N}\hspace{.17cm}and \hspace{.17cm}\left | K\Tilde{x}_N- \Tilde{w}_N \right|\leq \frac{2C_1}{\delta N}
    \end{equation*}
holds, where $\delta$ is a positive constant satisfying $\delta\geq 2\|y^*\|$, $N\geq 1$ and
$\Tilde x_N$, $\Tilde y_N$ are defined in \eqref{x_ntilde and w_ntilde}.    
   \end{enumerate}
\end{theorem}
\begin{proof} (1) Since $(x^*,w^*,y^*)$ is a saddle point of $\mathscr{L}$, it follows that $2\tau_n\mathbb{J}(x_n,w_n,y^*)\geq0$, $\forall n$. Now, (\ref{adap_4}) can be written as  
$$p_{n+1}(y)\leq p_n(y)- q_n~~\forall n,$$
where 
\begin{equation}\label{Theo4.1_1}
    \begin{aligned}
    & p_{n}(y) \coloneqq  \dfrac{\psi}{\psi-1}\lVert x^*-z_{n+1} \rVert^2+\dfrac{1}{\beta}\lVert y-y_{n-1} \rVert^2+\frac{\theta_{n-1}}{3}\lVert x_n-x_{n-1} \rVert^2 ,\\ &
    q_n \coloneqq {\theta_n}\lVert x_n-z_{n+1} \rVert^2.
    \end{aligned}
    \end{equation}
By using Lemma \ref{usefullemma2}, we obtain $\lim_{n\to\infty} p_n(y)\in \mathbb{R}$ and $\lim_{n\to\infty} q_n = 0$. Moreover, from Proposition \ref{prop_1}, we derive that $\lim_{n\to\infty} \lVert x_n - z_{n+1} \rVert^2 = 0$. Similarly, like in Theorem \ref{thm_3.1.1}, combining with Proposition \ref{prop_1}, we can show that $\lim_{n\to\infty} \lVert x_n - x_{n-1} \rVert^2 = 0$. Since $\lim_{n\to\infty} p_n(y) \in \mathbb{R}$ for any $y$ and $\lim_{n\to\infty} q_n = 0$, both imply that $\{x_n\}$, $\{y_n\}$ and $\{z_n\}$ are bounded sequences. Let $(\bar x, \bar y)$ be the subsequential limit of $\{(x_n,y_n)\}$. Then there exists a subsequence $\{(x_{n_{k}},y_{n_{k}})\}$ of $\{(x_n,y_n)\}$ such that $\{(x_{n_{k}},y_{n_{k}})\}$ converges to $(\bar x, \bar y)$, i.e., $\lim_{n\to\infty} x_{n_{k}}= \bar x$ and $\lim_{n\to\infty} y_{n_{k}}= \bar y$. Noting, $\lim_{n\to\infty}\lVert x_n-z_{n+1} \rVert^2=0$ implies $\lim_{n\to\infty} z_{n_{k}}= \bar x$. Furthermore, using Lemma \ref{usefullemma3} similarly like in Theorem \ref{thm_3.1.1}, for all $(x,y)\in\mathbb{X}\times\mathbb{Y}$, we have
\begin{equation*}
    \begin{aligned}
        \langle x_{n_{k}}-z_{n_{k}}+\tau_{n_{k}-1}K^*y_{n_{k}-1}+\tau_{n_{k}-1}\nabla h(x_{n_{k}-1}), \,x-x_{n_{k}}\rangle&\geq \tau_{n_{k}-1}(f(x_{n_{k}})-f(x)), \\
        \left\langle \frac{1}{\beta}(y_{n_{k}}-y_{n_{k}-1})+\tau_{n_{k}}Kx_{n_{k}}, \, y-y_{n_{k}-1}\right\rangle&\geq \tau_{n_{k}}(g^*(y_{n_{k}})-g^*(y)).
    \end{aligned}
\end{equation*}
Recalling that both $f,g^*$ are lower semi-continuous and letting $k\to\infty$, we derive
 \begin{equation}\label{theo4.1_5}
     \begin{aligned}
            \langle K^*\bar y+\nabla h(\bar x),\, x-\bar x \rangle\geq f(\bar x)-f(x)\hspace{.17cm} &\text{and}~
      \langle K\bar x,\, y-\bar y \rangle\geq g^*(\bar y)-g^*(y).
     \end{aligned}
 \end{equation}
 Equation (\ref{theo4.1_5}) implies that $(\bar x,\bar y)$ is a saddle point of (\ref{1.2}). Thus, all the above validation holds, including equation (\ref{Theo4.1_1}), replacing $x^*$ by $\bar x$ and $y^*$ by $\bar y$. This implies that $\lim_{n\to\infty} p_{n_k}(\bar y) =0$. Furthermore, the sequence $\{p_{n_k}(\bar y)\}$ is monotonically non-increasing and bounded below. Thus, we have $\lim_{n\to\infty} p_{n}(\bar y) =0$. Therefore, $\lim_{n\to\infty} z_{n}= \bar x$ and $\lim_{n\to\infty} y_{n}= \bar y$. Additionally, $\lim_{n\to\infty}\lVert x_n-z_{n+1} \rVert^2=0$ implies that $\lim_{n\to\infty} x_{n}= \bar x$. Hence, this completes the proof.\\
 \\
 (2) For any $y\in\mathbb{Y}$, using the definitions of $q_n$ and $p_n(y)$ in \eqref{Theo4.1_1} and applying \eqref{adap_4}, we obtain
    \begin{equation*}
         2\tau_n\mathbb{J}(x_n,w_n,y)\leq p_{n}(y)-p_{n+1}(y).
    \end{equation*}
 Now, from Proposition \ref{prop_1}, we get $\tau_n\geq\tau\coloneqq\dfrac{\psi}{9(\bar L^2+\beta\psi \|K\|^2)}\dfrac{1}{\tau_{\max}}$, for all $n$. Thus, we have
 \begin{equation*}
      2\tau\mathbb{J}(x_n,w_n,y)\leq p_{n}(y)-p_{n+1}(y).
 \end{equation*}
Furthermore, summing over $n=1,2,\ldots,N-1$, we obtain
\begin{equation*}
\begin{aligned}
   & 2\tau\sum_{n=1}^{N-1}\mathbb{J}(x_n,w_n,y)\leq p_{1}(y)-p_{N}(y) \\&
    \leq p_{1}(y)=\dfrac{\psi}{\psi-1}\lVert x^*-z_{2} \rVert^2+ \dfrac{1}{\beta}\lVert y-y_{0} \rVert^2+\frac{\theta_0}{3}\lVert x_{1}-x_{0} \rVert^2.
\end{aligned}
\end{equation*}
Following a similar argument as in Theorem \ref{thm_4.0.1}, we obtain 
\begin{equation*}
       \left | \Phi(\Tilde{x}_N,\, \Tilde{w}_N)- \Phi(x^*, w^*)\right| \leq \frac{C_1}{N}\hspace{.17cm}\text{and} \hspace{.17cm}\left | K\Tilde{x}_N- \Tilde{w}_N \right|\leq \frac{2C_1}{\delta N},
    \end{equation*}
where $C_1= \frac{1}{2\tau}\left[\dfrac{\psi}{\psi-1}\lVert x^*-z_{2} \rVert^2+ \dfrac{1}{\beta}(\delta+\lVert y_{0} \rVert^2)+\dfrac{\theta_0}{3}\lVert x_{1}-x_{0} \rVert^2\right]$.
\end{proof}
\section{Numerical Experiments}\label{section_6}
In this section, we provide numerical results on LASSO, Logistic regression, Graph Net, Fused LASSO, and Image inpainting problems for Algorithm \ref{algorithm 1} and Algorithm~\ref{algorithm 2}, and compare them with related existing methods. All experiments were executed in \textbf{Python 3.11.13} on a 64-bit Linux system hosted on Google Colab environment, equipped with an \textbf{Intel Xeon CPU @ 2.20 GHz} and \textbf{12.7 GB~RAM}. All source codes used in this paper are available at
\url{https://github.com/soesantanu/Adaptive-GRPDA-Experiments}.

\noindent Unless stated otherwise, we fix $\psi=1.5$, $\rho=\tfrac{1}{\psi}+\tfrac{1}{\psi^{2}}$, and $\tau_{\max}=10^{7}$ for aEGRPDA (Algorithm~\ref{algorithm 2}) and aGRAAL~\cite{malitsky2020golden}. The choices of $\psi$ and $\rho$ are aligned with \cite{malitsky2020golden}, which reported improved numerical performance with the same parameter values. Furthermore, we choose $(\tau,\sigma)$ to satisfy each method’s stepsize requirements: PDHG uses $\tau\sigma\|K\|^2<1$; GRPDA uses $\tau\sigma\|K\|^2<\psi$; Condat–V\~u uses $\tau\sigma\|K\|^2+\tfrac{\tau\bar{L}}{2}\le 1$; and E\text{-}GRPDA uses $\tau\sigma\|K\|^2+2\tau\bar{L}<\psi$. For GRPDA-L, we adopt the linesearch requirements of~\cite{chang2022goldenlinesearch}; and for P\text{-}GRPDA we choose $(\psi,\mu)$ in accordance with~\eqref{extension_on_mu}. In each of our experiments, $F(x_n)$ denotes the objective value evaluated at the iterate $x_n$ generated by the algorithms, and $F^* = \inf_{x} F(x)$ is computed by running the algorithm for a sufficiently large number of iterations when the true solution $x^*$ is unavailable.


\subsection{The LASSO Problem}\label{lasso}

The $\ell_1$-penalized linear regression or LASSO problem \cite{tibshirani1996regression} is formulated as:
\begin{equation*}
    \min_{x\in\mathbb{R}^n}F(x)\coloneqq\frac{1}{2} \lVert Kx - b \rVert^2 + \lambda \lVert x \rVert_1,
\end{equation*}
where the rows of $K \in \mathbb{R}^{m \times n}$ represent predictor variables, $\lambda$ is a regularization parameter, and $b \in \mathbb{R}^m$ is a given response vector. 
The objective is to find the unknown signal $x\in\mathbb{R}^n$. LASSO is a popular method in statistics for linear regression and is particularly interesting in compressive sensing when $n > m$, i.e., when the rank of $K$ is less than $n$. Comparing LASSO with (\ref{1.1}) gives $g(\cdot)=\frac{1}{2}\|\cdot-b\|^{2}$, $f(\cdot) = \lambda \lVert \cdot \rVert_1$, and $h=0$. Next, following \cite{chang2022grpdarevisited,chang2021goldengrpda,malitsky2018first}, we generate the required data for this problem. The $s$ nonzero coordinates of the true signal $x^*\in\mathbb{R}^n$ are taken randomly using uniform normal distribution $\mathcal{N}(-10,10)$ and the rest are set to $0$. The entries of additional noise $\omega\in\mathbb{R}^m$ are drawn from $\mathcal{N}(0,0.1)$ and we set $b=Kx^*+\omega$. We generate the matrix $K$ in one of the following ways:
\begin{enumerate}[(i)]
        \item First, we construct a  matrix $B$ whose entries are drawn independently from $\mathcal{N}(0,1)$. Then for $q\in(0,1)$, we generate the columns of $K$ as follows: set $K_1=\frac{B_1}{\sqrt{1-q^2}}$ and $K_j = qK_{j-1} + B_{j}$, where $K_j,B_j$ are columns of $K,B$ for $j=2,3,\ldots,n$,  respectively.\label{caseii}
        \item All entries of $K$ are sampled independently from $\mathcal{N}(0,1)$.\label{case1}
    \end{enumerate}
\vspace{0.1cm}
We set the initial values for all algorithms as $x_0 = 0$, $y_0 = -b$, and $\lambda = 0.1$. The remaining parameters are chosen as follows:
\vspace{0.1cm}
\begin{itemize}
    \item \textbf{PDHG}: $\tau = \frac{25}{\|K\|}$, $\sigma = \frac{0.04}{\|K\|}$, where $\|K\|^2 = \lambda_{\max}(K^*K)$.
    
    \item \textbf{GRPDA}: $\phi = 1.618$, $\tau = \frac{\sqrt{\phi}}{\|K\|}$, $\sigma = \frac{\sqrt{\phi}}{\|K\|}$, where $\|K\|^2 = \lambda_{\max}(K^*K)$.
    
    \item \textbf{GRPDA-L}: $\tau_0 = 10$, $\delta = 0.99$, $\mu = 0.7$, $\beta=0.2$ where $\delta$, $\mu$ and $\beta$ appear in the linesearch conditions of GRPDA-L.
    
    \item \textbf{P-GRPDA} (Algorithm~\ref{algorithm 1}): $\tau_0 = 10$, $\psi = 1.76$, $\mu = 0.77236$, $\beta = 0.2$.
\end{itemize}

The stepsizes $\tau$ and $\sigma$ for PDHG and GRPDA are chosen in accordance with their respective stepsize conditions and are tuned to enhance practical convergence; see~\cite{chang2021goldengrpda,chang2022grpdarevisited} for further details. For P\text{-}GRPDA, we empirically observe that taking a larger $\psi$ with $\mu\in[0.70,0.81]$ and $\beta=0.2$ improves performance, which is associated with a faster decay of the iterate gap $\|x_n - z_n\|$. The parameters required for GRPDA-L are selected following~\cite{chang2022goldenlinesearch}. Since primal--dual methods are typically sensitive to the parameter $\beta $, we proceed in two stages to ensure a fair comparison: we first run P\text{-}GRPDA and GRPDA-L using the same values of $\beta$ along with the other algorithms, and then we perform a sensitivity study for P\text{-}GRPDA over a range of $\beta$ values.

\begin{figure}[htbp]
  \centering
  \begin{subfigure}[b]{0.42\textwidth}
    \includegraphics[width=\textwidth]{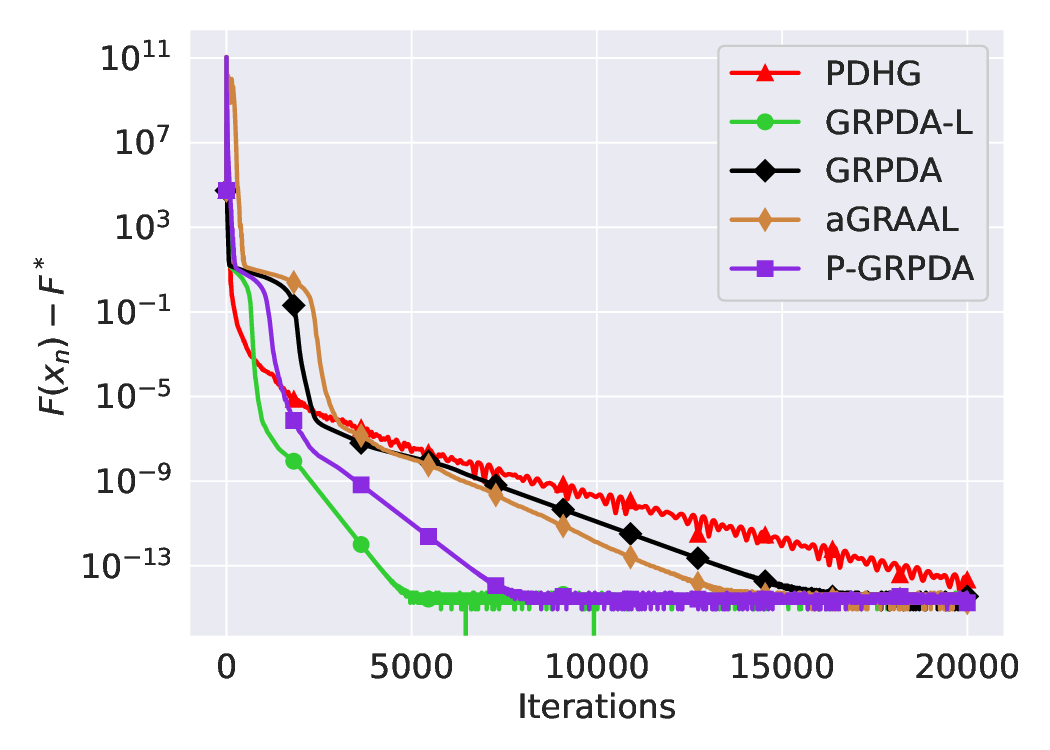}
    \subcaption{$(m,n,s)=(300,1000,10)$}
    \label{fig:lasso1a}
  \end{subfigure}\hfill
  \begin{subfigure}[b]{0.43\textwidth}
    \includegraphics[width=\textwidth]{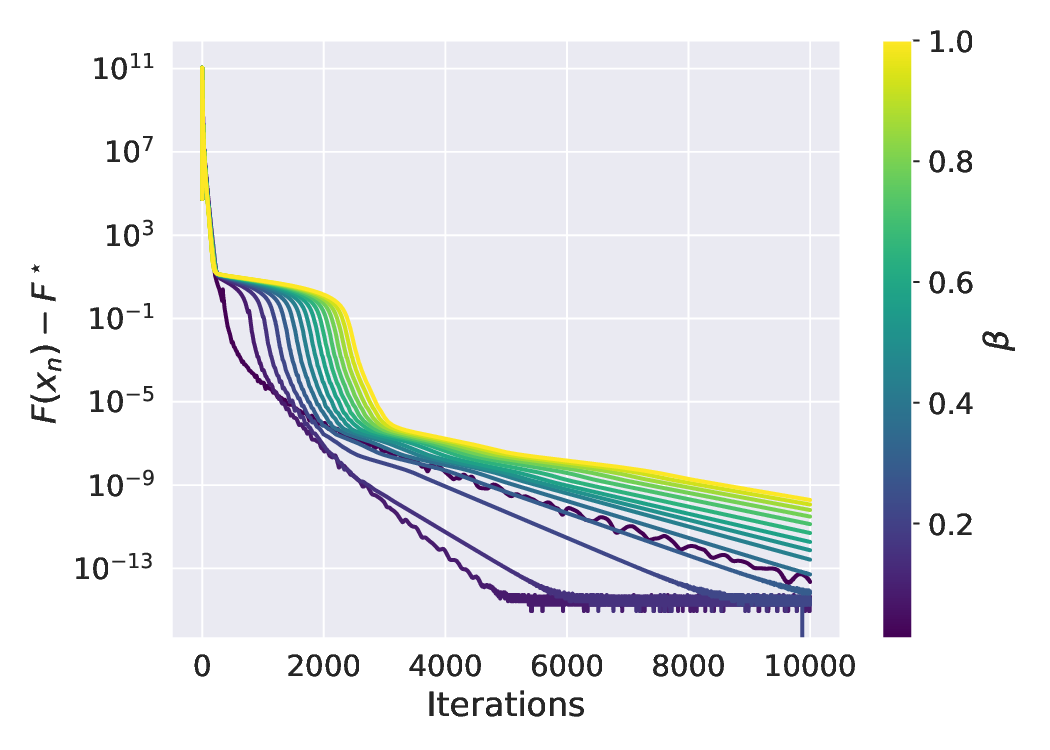}
    \subcaption{$(m,n,s)=(300,1000,10)$}
    \label{fig:lasso1b}
  \end{subfigure}

  \vspace{.3em}

  \begin{subfigure}[b]{0.42\textwidth}
    \includegraphics[width=\textwidth]{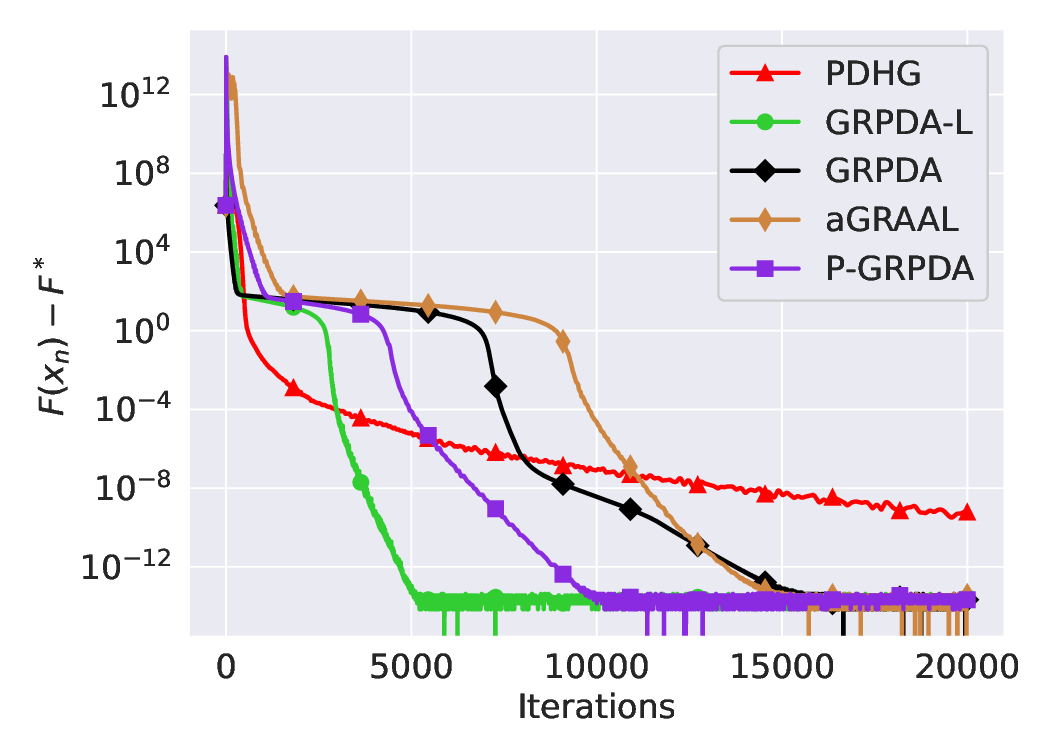}
    \subcaption{$(m,n,q,s)=(1000,2000,\,0.5,100)$}
    \label{fig:lasso1c}
  \end{subfigure}\hfill
  \begin{subfigure}[b]{0.43\textwidth}
    \includegraphics[width=\textwidth]{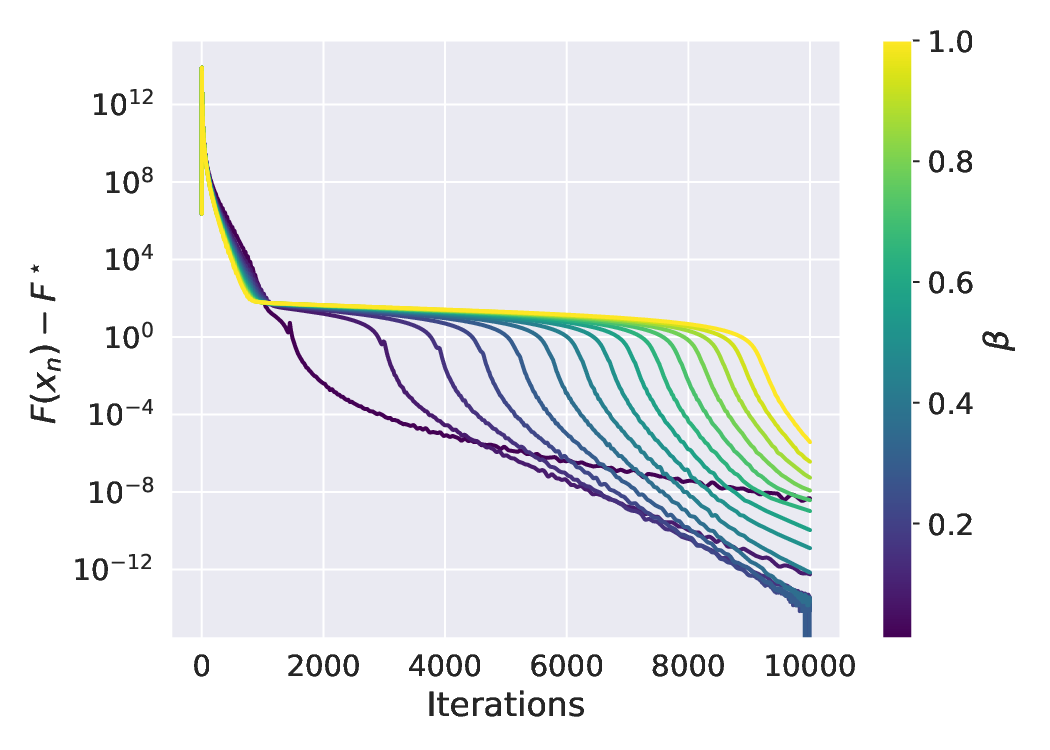}
    \subcaption{$(m,n,q,s)=(1000,2000,\,0.5,100)$}
    \label{fig:lasso1d}
  \end{subfigure}

  \caption{Convergence plot for LASSO of different algorithms. In (a)–(b), the matrix $K$ is generated by scheme (ii), and in (c)–(d) by scheme (i). Here, (b) and (d) illustrate the decay of the objective gap for P-GRPDA on the LASSO problem for different choices of~$\beta$}
  \label{fig:Lasso_Fig_1}
\end{figure}

We plot $F(x_n)-F^\ast$ against the iteration number for all algorithms, where $F^\ast$ is computed as described earlier. From Figure~\ref{fig:Lasso_Fig_1}, we see that, for the same value of $\beta$, GRPDA-L outperforms P\text{-}GRPDA. This behaviour can be attributed to the fact that, in P\text{-}GRPDA, the sequence $(\tau_n)$ is decreasing, which causes the updates to become more conservative and slows down the progress of the iterates. We also report in Figure~\ref{fig:Lasso_Fig_1} the \emph{sensitivity study} mentioned above, in which we run P\text{-}GRPDA for several values of $\beta$ while keeping the remaining parameters fixed. In this study, we observe that values of $\beta$ in the range $(0.1,0.4)$ yield superior performance compared with the other tested choices.

\subsection{Logistic Regression Problem}
Our next test problem is the regularized logistic regression model
\begin{equation}\label{log_reg}
   \min_{x\in\mathbb{R}^n}F(x)\coloneqq \min_{x\in\mathbb{R}^n}
    \sum_{i=1}^m \log\bigl(1+\exp(-b_i a_i^\top x)\bigr)
    + R(x),
\end{equation}
where $\{(a_i,b_i)_{i=1}^m\}\in\mathbb{R}^n\times\{-1,1\}$ and $R$ is a convex regularizer. 
Although primal--dual splitting is not the most direct approach for this formulation, our methods can still be applied by setting $h(x) := \sum_{i=1}^m \log\!\bigl(1+\exp(-b_i a_i^\top x)\bigr),$ and the nonsmooth component as $f(x)=0$ and $g(x)=R(x)$. The gradient of $h$ is given by
$$
\nabla h(x) = -\sum_{i=1}^m \frac{b_i\,a_i}{1 + \exp(b_i a_i^\top x)}.
$$
By a simple computation, one can show that $h$ is $\frac{1}{4}\|A\|^2_2$- smooth.
In this experiment, we consider two types of regularizers:

\begin{figure}[htbp]
  \centering
  \begin{subfigure}[b]{0.33\textwidth}
    \includegraphics[width=\textwidth]{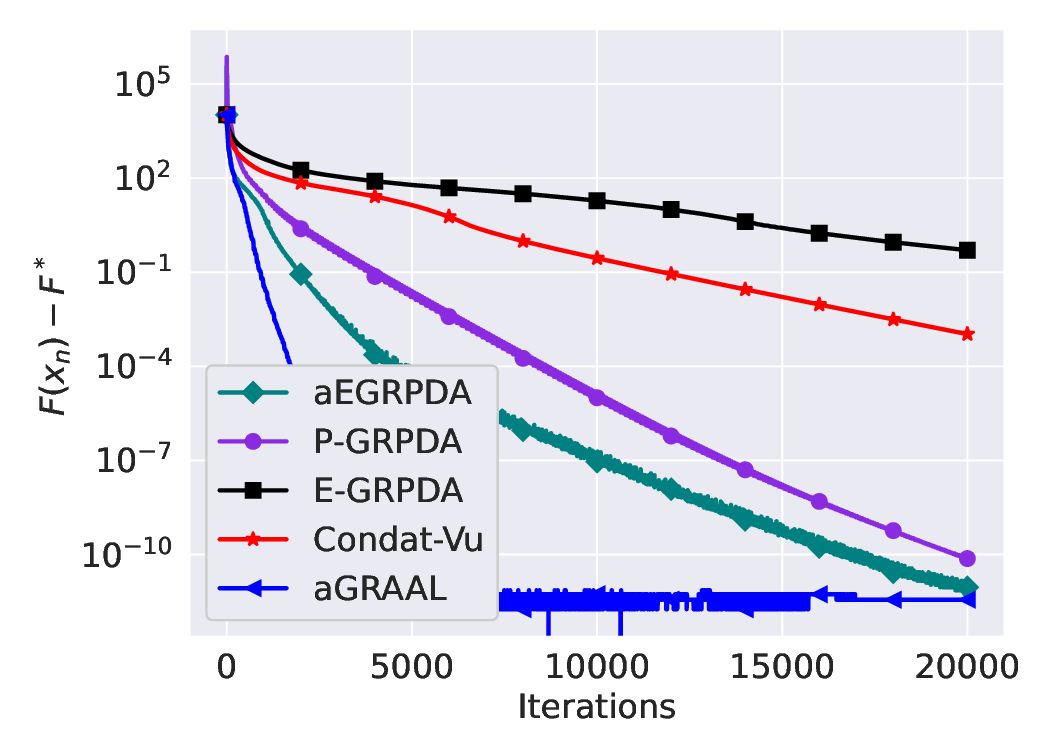}
    \subcaption{a9a}
    \label{fig:slr_a9a_main}
  \end{subfigure}\hfill
  \begin{subfigure}[b]{0.33\textwidth}
    \includegraphics[width=\textwidth]{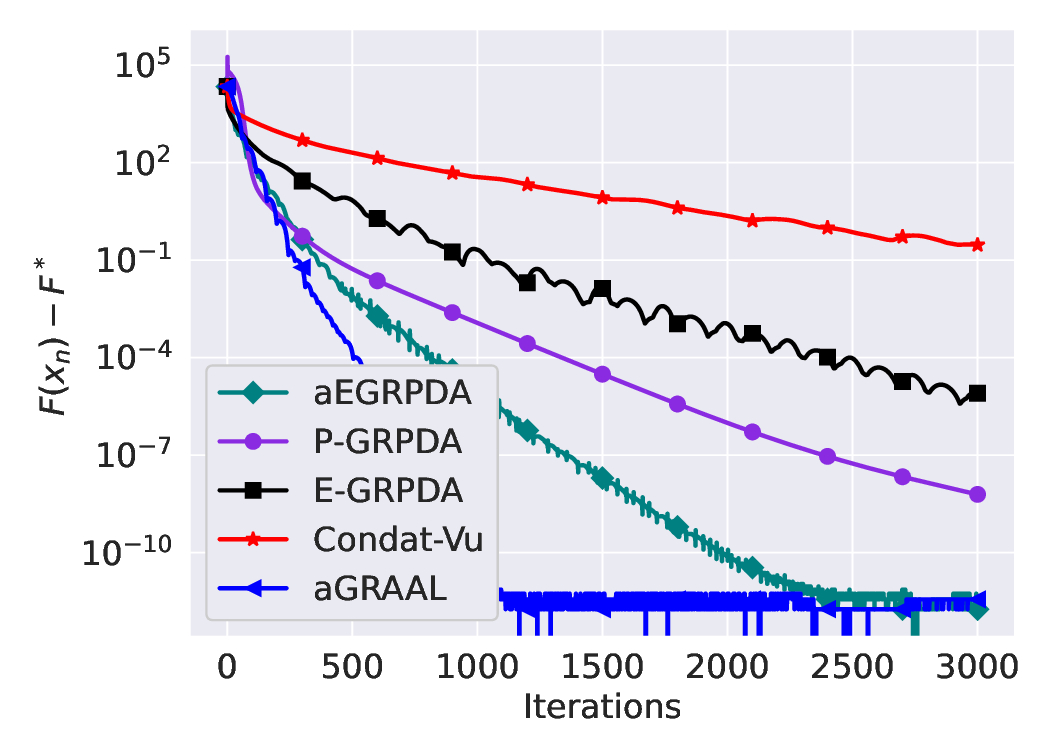}
    \subcaption{ijcnn1}
    \label{fig:slr_a9a_primal}
  \end{subfigure}\hfill
  \begin{subfigure}[b]{0.33\textwidth}
    \includegraphics[width=\textwidth]{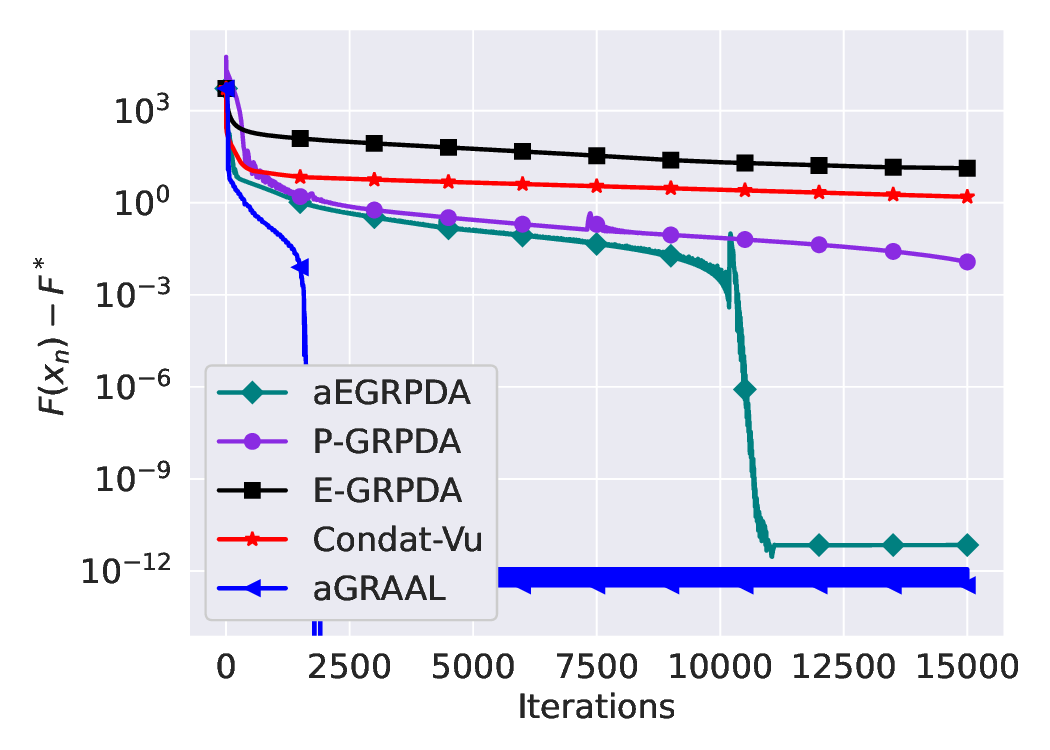}
    \subcaption{mushrooms}
    \label{fig:slr_a9a_dual}
  \end{subfigure}
  \caption{Convergence plots for logistic regression with \textbf{setting 1} on the three LIBSVM datasets \texttt{a9a}, \texttt{ijcnn1}, and \texttt{mushrooms}.}
  \label{fig:slr_all}
\end{figure}

\begin{figure}[htbp]
  \centering
  \begin{subfigure}[b]{0.32\textwidth}
    \includegraphics[width=\textwidth]{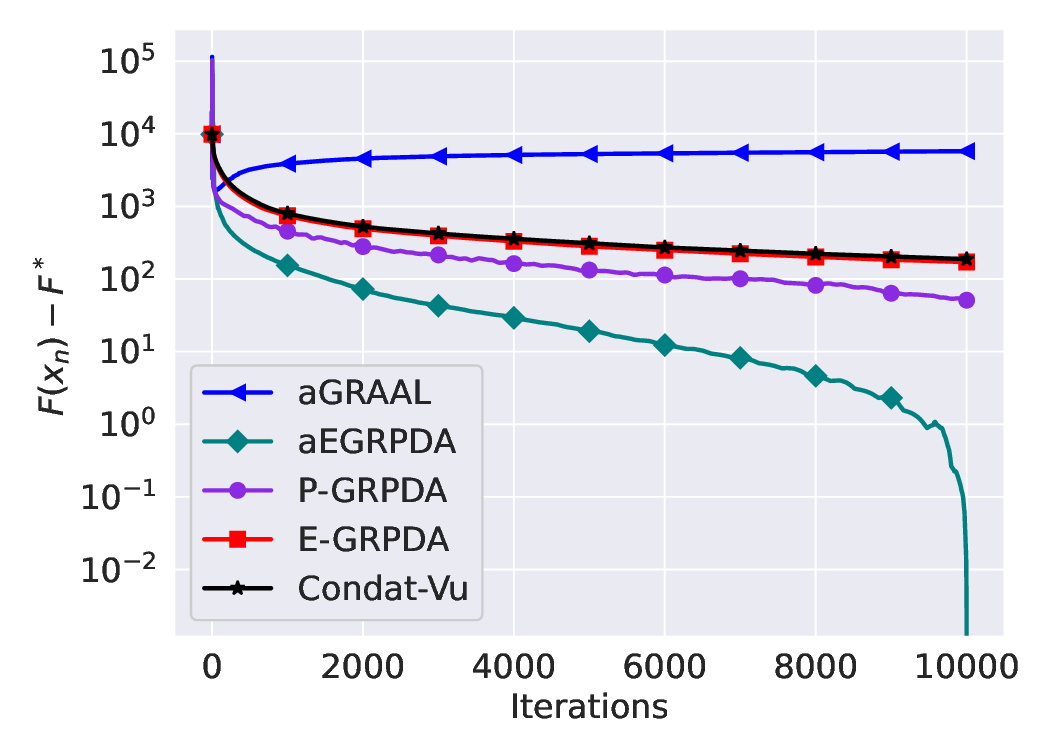}
    \subcaption{a9a}
  \end{subfigure}\hfill
  \begin{subfigure}[b]{0.32\textwidth}
    \includegraphics[width=\textwidth]{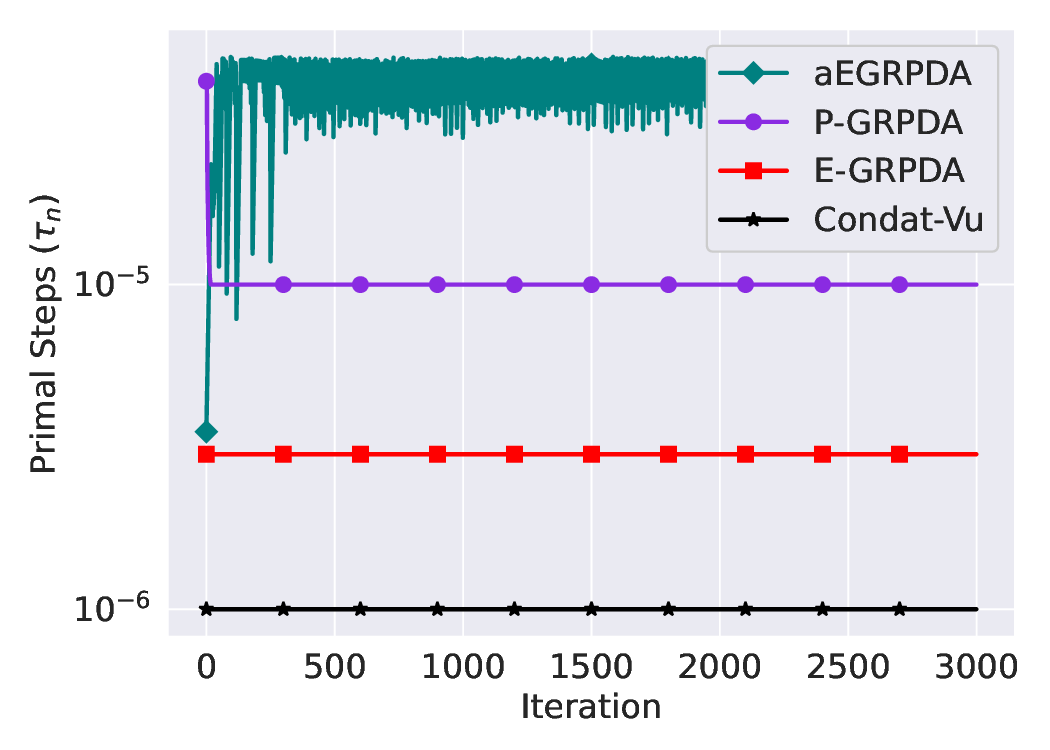}
    \subcaption{Primal stepsize ($\tau_n$)}
  \end{subfigure}\hfill
  \begin{subfigure}[b]{0.32\textwidth}
    \includegraphics[width=\textwidth]{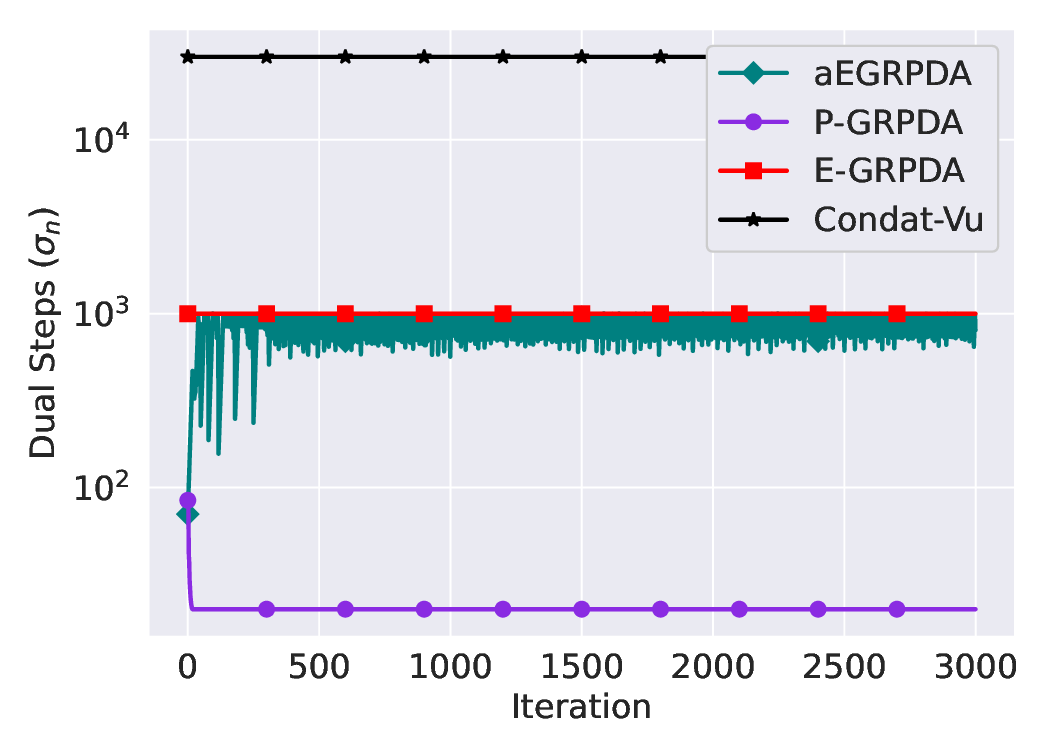}
    \subcaption{Dual stepsize ($\sigma_n$)}
  \end{subfigure}

  \vspace{1em}

  \begin{subfigure}[b]{0.32\textwidth}
    \includegraphics[width=\textwidth]{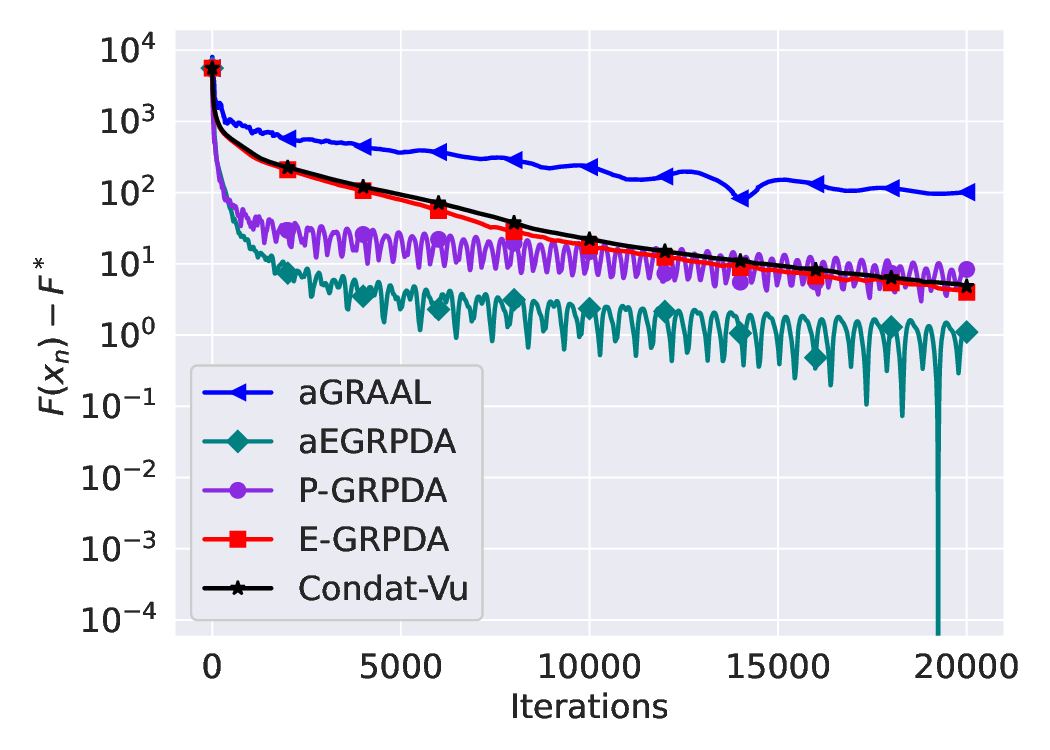}
    \subcaption{mushrooms}
  \end{subfigure}\hfill
  \begin{subfigure}[b]{0.32\textwidth}
    \includegraphics[width=\textwidth]{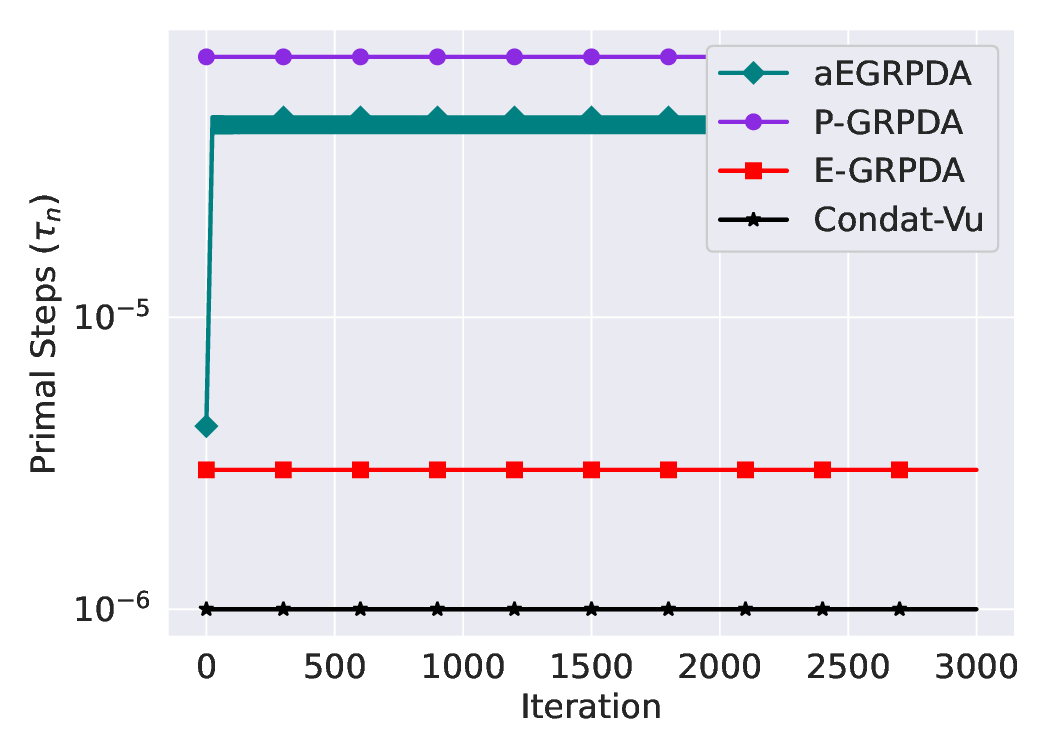}
    \subcaption{Primal stepsize ($\tau_n$)}
  \end{subfigure}\hfill
  \begin{subfigure}[b]{0.32\textwidth}
    \includegraphics[width=\textwidth]{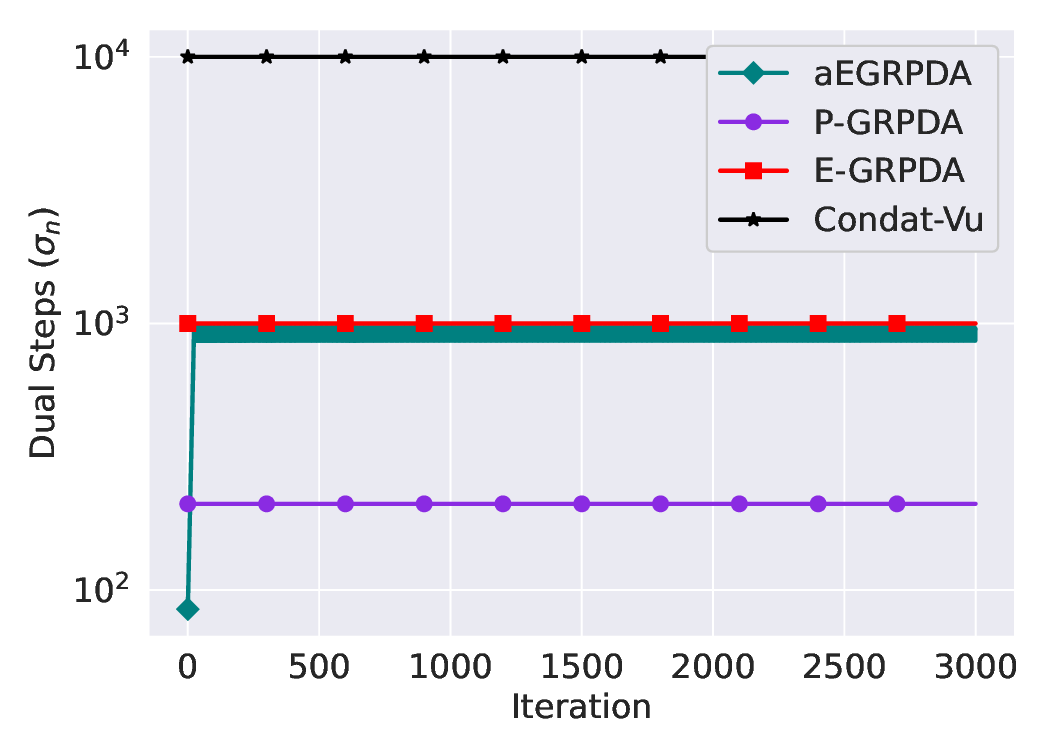}
    \subcaption{Dual stepsize ($\sigma_n$)}
  \end{subfigure}

  \vspace{1em}

  \begin{subfigure}[b]{0.32\textwidth}
    \includegraphics[width=\textwidth]{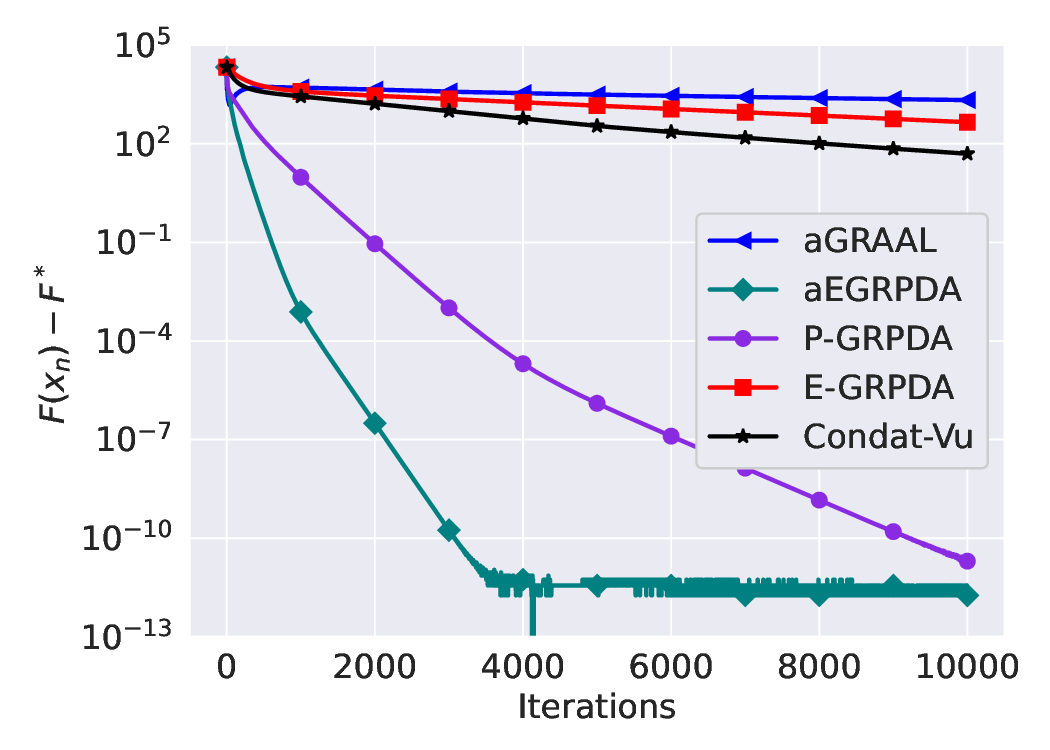}
    \subcaption{ijcnn1}
  \end{subfigure}\hfill
  \begin{subfigure}[b]{0.32\textwidth}
    \includegraphics[width=\textwidth]{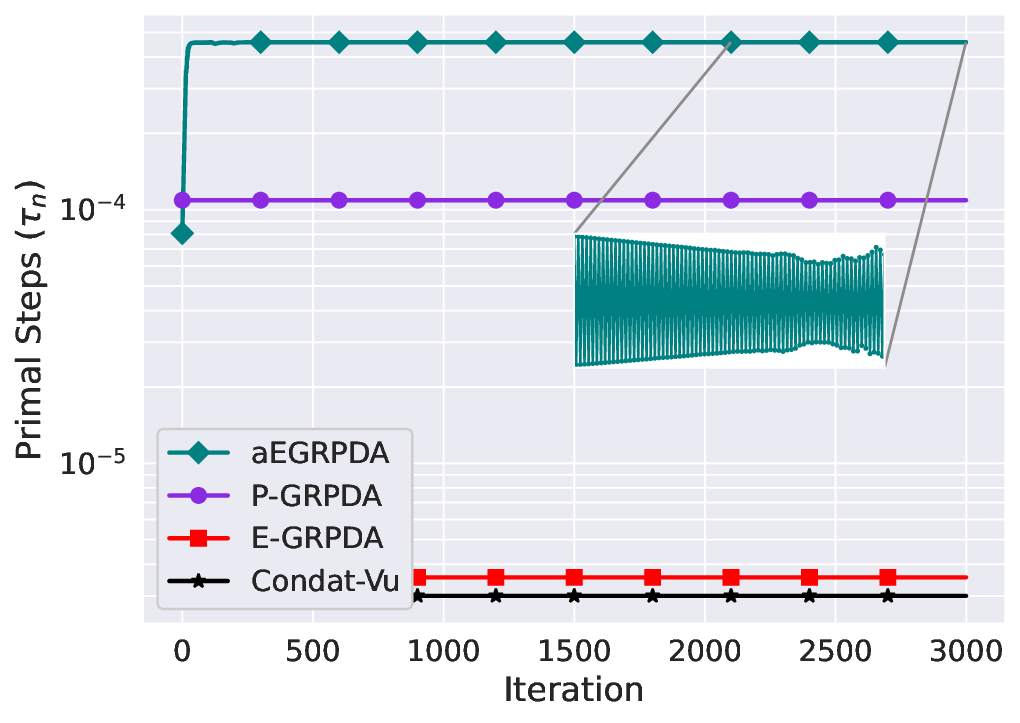}
    \subcaption{Primal stepsize ($\tau_n$)}
  \end{subfigure}\hfill
  \begin{subfigure}[b]{0.32\textwidth}
    \includegraphics[width=\textwidth]{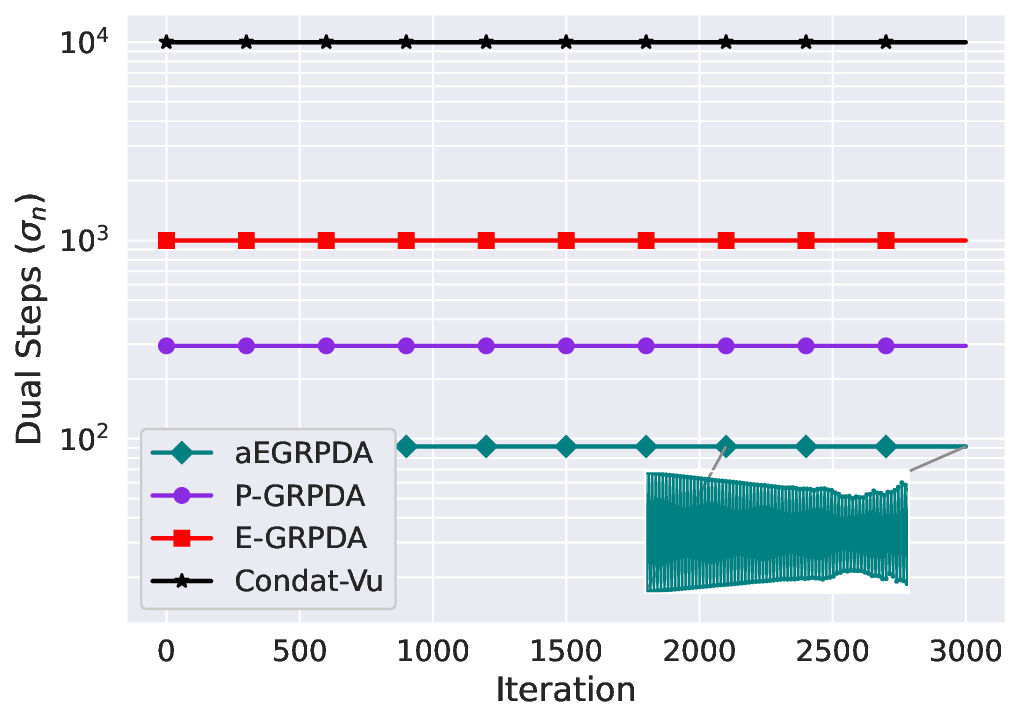}
    \subcaption{Dual stepsize ($\sigma_n$)}
  \end{subfigure}

  \caption{Convergence plots for logistic regression with \textbf{setting 2} on the three LIBSVM datasets \texttt{a9a}, \texttt{ijcnn1}, and \texttt{mushrooms}. In each row: (left) objective gap $F(x_n)-F^*$ vs iterations, (center) primal stepsize $\tau_n$, and (right) dual stepsize $\sigma_n$.}
  \label{fig:slr_comparison}
\end{figure}

\medskip
\textbf{Setting 1:} When $R(x) = \lambda\|x\|_1$, $\lambda>0$, \eqref{log_reg} fits our model \eqref{1.1} with $K=I$ (which reduces the computational cost of matrix–vector multiplication), $g(\cdot) = \lambda\|\cdot\|_1$, and $f,h$ chosen as above. Since both of our algorithms (Algorithm \ref{algorithm 1} and \ref{algorithm 2}) take advantage over aGRAAL when the linear operator $K$ in \eqref{1.1} is nonzero and non-identity by exploiting the structure. So, in this case, instead of applying aGRAAL in the primal-dual form, we apply aGRAAL directly to \eqref{log_reg}. 

\medskip
\textbf{Setting 2:} In the second scenario, we take $R(x) = \lambda_1\|x\|_1 + \lambda_2\|Dx\|_1$, where $D$ is the first-order difference matrix and $\lambda_1,\lambda_2>0$. This corresponds to~\eqref{1.1} with $f(\cdot)=\lambda_1\|\cdot\|_1$, $g(\cdot)=\lambda_2\|\cdot\|_1$, $K=D$, and $h$ as above.

\medskip
Numerical experiments are conducted on the standard LIBSVM datasets \texttt{a9a}, \texttt{ijcnn1}, and \texttt{mushrooms}.
Following \cite{vladarean2021first}, we set $\lambda = 0.005\|A^{\top} b\|_{\infty}$, where $A=[a_1^\top;\dots;a_m^\top]\in\mathbb{R}^{m\times n}$ and $b=[b_1,\ldots,b_m]\in\mathbb{R}^m$, and $\lambda_1=1, \lambda_2=150$.

All methods are initialized with $x_0=0$ and $y_0=x_0+\eta$, where $\eta\sim\mathcal{N}(0,10^{-9}I)$ is a small perturbation. Since there is no formal rule for selecting $(\tau,\sigma)$ for Condat--V\~u and E\mbox{-}GRPDA, we performed moderate parameter tuning and found that $\sigma\in[2,4]$ together with $\tau\in[2\times10^{-6},\,6.1\times10^{-4}]$ yields stable and fast convergence for \texttt{ijcnn1} dataset, $\sigma\in[10^2,10^4]$ together with $\tau\in[10^{-6},\,5\times10^{-5}]$ for \texttt{a9a} and \texttt{mushrooms}. Moreover, the computation of $\tau_n$ in P\mbox{-}GRPDA and aEGRPDA follows the local stepsize inequality of E\mbox{-}GRPDA; thus, it is sensible to choose the initial parameters (e.g., $\tau_0$) close to the $\tau$ used for E\mbox{-}GRPDA for good performance, while values such as $\tau_0=1$ or $\tau_0=10$ can be used as quick sanity checks. Since $\sigma_n=\beta\,\tau_n$, we sweep $\beta$ over $[8.1\times10^{2},\,10^{7}]$, a range motivated by the observed ratio $\sigma/\tau$ in Condat-V\~u and E\mbox{-}GRPDA. We therefore adopt the following parameter choices for the logistic regression problem in \textbf{Setting~1}:
\begin{itemize}
  \item \textbf{Condat-V\~u}: $\tau = 6.5 \times 10^{-5}$, $\sigma = 2$ for \texttt{ijcnn1}; $\tau = 3 \times 10^{-6}$, $\sigma = 10^4$ for \texttt{a9a}; $\tau = 10^{-6}$, $\sigma = 10^4$ for \texttt{mushrooms}  .
  \item \textbf{E-GRPDA}: $\tau = 6.321 \times 10^{-4}$, $\sigma = 2.78$ for \texttt{ijcnn1};  $\tau = 3.55 \times 10^{-6}$, $\sigma = 10^3$ for \texttt{a9a}; $\tau = 3\times 10^{-6}$, $\sigma = 10^3$ for \texttt{mushrooms}, and $\phi = \frac{1 + \sqrt{5}}{2}$.
  \item \textbf{P-GRPDA}: $\tau_0 = 9 \times 10^{-3}$, $\beta = 3.5\times 10^{6}$ for \texttt{ijcnn1}; $\beta = 3\times 10^6$ for \texttt{a9a}; $\beta = 1.93 \times 10^4$ for \texttt{mushrooms}, and $\psi = 1.76$, $\mu = 0.79$, $\mu' = 0.26$.
  \item \textbf{aEGRPDA}: $\tau_0 = 10^{-3}$, $\beta = 4.642 \times 10^{3}$ for \texttt{ijcnn1}, $\beta = 7\times 10^6$ for \texttt{a9a} and $\beta=2.7\times10^6$ for \texttt{mushrooms}.
\end{itemize}
In a similar manner, we select the following parameters for \textbf{Setting 2}.

\begin{itemize}
  \item \textbf{Condat-V\~u}: $\tau = 3\times 10^{-5}$, $\sigma = 10^4$ for \texttt{ijcnn1}; $\tau = 10^{-6}$, $\sigma = 3\times 10^4$ for \texttt{a9a}; $\tau = 10^{-6}$, $\sigma = 10^4$ for \texttt{mushrooms}  .
  \item \textbf{E-GRPDA}: $\tau = 3.55 \times 10^{-6}$, $\sigma = 10^3$ for \texttt{ijcnn1};  $\tau = 3\times 10^{-6}$, $\sigma = 10^3$ for \texttt{a9a} and \texttt{mushrooms}, and $\phi = \frac{1 + \sqrt{5}}{2}$.
  \item \textbf{P-GRPDA}: $\tau_0 = 3 \times 10^{-4}$, $\beta = 2.7\times 10^{6}$ for \texttt{ijcnn1} and \texttt{mushrooms}; $\beta = 2\times 10^6$ for \texttt{a9a}, and $\psi = 1.76$, $\mu = 0.79$, $\mu' = 0.26$.
  \item \textbf{aEGRPDA}: $\tau_0 = 10^{-3}$, $\beta = 2 \times 10^{5}$ for \texttt{ijcnn1}; $\beta = 2\times 10^7$ for \texttt{a9a} and \texttt{mushrooms}.
\end{itemize}

The numerical results for the two settings are reported in Figures~\ref{fig:slr_comparison} and~\ref{fig:slr_all}.  
From Figure~\ref{fig:slr_all}, we observe that when $K=I$, aGRAAL takes a clear advantage and consistently outperforms the other methods, in particular P\text{-}GRPDA (Algorithm~\ref{algorithm 1}) and aEGRPDA (Algorithm~\ref{algorithm 2}), in terms of the fastest decay of the objective gap $F(x_n)-F^\ast$. In contrast, Figure~\ref{fig:slr_comparison} indicates that aEGRPDA achieves the most rapid decrease of $F(x_n)-F^\ast$ throughout the run, with P\text{-}GRPDA typically ranking second and often closely tracking aEGRPDA during the initial iterations. Meanwhile, E\text{-}GRPDA and Condat--V\~u remain reliable but converge more slowly under the same conditions, reflecting their more conservative stepsize regimes.



\subsection{GraphNet Problem}
Consider the GraphNet \cite{Grosenick2013GraphNet} regularized least–squares problem
\begin{equation}\label{eq:graphnet}
    \min_{x\in\mathbb{R}^{n}}
        F(x)\coloneqq \lambda_{1}\,\|x\|_{1}+\frac{\lambda_{2}}{2}\,x^{\top} Wx~+\frac{1}{2m}\,\|Ax-b\|_2^{2}.
\end{equation}
where $A\in\mathbb{R}^{m\times n}$ is a design matrix, $W$ is a sparse graph, $b\in\mathbb{R}^m$, and $\lambda_{1},\lambda_{2}\!>\!0$ are regularization parameters. When $W=I_{n}$, \eqref{eq:graphnet} reduces to the elastic net \cite{zou2005regularization}, see \cite{Grosenick2013GraphNet}. Furthermore, the middle term $x^{T}Wx$ simplifies to $\|Dx\|^2_2$ with $W=D^{T}D$; the graph Laplacian, where $D\in\mathbb{R}^{|E|\times n}$ is the incidence matrix of a graph $G=(V,E)$. With this choice of $W$, comparing \eqref{eq:graphnet} with \eqref{1.1} yields $f(\cdot)=\lambda_1 \lVert \cdot \rVert_1$, $g(\cdot)=\frac{\lambda_2}{2}\lVert \cdot\rVert^2$, $K=D$, and $h(\cdot)=\frac{1}{2m}\|A(\cdot)-b\|^2_2$. 

We generate the following datasets for this experiment.
\begin{figure}[htbp]
  \centering
  \begin{subfigure}[b]{0.31\textwidth}
    \includegraphics[width=\textwidth]{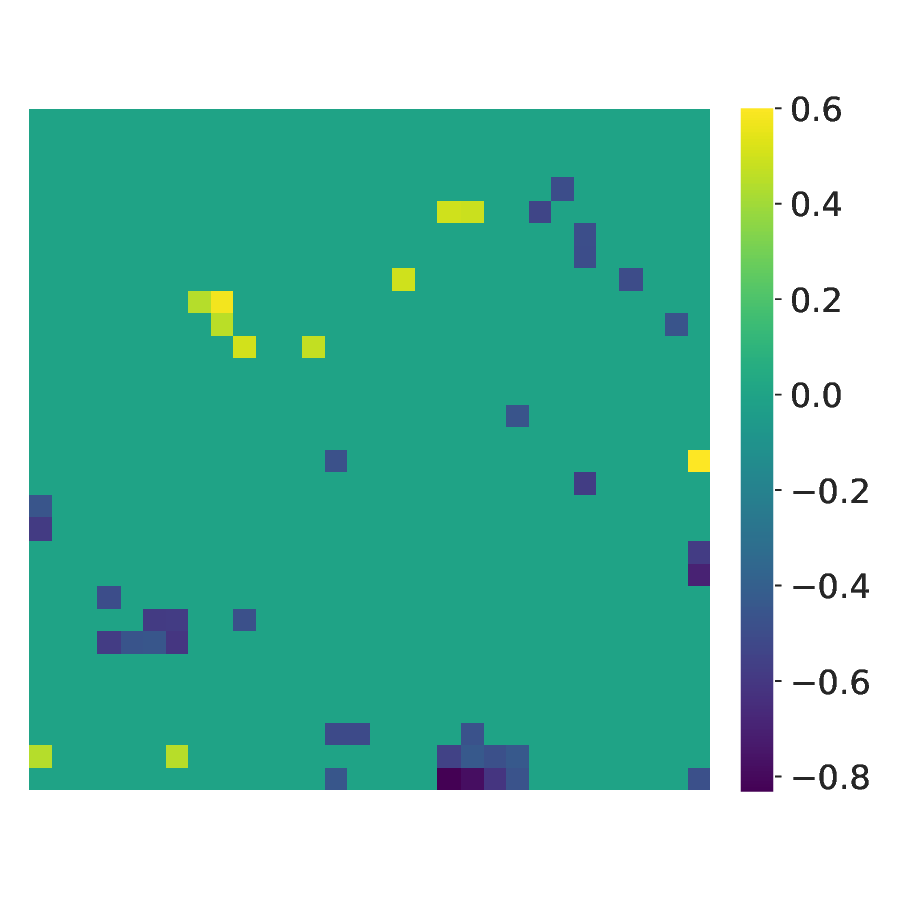}
    \caption{True GraphNet signal}
  \end{subfigure}
  \begin{subfigure}[b]{0.31\textwidth}
    \includegraphics[width=\textwidth]{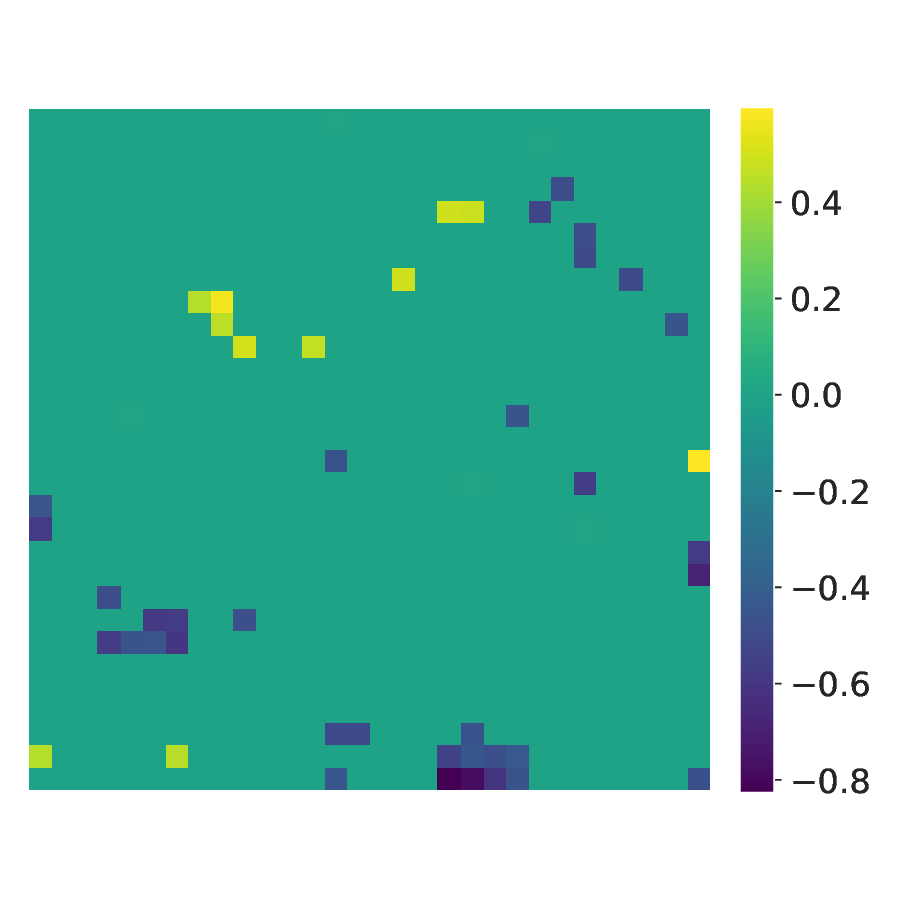}
    \caption{aEGRPDA}
  \end{subfigure}
  \begin{subfigure}[b]{0.31\textwidth}
    \includegraphics[width=\textwidth]{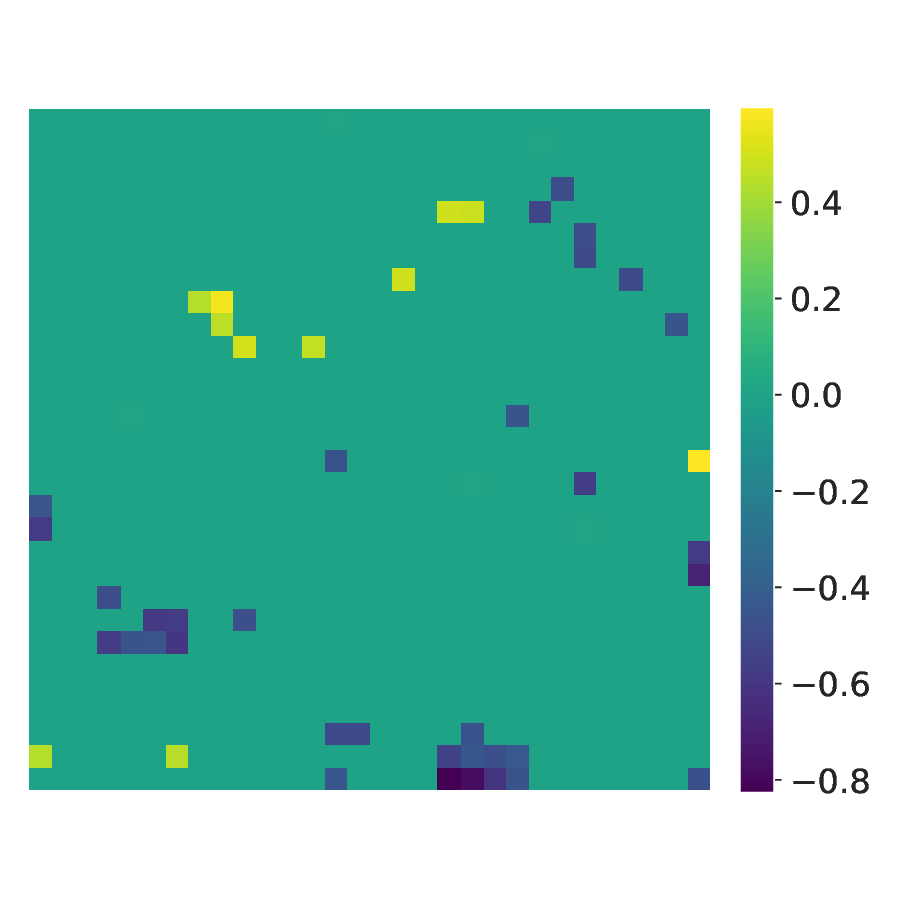}
    \caption{P‑GRPDA}
  \end{subfigure}

  \vspace{.1mm}

  \begin{subfigure}[b]{0.31\textwidth}
    \includegraphics[width=\textwidth]{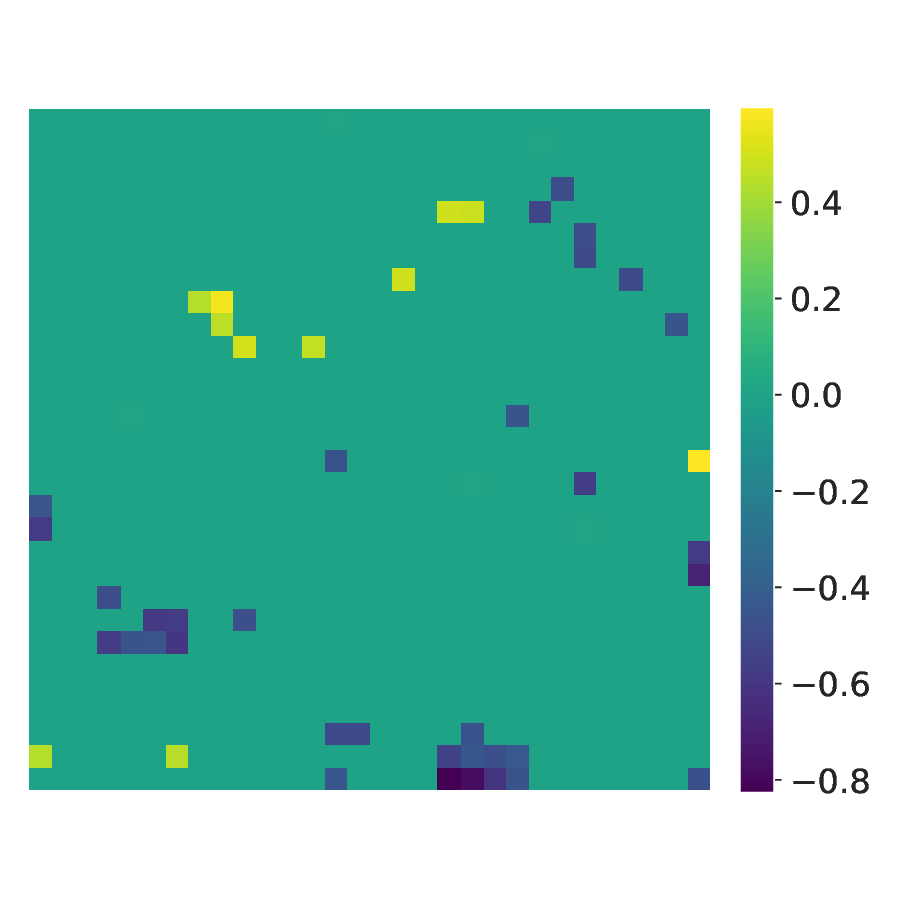}
    \caption{E‑GRPDA}
  \end{subfigure}
  \begin{subfigure}[b]{0.31\textwidth}
    \includegraphics[width=\textwidth]{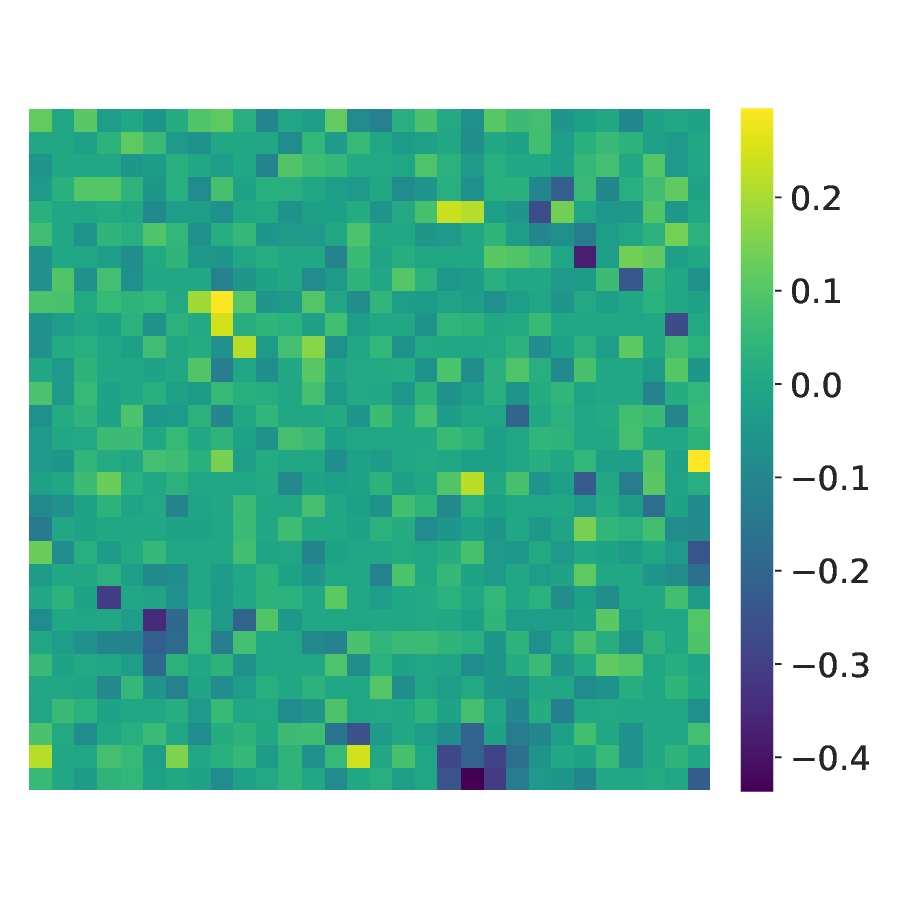}
    \caption{aGRAAL}
  \end{subfigure}
  \begin{subfigure}[b]{0.31\textwidth}
    \includegraphics[width=\textwidth]{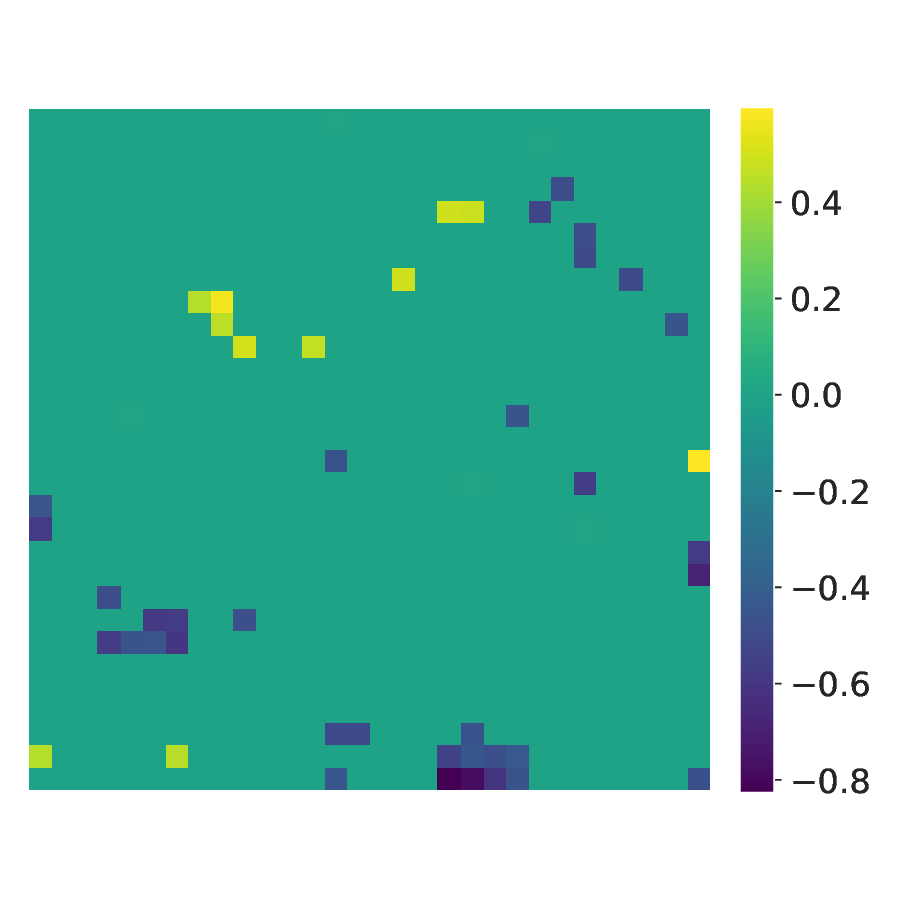}
    \caption{Condat--V\~{u} }
  \end{subfigure}

\caption{Recovered GraphNet signals on a $30 \times 30$ grid with $m = 300$. (a) True signal; (b) Recovered by aEGRPDA; (c) Recovered by P‑GRPDA; (d) Recovered by E‑GRPDA; (e) Recovered by aGRAAL; (f) Recovered by Condat--V\~{u} .}\label{fig:graphnet_recovery}
\end{figure}
Let $G=(V,E)$ be a graph on a two-dimensional grid of size $n_{v_1}\times n_{v_2}$ (so $|V|=n=n_{v_1}n_{v_2}$). We connect every horizontal and vertical neighbour pair with an edge and the incidence matrix $D\in\mathbb{R}^{|E|\times n}$ is then constructued for each pair $e=(i,j)$, by setting $(Dx)_{e}=x_i-x_j$, see \cite{shuman2013emerging} for an overview of graph constructions and operators.
We generate entries of $x_0\in\mathbb{R}^n$ from $\mathcal{N}(0,1)$, and obtain a smoothed signal $x_{\mathrm{smth}}$ by solving the sparse linear system
$$
  (I_n + \alpha W)\,x_{\mathrm{smth}} = x_0,~~\text{where}~~ W := D^\top D,~~ \alpha>0.
$$
Note that the operator $(I+\alpha W)^{-1}$ is called the Tikhonov smoothing filter \cite{li2020graph}, which eliminates the high graph‑frequencies and preserves the low ones. The approximate solution of the above sparse linear system (one can use $\alpha=2$ or $3$) can be obtained using the conjugate gradient method by taking a maximum of $1000$ iterations. After obtaining $x_{\mathrm{smth}}$, we construct the ground truth (smooth + sparse) $x^\star$
by keeping only the largest $5\%$ of its entries and setting the rest to zero. Specifically, one can take $k=\lfloor 0.05\,n\rfloor$ and let
$$
  c := \left|x_{\mathrm{smth}}\right|_{(k)},
$$
which denotes the $k$-th order statistic of the set $\{|(x_{\mathrm{smth}})_i|\}_{i=1}^n$ that is the $k$-th largest absolute value, see \cite{donoho2006compressed,david2004order, tibshirani2005sparsity} for more details.
Then, for $i=1,2,\ldots,n$, we construct the true signal $x^\star$ as
\begin{equation*}
  x^\star_i =
  \begin{cases}
    (x_{\mathrm{smth}})_i, & \text{if } |(x_{\mathrm{smth}})_i|\ge c,\\[0.2em]
    0,          & \text{otherwise}.
  \end{cases}
\end{equation*}
\begin{figure}[htbp]
  \centering
  \begin{subfigure}[b]{0.40\textwidth}
    \includegraphics[width=\textwidth]{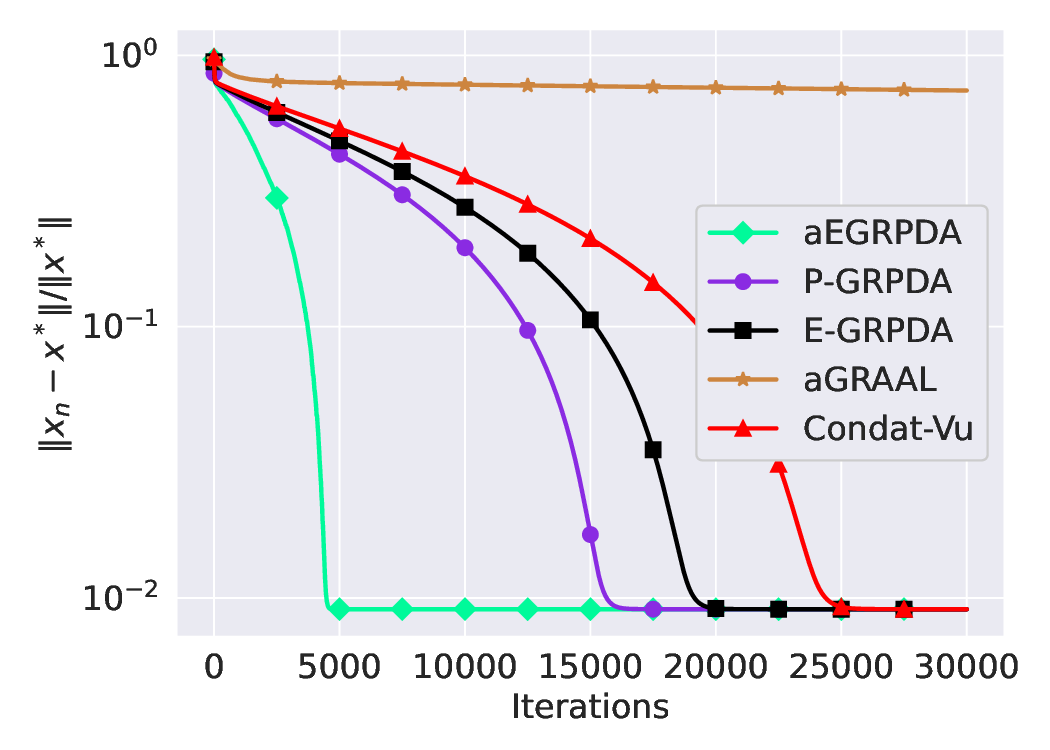}
    \caption{Relative error ($m=300$)}
    \label{fig:graphnet_err_300}
  \end{subfigure}\hfill
  \begin{subfigure}[b]{0.40\textwidth}  
    \includegraphics[width=\textwidth]{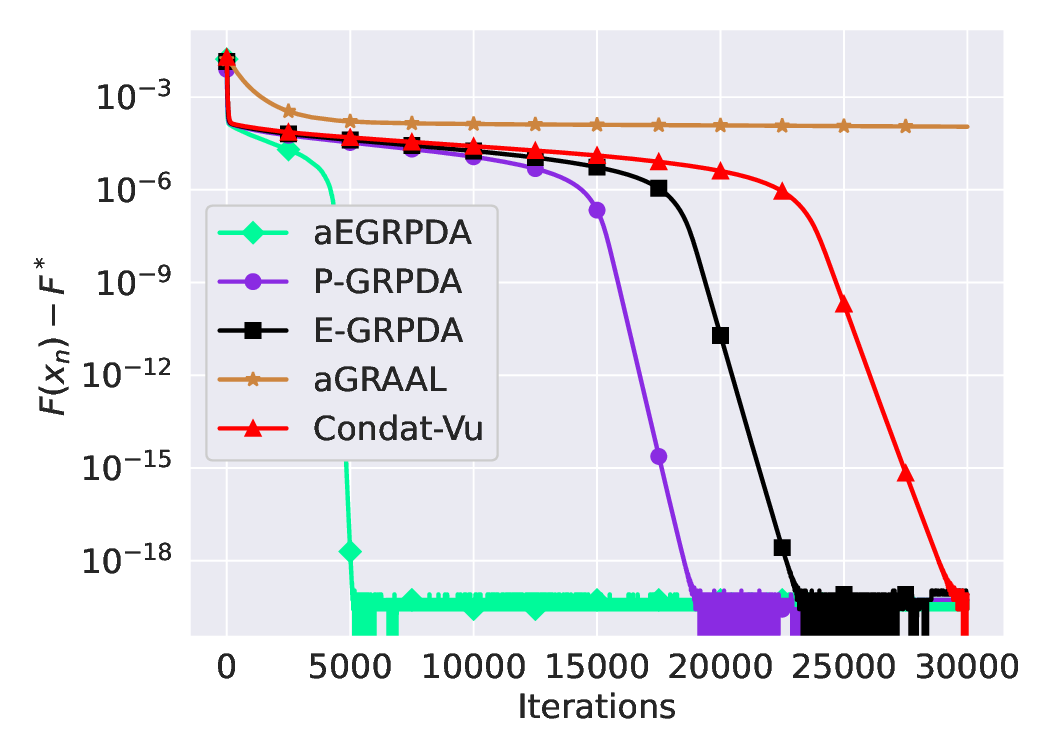}
    \caption{Convergence plot ($m=300$)}
    \label{fig:graphnet_gap_300}
  \end{subfigure}

  \vspace{.1em}

  \begin{subfigure}[b]{0.40\textwidth}
    \includegraphics[width=\textwidth]{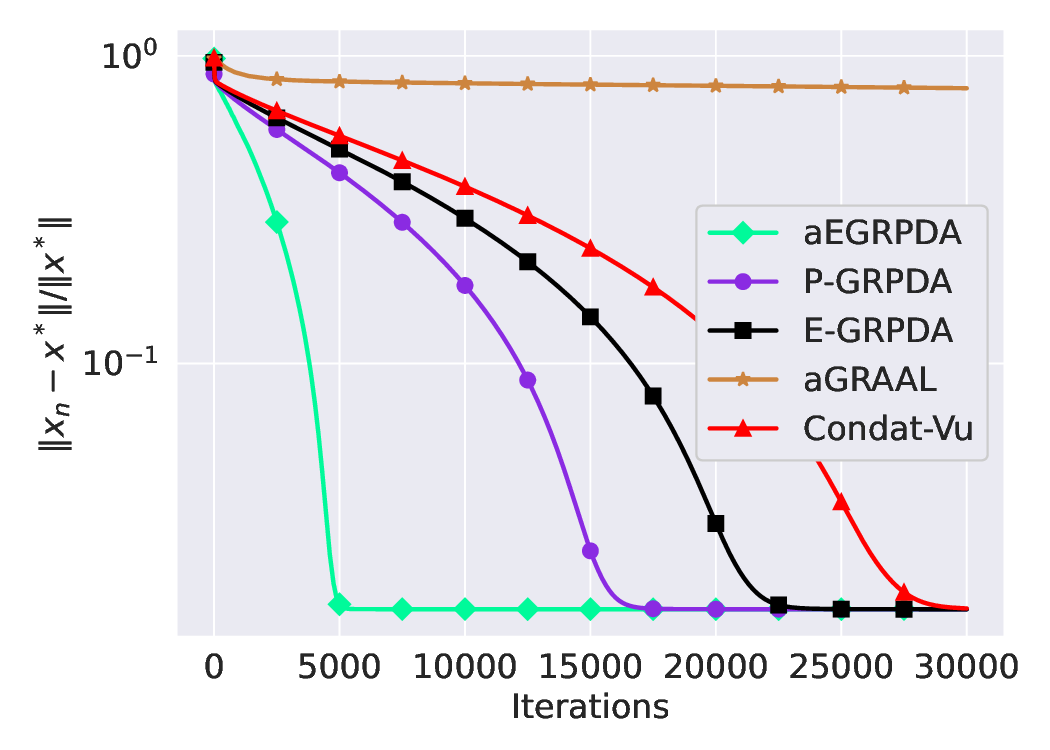}
    \caption{Relative error ($m=500$)}
    \label{fig:graphnet_err_500}
  \end{subfigure}\hfill
  \begin{subfigure}[b]{0.40\textwidth}
    \includegraphics[width=\textwidth]{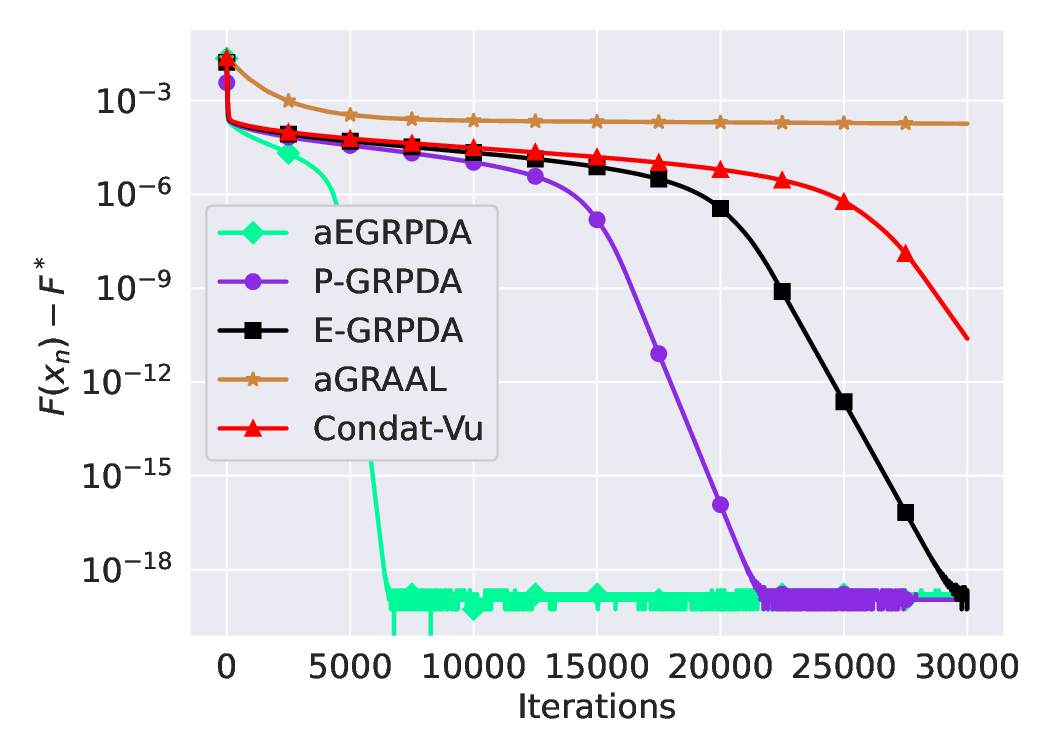}
    \caption{Convergence plot ($m=500$)}
    \label{fig:graphnet_gap_500}
  \end{subfigure}

  \caption{Convergence plot of the GraphNet by different algorithms: (a)--(b) on $30\times30$ grid (with $m=300$);
    (c)--(d) on $40\times40$ grid (with $m=500$). Here (a) and (c) represent the relative error per iteration; (b) and (d) represent the objective‐value gap per iteration.}
  \label{fig:graphnet_convergence}
\end{figure}
We generate the enrtries of $A\in\mathbb{R}^{m\times n}$ randomly from $\mathcal{N(}0,\frac{1}{m})$ and set $b= Ax^{\star}+\epsilon$, where $\epsilon$ is the Gaussian noise drawn from $\mathcal{N}(0,0.01)$. We set the regularization parameters as $\lambda_{1} = 6.64\times 10^{-6}$, $\lambda_{2} = 10^{-6}$, and initialized the variables with $x_0 = y_0 = 0$. A moderate parameter tuning was performed to select the initial parameters for each algorithm.
We ran short trajectories on a coarse grid, then narrowed to a local neighborhood and selected values that achieved stable progress and the fastest early-iteration decay of the objective gap. The effective ranges of the parameters were:
Condat--V\~u with $\tau\!\in[5,8]$, $\sigma\!\in[1.5\times10^{-3},\,3\times10^{-2}]$;
E-GRPDA (with $\phi=\tfrac{1+\sqrt{5}}{2}$) with $\tau\!\in[20,30]$, $\sigma\!\in[1.5\times10^{-3},\,3\times10^{-2}]$;
P-GRPDA (with $\phi=\frac{1 + \sqrt{5}}{2}$), $\beta\!\in\left[7.5\times 10^{-5},\ 10^{-3}\right]$; and aEGRPDA with $\beta\!\in \left[3\times 10^{-7},\ 3.75\times 10^{-3}\right]$.
As stated earlier, we took $\tau_0$ close to the $\tau$ used in E-GRPDA, or slightly higher than the $\tau$ used in Condat-V\~u, to provide a good initial kick. We then settled on the following values:
\begin{itemize}
  \item \textbf{Condat--V\~{u} :} $\tau = 7.720572$, $\sigma = 1.9458604 \times 10^{-2}$.
  \item \textbf{E-GRPDA:} $\tau = 25.735240623$, $\sigma = 1.9458604 \times 10^{-2}$, and $\phi = \frac{1 + \sqrt{5}}{2}$.
  \item \textbf{P-GRPDA:} $\tau_0=10$, $\beta = 10^{-4}$, $\psi = \frac{1 + \sqrt{5}}{2}$, $\mu = 0.81$, $\mu' = 0.40$.
  \item \textbf{aEGRPDA:} $\tau_0=10$, $\beta = 10^{-6}$.
\end{itemize}
Figure~\ref{fig:graphnet_recovery} collects the recovered signal $x^*$ obtained by various algorithms, while Figure~\ref{fig:graphnet_convergence} shows the corresponding convergence behavior. We observe that all the algorithms nearly recover $x^*$; however, aGRAAL fails to accurately reconstruct the nonzero elements of $x^*$. This is consistent with the convergence plots in Figure~\ref{fig:graphnet_convergence}, which indicate that aGRAAL requires significantly more iterations to achieve competitive performance.



\subsection{Fused LASSO}
The Fused Lasso problem is written as
\begin{equation}
\min_{x\in\mathbb{R}^n}F(x) := \lambda_1 \lVert x \rVert_1 + \lambda_2\|Dx\|_1 + \frac{1}{2} \lVert Ax - b \rVert^2,
\end{equation} 
where the difference matrix  $D\in\mathbb{R}^{(n-1) \times n}$ is given by  
$$K=D=  
\begin{bmatrix}
-1 & 1 & 0 & \cdots & 0 \\
0 & -1 & 1 & \cdots & 0 \\
\vdots & \vdots & \vdots & \ddots & \vdots \\
0 & 0 & 0 & \cdots & -1 & 1
\end{bmatrix}.$$

\begin{figure}[htbp]
  \centering
  \begin{subfigure}[b]{0.32\textwidth}
    \includegraphics[width=\textwidth]{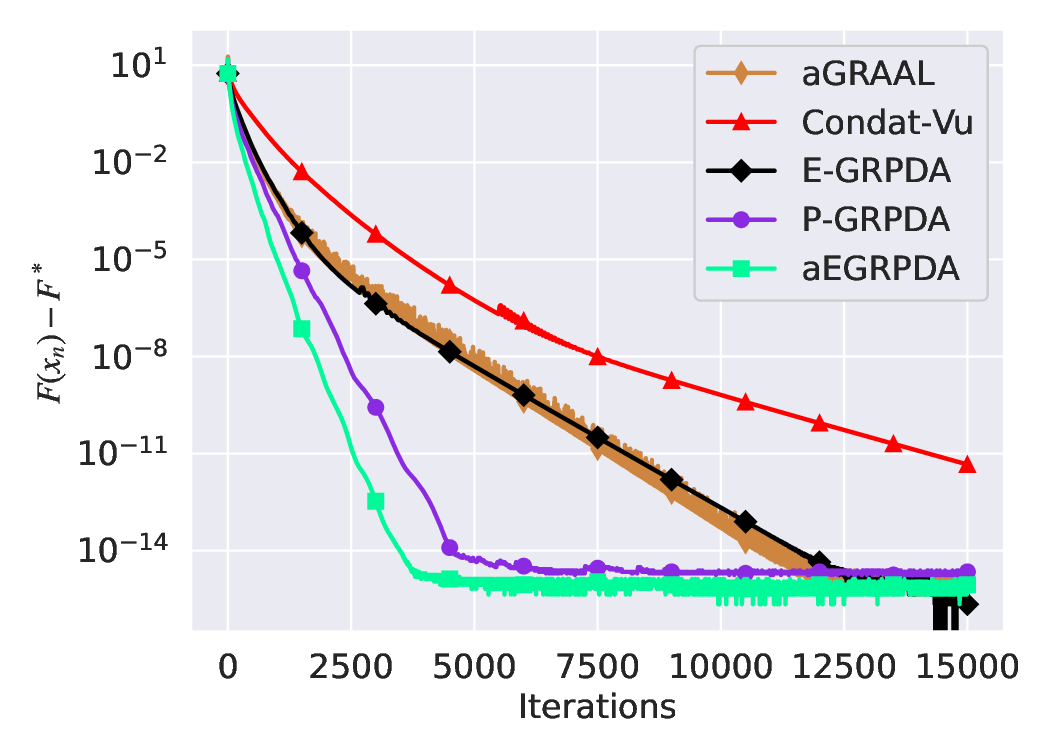}
    \subcaption{Convergence plot}
    \label{fig:fl500x1000_conv}
  \end{subfigure}\hfill
  \begin{subfigure}[b]{0.32\textwidth}
    \includegraphics[width=\textwidth]{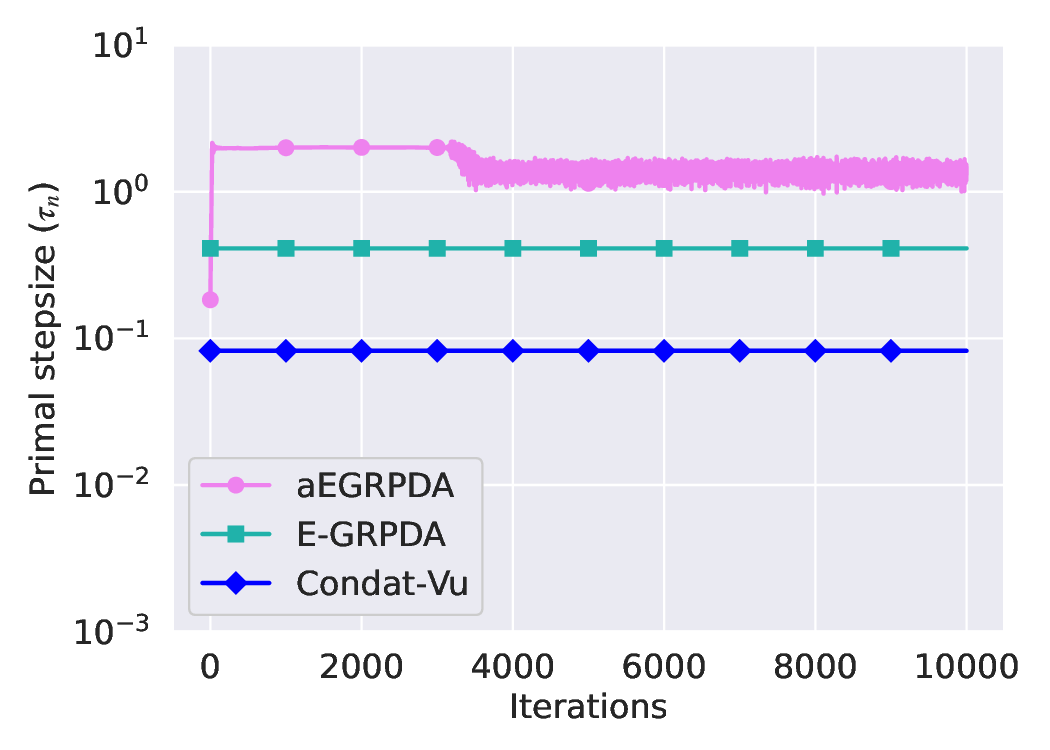}
    \subcaption{Primal steps ($\tau_n$)}
    \label{fig:fl500x1000_primal}
  \end{subfigure}\hfill
  \begin{subfigure}[b]{0.32\textwidth}
    \includegraphics[width=\textwidth]{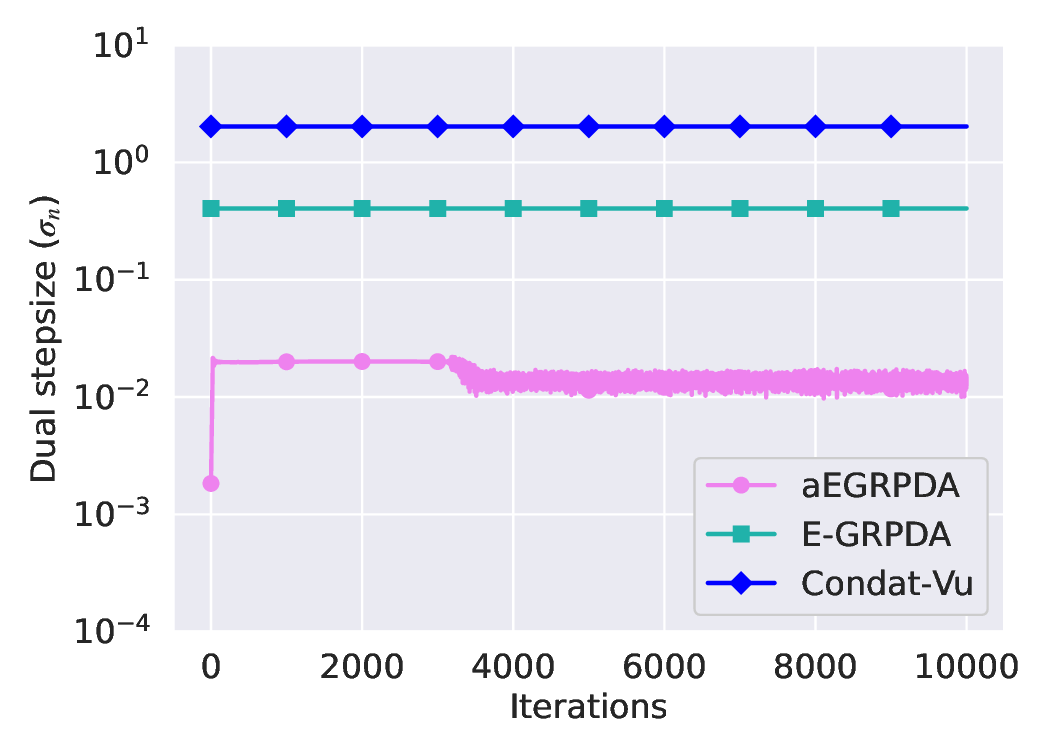}
    \subcaption{Dual steps ($\sigma_n$)}
    \label{fig:fl500x1000_dual}
  \end{subfigure}

  \vspace{.3em}

  \begin{subfigure}[b]{0.32\textwidth}
    \includegraphics[width=\textwidth]{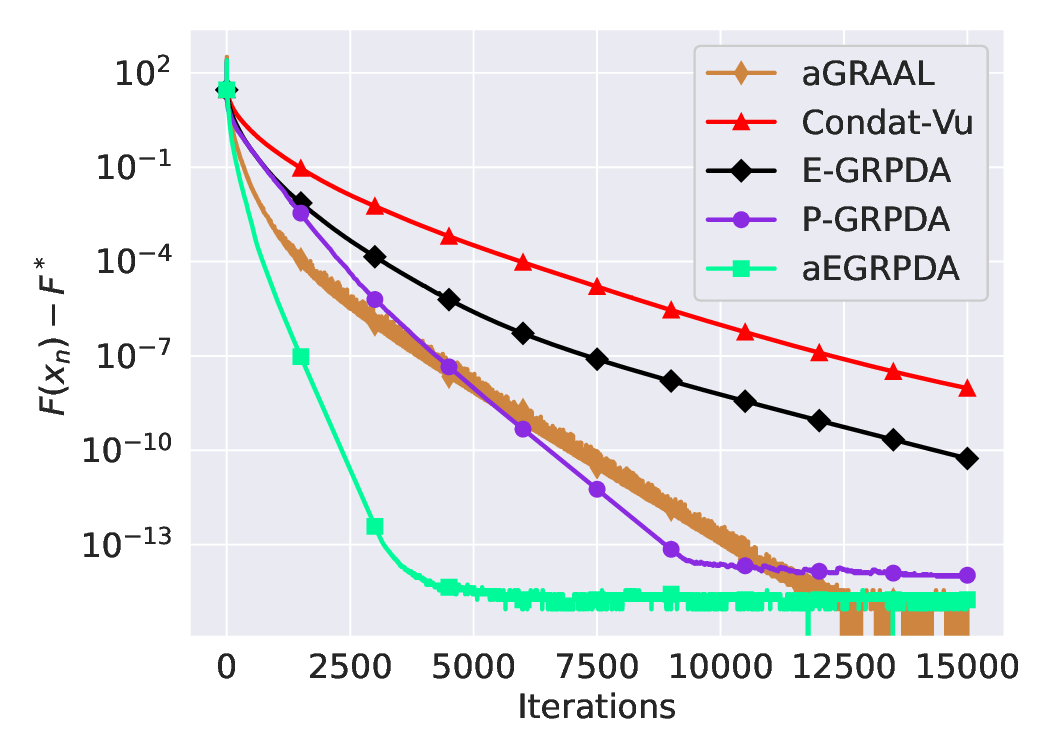}
    \subcaption{Convergence plot}
    \label{fig:fl1000x2000_conv}
  \end{subfigure}\hfill
  \begin{subfigure}[b]{0.32\textwidth}
    \includegraphics[width=\textwidth]{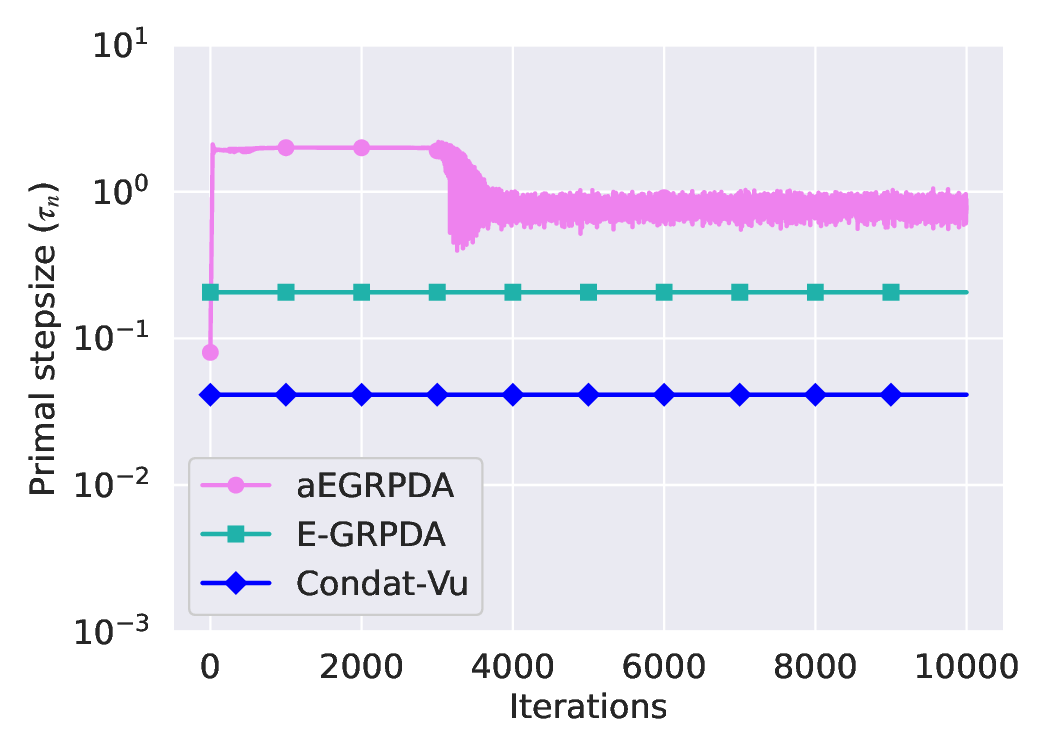}
    \subcaption{Primal steps ($\tau_n$)}
    \label{fig:fl1000x2000_primal}
  \end{subfigure}\hfill
  \begin{subfigure}[b]{0.32\textwidth}
    \includegraphics[width=\textwidth]{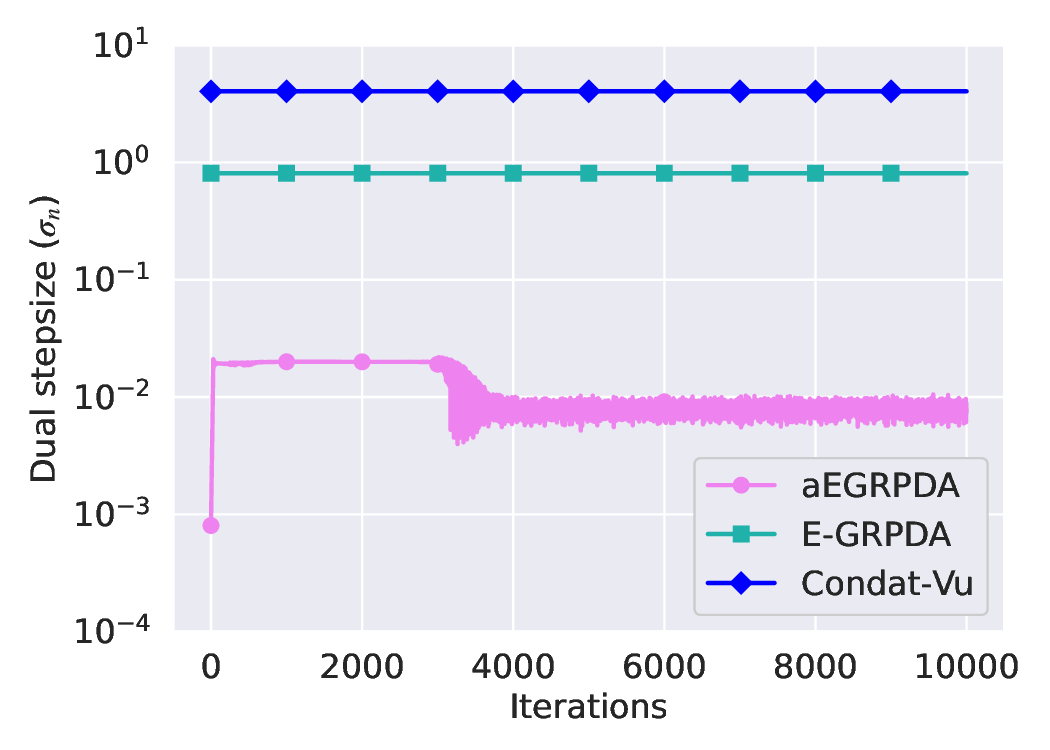}
    \subcaption{Dual steps ($\sigma_n$)}
    \label{fig:fl1000x2000_dual}
  \end{subfigure}

\caption{Convergence results for the Fused Lasso problem with two different problem sizes: $(m,n) = (500, 1000)$ (top row) and $(m,n) = (1000, 2000)$ (bottom row). Each row shows (from left to right) the convergence plot and the comparisons of primal stepsizes $(\tau_n)$ and dual stepsizes $(\sigma_n)$.}
\label{fig:fused_lasso_all}
\end{figure}

\noindent
In this experiment, the elements of $A = (A_{ij})$ and the elements of the additional noise vector $\omega\in\mathbb{R}^m$ are sampled from a normal distribution  $\mathcal{N}(0,0.01)$. The vector $b$ is then constructed as $b=Ax^*+\omega$, where the entries of $x^*$ are independently generated from $\mathcal{N}(0,1)$. We set $\lambda_1 = 0.001$ and $\lambda_2= 0.03$. For the data-fitting term $h(x) = \frac{1}{2}\|Ax-b\|^2$, its gradient is $\nabla h(x) = A^*(Ax-b)$. Therefore, the Lipschitz constant $\bar L$ is nothing but $\|A^*A\|$, which corresponds to the largest eigenvalue of $A^*A$. After a parameter sweep, we are satisfied with the following initial values.
\begin{itemize}
    \item \textbf{Condat--V\~u}: $\tau = 2.353 \times 10^{-2}$,  $\sigma = 7.100 \times 10^{-2}$.
    \item \textbf{E-GRPDA}: $\tau = 1.177 \times 10^{-1}$, $\sigma = 1.420 \times 10^{-2}$, $\psi = \frac{1 + \sqrt{5}}{2}$.
    \item \textbf{P-GRPDA}: $\tau_0 = 10$, $\beta = 10^{-2}$, $\psi = 1.70$, $\mu = 0.79$, $\mu' = 0.266$.
    \item \textbf{aEGRPDA}: $\tau_0=10$, $\beta = 10^{-2}$.
\end{itemize}
Here, we start with $\tau_0=10$, which may also be set to other values (e.g., the $\tau$ used in E\mbox{-}GRPDA). For $(\tau,\sigma)$, we sweep $\tau\in[10^{-3},\,2\times10^{-1}]$ and $\sigma\in[10^{-3},\,10^{-1}]$, followed by a local refinement around the best candidates. The parameter $\beta$ is then selected via a grid search over $[2\times10^{-3},\,10^{-1}]$. In Figure~\ref{fig:fused_lasso_all}, we observe that the primal–dual stepsizes of aEGRPDA remain essentially flat up to about ($\approx$ 3800 iterations), after which they oscillate within a narrow band. In terms of objective gap, aEGRPDA converges fastest and consistently outperforms the other methods, with P\text{-}GRPDA the next best.
\subsection{Image Inpainting}
Consider the TV--$L_2$ image inpainting problem \cite{chambolle2011first}
\begin{equation}\label{eq:inpaint}
    \min_{x\in\mathbb{R}^{m\times n}} F(x)\coloneqq\frac12\|M\odot(x-b)\|_2^2+\lambda\|\nabla x\|_{2,1},
\end{equation}
where $b\in\mathbb{R}^{m\times n}$ is the observed (generally noisy or damaged) image, $M\in\{0,1\}^{m\times n}$ is a binary mask (that is 1 on known pixels and 0 on missing), $\lambda>0$, and $\odot$ denotes the \emph{Hadamard product}. Comparing \eqref{eq:inpaint} with \eqref{1.1} gives $f=0$, $g = \lambda\|\cdot\|_{2,1}$, and $h(x)= \frac12\|M\odot(x-b)\|_2^2$. We use isotropic TV \cite{condat2017discrete, RudinOsherFatemi1992} with the \emph{discrete gradient operator} $K=\nabla:\mathbb{R}^{m\times n}\to\mathbb{R}^{m\times n\times 2}$ given by 
\begin{equation*}
  \|Kx\|_{2,1}=\|\nabla x\|_{2,1}:=\sum_{i,j}\sqrt{(\nabla_x x_{i,j})^2+(\nabla_y x_{i,j})^2},
\end{equation*}
with adjoint $K^\top=\nabla^\top=-\mathrm{div}$, and it is known that $\|K\|\le\sqrt{8}$, see \cite{chambolle2011first}. For $g(u)= \lambda\|u\|_{2,1}=\lambda\sum_{i,j}\|u_{i,j}\|_2$, we have $g^*(p)=\delta_{\{\|p_{i,j}\|_2\le\lambda\}}$, and hence the proximal conjugate of $g^*$ is the pixelwise Euclidean projection
\begin{equation*}
  \operatorname{prox}_{\sigma g^*}(q)=\Pi_{\{\|\,\cdot\,\|_2\le\lambda\}}(q)
  =\frac{q}{\max\{1,\ \|q\|_2/\lambda\}},
\end{equation*}
where $\Pi_C$ denotes the projection onto $C$ and $\|\cdot\|_2$ denotes the $\ell_2$ norm.
\begin{figure}[htbp]
  \centering
  \begin{minipage}{0.20\textwidth}
    \centering \textbf{Condat--V\~{u} }\par\medskip
    \includegraphics[width=\linewidth]{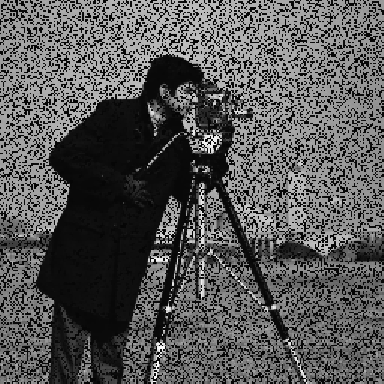}\\[-0.2ex]
    {\footnotesize $\lambda=0.0001$}\par\medskip
    \includegraphics[width=\linewidth]{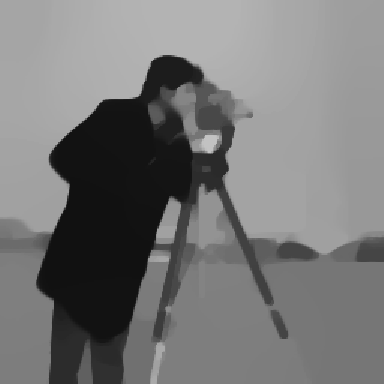}\\[-0.2ex]
    {\footnotesize $\lambda=0.03$}\par\medskip
    \includegraphics[width=\linewidth]{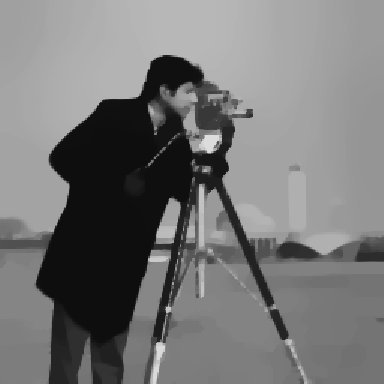}\\[-0.2ex]
    {\footnotesize $\lambda=0.1$}\par\medskip
    \includegraphics[width=\linewidth]{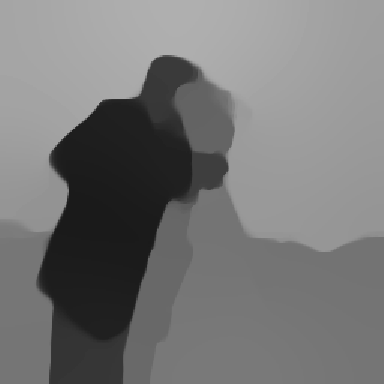}\\[-0.2ex]
    {\footnotesize $\lambda=1$}
  \end{minipage}
  \hfill
  \begin{minipage}{0.20\textwidth}
    \centering \textbf{E-GRPDA}\par\medskip
    \includegraphics[width=\linewidth]{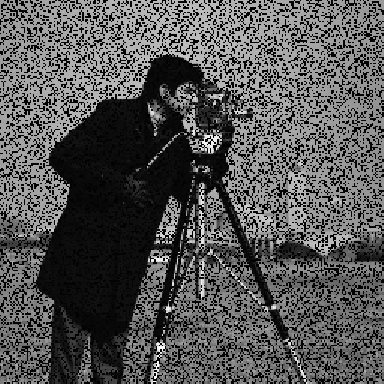}\\[-0.2ex]
    {\footnotesize $\lambda=0.0001$}\par\medskip
    \includegraphics[width=\linewidth]{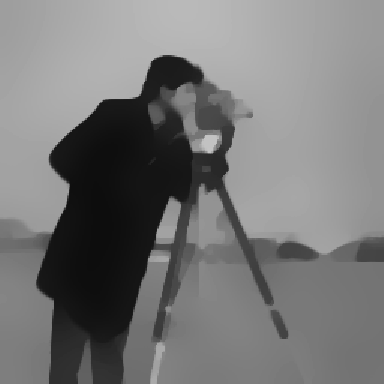}\\[-0.2ex]
    {\footnotesize $\lambda=0.03$}\par\medskip
    \includegraphics[width=\linewidth]{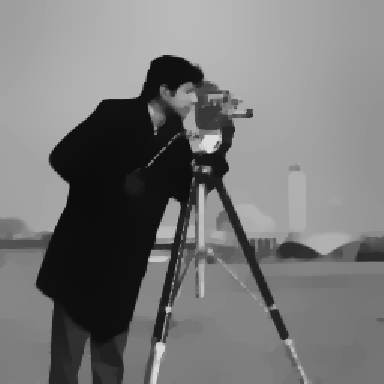}\\[-0.2ex]
    {\footnotesize $\lambda=0.1$}\par\medskip
    \includegraphics[width=\linewidth]{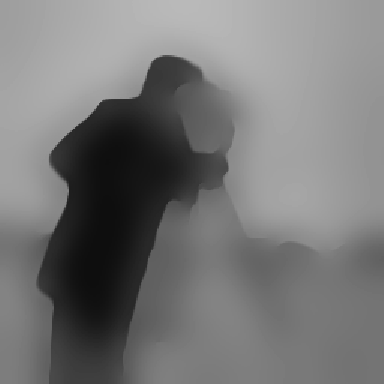}\\[-0.2ex]
    {\footnotesize $\lambda=1$}
  \end{minipage}
  \hfill
  \begin{minipage}{0.20\textwidth}
    \centering \textbf{P-GRPDA}\par\medskip
    \includegraphics[width=\linewidth]{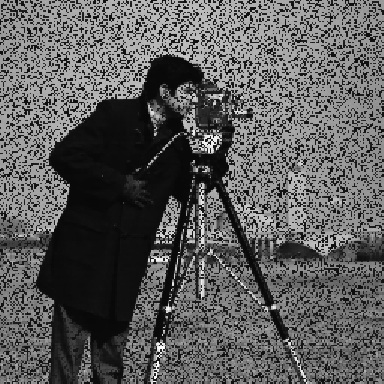}\\[-0.2ex]
    {\footnotesize $\lambda=0.0001$}\par\medskip
    \includegraphics[width=\linewidth]{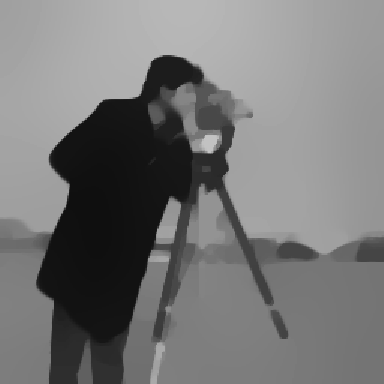}\\[-0.2ex]
    {\footnotesize $\lambda=0.03$}\par\medskip
    \includegraphics[width=\linewidth]{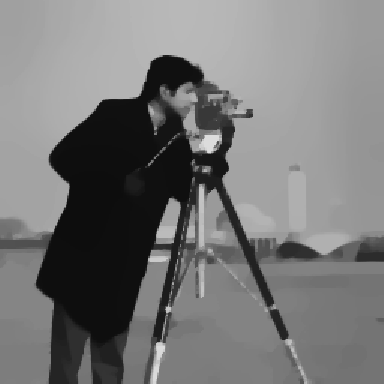}\\[-0.2ex]
    {\footnotesize $\lambda=0.1$}\par\medskip
    \includegraphics[width=\linewidth]{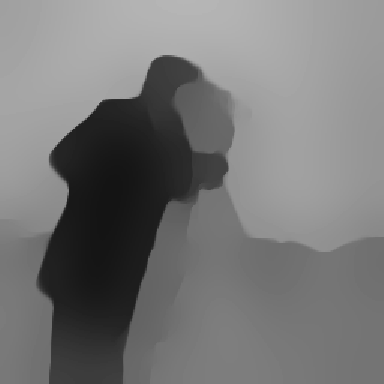}\\[-0.2ex]
    {\footnotesize $\lambda=1$}
  \end{minipage}
  \hfill
  \begin{minipage}{0.20\textwidth}
    \centering \textbf{aEGRPDA}\par\medskip
    \includegraphics[width=\linewidth]{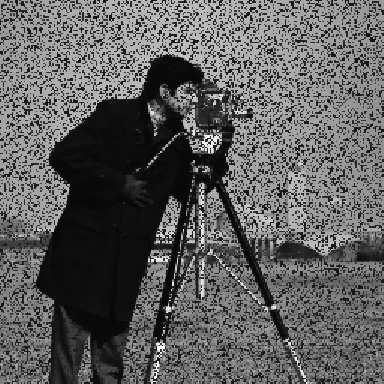}\\[-0.2ex]
    {\footnotesize $\lambda=0.0001$}\par\medskip
    \includegraphics[width=\linewidth]{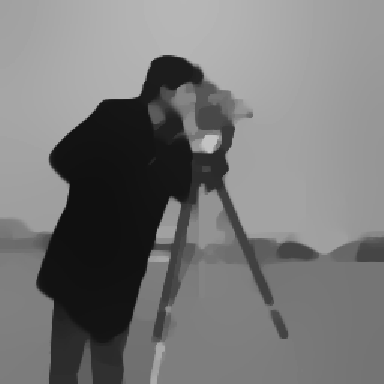}\\[-0.2ex]
    {\footnotesize $\lambda=0.03$}\par\medskip
    \includegraphics[width=\linewidth]{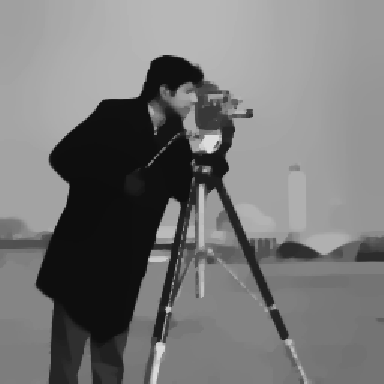}\\[-0.2ex]
    {\footnotesize $\lambda=0.1$}\par\medskip
    \includegraphics[width=\linewidth]{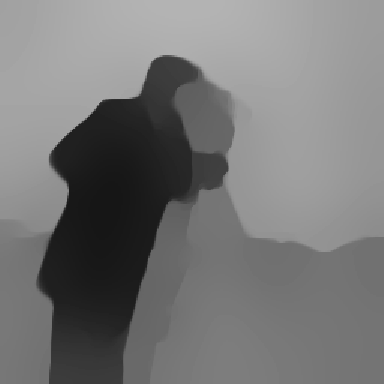}\\[-0.2ex]
    {\footnotesize $\lambda=1$}
  \end{minipage}

\caption{Effect of $\lambda$ on the TV--$L_2$ problem. Each column corresponds to a specific algorithm (Condat--V\~{u}, E-GRPDA, P-GRPDA, and aEGRPDA), while each row shows the reconstructed image obtained with a different value of $\lambda \in \{0.0001,~ 0.03,~ 0.1,~ 1\}$.}\label{fig:lambda-columns}
\end{figure}
 We scaled the grayscale test image to $[0,1]$ and then removed $30\%$ of the pixels uniformly at random to form the mask $M$ and set $b=M\odot x_{\mathrm{true}}$. We applied the same moderate tuning protocol as in the other experiments: runs on a coarse grid followed by refinement around settings that produced stable progress and the fastest early-iteration decay of the objective gap. We fixed $\lambda=10^{-2}$ and used the following parameters:
\begin{itemize}
  \item \textbf{Condat--V\~{u}:} $\tau = 4 \times 10^{-2}$, \ $\sigma = 10^{-1}$.
  \item \textbf{E\mbox{-}GRPDA:} $\tau = 4 \times 10^{-2}$, \ $\sigma = 10^{-2}$, \ $\phi = \tfrac{1 + \sqrt{5}}{2}$.
  \item \textbf{P\mbox{-}GRPDA:} $\tau_0=1$, \ $\beta = 0.1$, \ $\psi = 1.70$, \ $\mu = 0.79$, \ $\mu' = 0.26$.
  \item \textbf{aEGRPDA:} $\tau_0=1$, \ $\beta = 0.1$.
\end{itemize}
\begin{figure*}[htbp]
  \centering
  \begin{subfigure}{0.23\textwidth}
    \includegraphics[width=\linewidth]{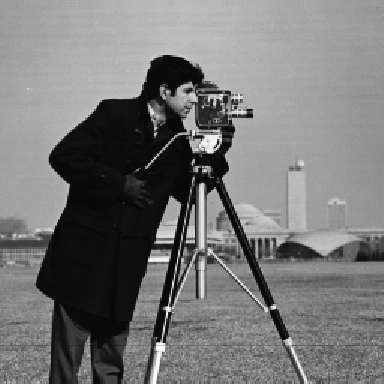}
    \caption{True image}
    \label{fig:recon-gt}
  \end{subfigure}
  \hfill
  \begin{subfigure}{0.23\textwidth}
    \includegraphics[width=\linewidth]{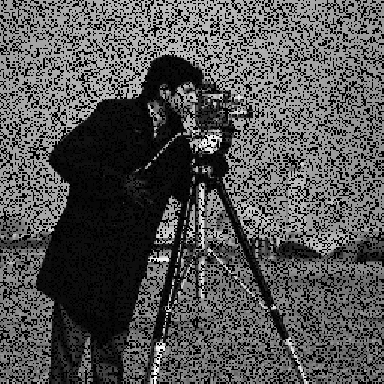}
    \caption{Damaged}
    \label{fig:recon-obs}
  \end{subfigure}
  \hfill
  \begin{subfigure}{0.23\textwidth}
    \includegraphics[width=\linewidth]{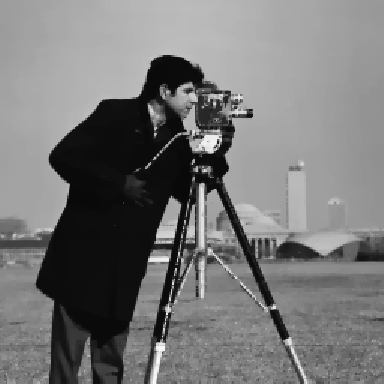}
    \caption{aEGRPDA}
    \label{fig:recon-full}
  \end{subfigure}

  \vspace{0.6em}

  \begin{subfigure}{0.23\textwidth}
    \includegraphics[width=\linewidth]{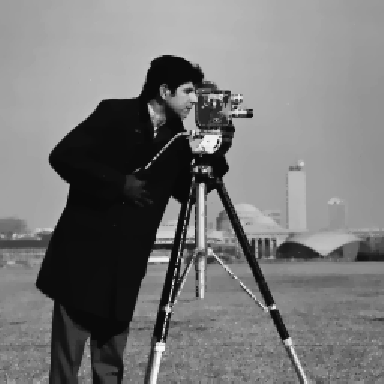}
    \caption{P-GRPDA}
    \label{fig:recon-part}
  \end{subfigure}
    \hfill
  \begin{subfigure}{0.23\textwidth}
    \includegraphics[width=\linewidth]{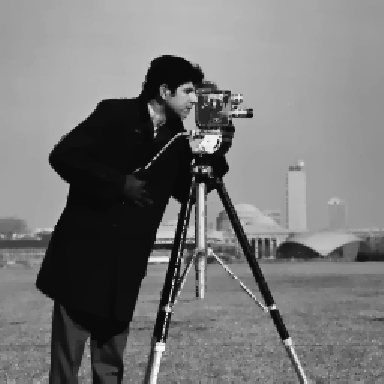}
    \caption{E-GRPDA}
    \label{fig:recon-fixed}
  \end{subfigure}
    \hfill
  \begin{subfigure}{0.23\textwidth}
    \includegraphics[width=\linewidth]{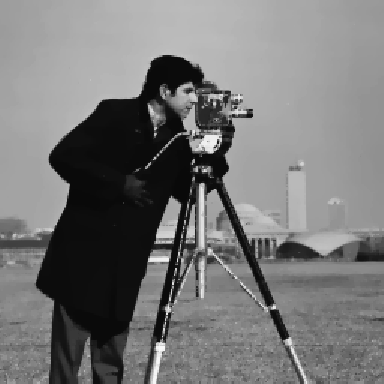}
    \caption{Condat--V\~{u} }
    \label{fig:recon-cv}
  \end{subfigure}

\caption{Reconstruction results for the TV--$L_2$ inpainting problem. (a) True image; (b) damaged (30\% pixels are removed); (c)--(f) reconstructed images obtained using aEGRPDA, P-GRPDA, E-GRPDA, and Condat--V\~{u}, respectively. The original image was downloaded from \url{https://webpages.tuni.fi/foi/GCF-BM3D/}.}\label{fig:Tv-L_2 recovery}
\end{figure*}
To track the recovery of the image, we compute the PSNR as
\begin{equation*}
\mathrm{PSNR}(x,x_{\mathrm{true}})
~=~
10\log_{10}\!\left(\frac{1}{\mathrm{MSE}(x,x_{\mathrm{true}})}\right),
\end{equation*}
where the \emph{mean squared error} ($\mathrm{MSE}$) is given by $\mathrm{MSE}(x,x_{\mathrm{true}})=\frac{1}{mn}\|x-x_{\mathrm{true}}\|_2^2$. First, we collect the effect of varying $\lambda$ on the solutions obtained by different algorithms, as presented in Figure~\ref{fig:lambda-columns}. Then, in Figure~\ref{fig:Tv-L_2 recovery}, we present the true image, the damaged image, and the corresponding reconstructions by various algorithms. Finally, Figure~\ref{fig:TV-L_2 convergence plots} reports measures for image recovery such as PSNR and convergence behavior, demonstrating that aEGRPDA achieves faster recovery with fewer iterations, followed by P-GRPDA and other algorithms.

\begin{figure}[htbp]
  \centering
  \begin{subfigure}{0.4\linewidth}
    \includegraphics[width=\linewidth]{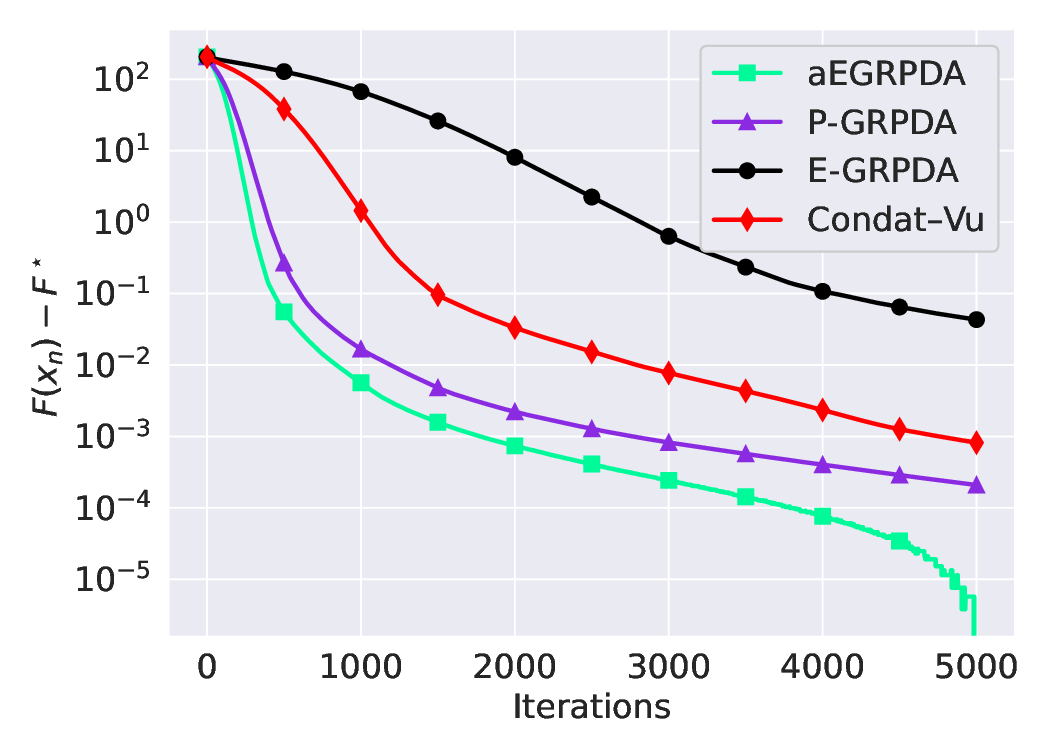}
    \caption{Objective gap vs. iterations}
    \label{fig:energy}
  \end{subfigure}\hfill
  \begin{subfigure}{0.4\linewidth}
    \includegraphics[width=\linewidth]{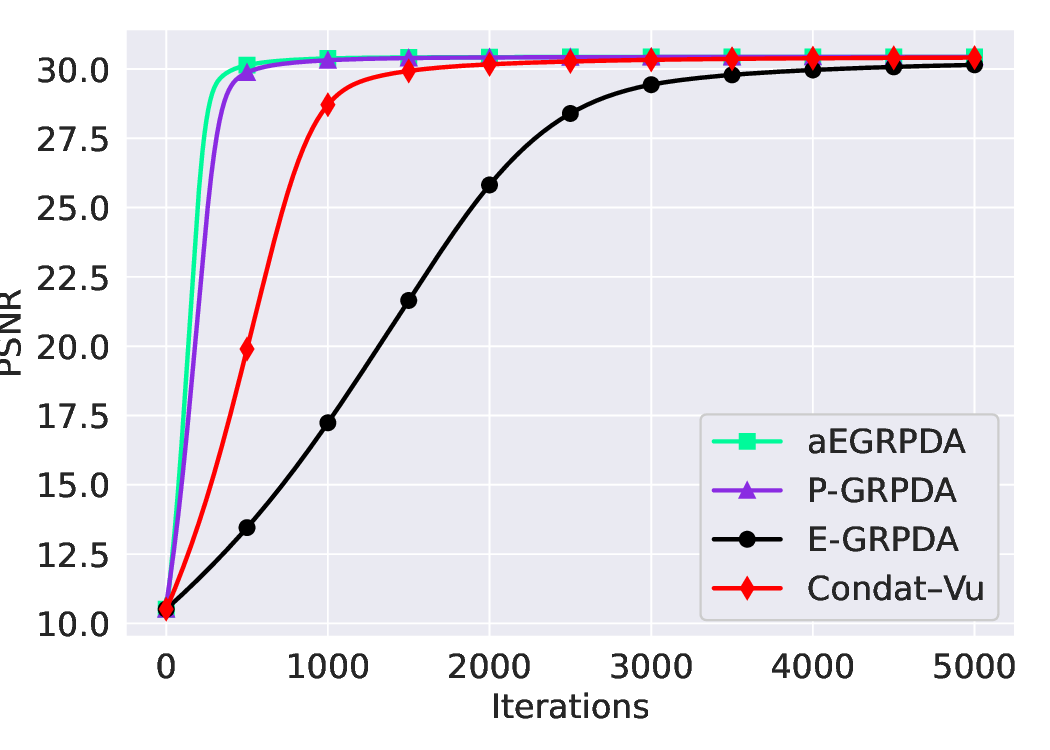}
    \caption{PSNR vs.\ iterations}
    \label{fig:psnr}
  \end{subfigure}
  \caption{Convergence plot and PSNR comparison for the TV-$L_2$ inpainting problem}
  \label{fig:TV-L_2 convergence plots}
\end{figure}


\section{Conclusion}
In this work, we introduced two new stepsize rules for E-GRPDA~\eqref{E-GRPDA} that eliminate the need for backtracking and remove the dependence on computing expensive parameters. Specifically, in the first stepsize rule, the primal stepsize $\tau_n$ converges, with its limit bounded below by a positive value. This crucial fact enabled us to establish both the global iterate convergence and the R-linear rate of convergence (when $h$ and $g^*$ are strongly convex) for Algorithm \ref{algorithm 1} without explicitly computing $\|K\|$ and $\bar L$. Additionally, we extended the convergence of the Algorithm \ref{algorithm 1} from the interval $(1, \phi]$ to $(1, 1 + \sqrt{3})$ as it was done for GRPDA \cite{chang2022grpdarevisited}. Next, we leveraged local information on the gradient of the smooth component function around the current iterates to establish the global convergence of Algorithm~\ref{algorithm 2}, with the mere assumption that the norm of the operator was known. We validated our approaches on a range of convex optimization problems, including Lasso, Fused Lasso, GraphNet, Logistic regression, and TV\text{--}$L_2$ image inpainting. The results consistently demonstrate the advantages of our proposed approaches over the existing ones.

We conclude this paper by discussing several directions that can be explored in future work.
\begin{itemize}
  \item \textbf{Acceleration.} It is desirable to investigate accelerated variants of Algorithm~\ref{algorithm 1} targeting an $\mathcal{O}(1/N^2)$ rate, in the spirit of GRPDA and E\text{--}GRPDA~\cite{chang2021goldengrpda,zhou2022new,chang2022goldenlinesearch}. It would also be interesting to explore Algorithm \ref{algorithm 2} in the case when $h$ is only locally smooth.
  \item \textbf{Estimating $\|K\|$ and $\bar{L}$.} In Algorithm~\ref{algorithm 2}, we estimate only the local Lipschitz constant while assuming knowledge of $\|K\|$. It would be desirable to estimate both $\|K\|$ and $\bar{L}$ using the same adaptive structure in Algorithm \ref{algorithm 2} without involving backtracking procedures.
  \item \textbf{Adaptation of Algorithm \ref{algorithm 1} and Algorithm \ref{algorithm 2}:} The stepsize strategy \eqref{eq:39c} and~\eqref{eq:alg2_c} can be adapted to explore other algorithm like SPDA-ie, proposed in \cite{chang2024splitting} to avoid explicit norm computation and to enhance practical robustness. Moreover, the same can be applied to explore the distributed setting proposed in \cite{yin2024golden}, when one of the local functions involved in the objective is locally smooth.
  \item \textbf{Bregman Extension:} One can change the metric of the Euclidean distance to Bregman divergence, and thus, Algorithm \ref{algorithm 1} and \ref{algorithm 2} can be studied with this change in effect. Also, stochastic extensions can be studied in this perspective for large-scale optimization problems.
  \item\textbf{Nonconvex Extensions:} It is unclear for now how to use nonconvexity in both P-GRPDA and aEGRPDA. However, in practice, many problems occur in nonconvex settings, so it would be a great direction to explore this setting, like using weak convexity, such as quasiconvex, instead of relying on the convexity of the component functions.
\end{itemize}

\section*{Acknowledgements}
We thank the two anonymous referees for their careful observations, which have helped significantly improve the clarity of this paper.
We thank Dr. Felipe Atenas for the valuable feedback on this paper. The research of Santanu Soe was supported by the Prime Minister's Research Fellowship program (Project number SB23242132MAPMRF005015), Ministry of Education, Government of India, and the Melbourne Research Scholarship. Matthew K. Tam was supported in part by Australian Research Council grant DP230101749. V. Vetrivel was supported by the Department of Science and Technology, SERB, India, under the project ‘FIST Program-2022’ (SR/FST/MSIAI/2022/107).
\section*{Data availability}
The datasets used in our numerical experiments are based on setups described in~\cite{malitsky2018first, chang2021goldengrpda,malitsky2020golden}. All implementation code supporting the results of this paper is publicly available at \url{https://github.com/soesantanu/Adaptive-GRPDA-Experiments}.
\section*{Conflict of interest}
The authors declare that they have no conflict of interest in this paper.


\bibliography{sn-bibliography}
\end{document}